\documentclass[a4paper,11pt,reqno]{article}

\usepackage[utf8]{inputenc}
\usepackage[T1]{fontenc}
\usepackage{lmodern}
\usepackage[english]{babel}
\usepackage{microtype}

\usepackage{amsmath,amssymb,amsfonts,amsthm,upgreek}
\usepackage{mathtools,accents}
\usepackage{mathrsfs}
\usepackage{aliascnt}
\usepackage{braket}
\usepackage{bm}
\usepackage{esint}

\mathtoolsset{showonlyrefs} 

\usepackage[a4paper,margin=3cm]{geometry}
\usepackage[citecolor=blue,colorlinks]{hyperref}

\usepackage{enumerate}
\usepackage{xcolor}

\makeatletter
\g@addto@macro\@floatboxreset\centering
\makeatother

\makeatletter
\def\newaliasedtheorem#1[#2]#3{
  \newaliascnt{#1@alt}{#2}
  \newtheorem{#1}[#1@alt]{#3}
  \expandafter\newcommand\csname #1@altname\endcsname{#3}
}
\makeatother

\numberwithin{equation}{section}

\newtheoremstyle{slanted}{\topsep}{\topsep}{\slshape}{}{\bfseries}{.}{.5em}{}

\theoremstyle{plain}
\newtheorem{theorem}{Theorem}[section]
\newaliasedtheorem{proposition}[theorem]{Proposition}
\newaliasedtheorem{lemma}[theorem]{Lemma}
\newaliasedtheorem{claim}[theorem]{Claim}
\newaliasedtheorem{corollary}[theorem]{Corollary}
\newaliasedtheorem{counterexample}[theorem]{Counterexample}

\theoremstyle{definition}
\newaliasedtheorem{definition}[theorem]{Definition}
\newaliasedtheorem{question}[theorem]{Question}
\newaliasedtheorem{example}[theorem]{Example}
\newaliasedtheorem{conjecture}[theorem]{Conjecture}

\theoremstyle{remark}
\newaliasedtheorem{remark}[theorem]{Remark}

\usepackage{comment}

\newcommand{\setC}{\mathbb{C}}
\newcommand{\setN}{\mathbb{N}}

\newcommand{\setR}{\mathbb{R}}

\newcommand{\bP}{\mathbb{P}}
\newcommand{\setB}{\mathbb{B}}
\newcommand{\cE}{\mathcal{E}}
\newcommand{\cV}{\mathcal{V}}
\newcommand{\ucV}{\underline{\mathcal{V}}}
\newcommand{\cW}{\underline{\mathcal{W}}}
\newcommand{\cD}{\mathcal{D}}

\newcommand{\eps}{\upepsilon}

\let\phi\varphi

\DeclareMathOperator{\Rad}{\mathrm{Rad}}

\newcommand{\di}{\mathop{}\!\mathrm{d}}

\newcommand{\bs}{{\rm bs}}
\newcommand{\loc}{{\rm loc}}

\DeclareMathOperator{\supp}{supp}

\newcommand{\scal}[2]{\ensuremath{\langle #1 , #2 \rangle}} 

\DeclareMathOperator{\Lip}{Lip}

\newcommand{\leb}{\mathscr{L}}

\newcommand{\measrestr}{%
  \,\raisebox{-.127ex}{\reflectbox{\rotatebox[origin=br]{-90}{$\lnot$}}}\,%
}

\newcommand{\dist}{\mathsf{d}}

\DeclareMathOperator{\RCD}{RCD}
\DeclareMathOperator{\Span}{Span}
\DeclareMathOperator{\vol}{\mathrm{vol}}

\DeclareMathOperator{\Ric}{Ric}

\newfont{\tmpf}{cmsy10 scaled 2500}

\DeclareMathOperator{\uX}{\underline{X}}
\DeclareMathOperator{\uB}{\underline{B}}
\DeclareMathOperator{\ub}{\underline{b}}
\DeclareMathOperator{\usigma}{\underline{\sigma}}
\DeclareMathOperator{\udist}{\underline{\dist}}
\DeclareMathOperator{\uo}{\underline{o}}
\DeclareMathOperator{\ux}{\underline{x}}
\DeclareMathOperator{\uz}{\underline{z}}
\DeclareMathOperator{\uc}{\underline{c}}
\DeclareMathOperator{\uL}{\underline{L}}
\DeclareMathOperator{\uphi}{\underline{\phi}}
\DeclareMathOperator{\uD}{\underline{D}}
\DeclareMathOperator{\uH}{\underline{H}}
\DeclareMathOperator{\uS}{\underline{S}}
\DeclareMathOperator{\umu}{\underline{\mu}}
\DeclareMathOperator{\unu}{\underline{\nu}}
\DeclareMathOperator{\ubeta}{\underline{\beta}}
\DeclareMathOperator{\ugamma}{\underline{\gamma}}
\newcommand{\uy}{\underline{y}}
\newcommand{\up}{\underline{p}}
\newcommand{\uh}{\underline{h}}

\DeclareMathOperator*{\essssup}{ess\,sup}
\DeclareMathOperator*{\esssinf}{ess\,inf}

\title{Main file}

\begin{document}

\title{A rigidity result for metric measure spaces with Euclidean heat kernel}
\author{Gilles Carron
\thanks{Université de Nantes, \url{Gilles.Carron@univ-nantes.fr}} \and
David Tewodrose
\thanks{CY Cergy Paris University,  \url{david.tewodrose@u-cergy.fr}}} \maketitle

\begin{abstract}
We prove that a metric measure space equipped with a Dirichlet form admitting an Euclidean heat kernel is necessarily isometric to the Euclidean space. This helps us providing an alternative proof of Colding's celebrated almost rigidity volume theorem via a quantitative version of our main result. We also discuss the case of a metric measure space equipped with a Dirichlet form admitting a spherical heat kernel.
\end{abstract}

\tableofcontents

\section{Introduction}

In $\setR^n$, the classical Dirichlet energy is the functional defined on $H^1$ by
$$E(u):= \int_{\setR^n} |\nabla u|^2$$ for any $u \in H^1$. As well-known, it is related to the Laplace operator $\Delta:=\sum_{k=1}^{n}\partial_{kk}$ by the integration by parts formula, namely
$$
E(u,v)=-\int_{\setR^n} (\Delta u) v
$$
for any $u,v \in H^1$ such that  $\nabla u\in H^1$, where $E(u,v):=\int_{\setR^n} \langle \nabla u, \nabla v\rangle$. Standard tools from spectral theory show that $\Delta$ generates a semi-group of operators $(e^{t\Delta})_{t>0}$ sending any $u_0 \in L^2$ to the family $(u_t)_{t>0}\subset H^1$ satisfying the heat equation $\partial_t u_t = \Delta u_t$ with $u_0$ as an initial condition. The semi-group $(e^{t\Delta})_{t>0}$ admits a smooth kernel $p$, so that for any $f \in L^2$, $x \in \setR^n$ and $t>0$,
$$
e^{t\Delta} f(x) = \int_{\setR^n} p(x,y,t)f(y) \di y.
$$
The explicit expression of this heat kernel is well-known: for any $x,y \in \setR^n$ and $t>0$,
$$
p(x,y,t) = \frac{1}{(4 \pi t)^{n/2}} e^{-\frac{|x-y|^2}{4t}}.
$$

In the more general context of a measured space $(X,\mu)$, the Dirichlet energy possesses abstract analogues called Dirichlet forms. Associated with any such a form $\cE$ is a self-adjoint operator $L$ whose properties are similar to the Laplace operator; in particular, the spectral theorem applies to it and provides a semi-group $(P_t)_{t>0}$ delivering the solution of the equation $\partial_t u_t = L u_t$ starting from any square integrable initial condition. Under suitable assumptions, this semi-group admits a kernel. When the space $X$ is equipped with a metric $\dist$ generating the $\sigma$-algebra on which $\mu$ is defined, this kernel is often compared with the Gaussian term
$$
\frac{1}{(4 \pi t)^{n/2}} e^{-\frac{\dist^2(x,y)}{4t}}
$$
through upper and lower estimates: see \cite{Sturm2}, for instance. From this perspective, a natural question arises: what happens when the kernel of $\cE$ coincides with this Gaussian term? In this article, we answer this question by showing that the unique metric measure space admitting such a kernel is the Euclidean space. The precise statement of our main result is the following:

\begin{theorem}\label{th:main}
Let $(X,\dist)$ be a complete metric space equipped with a non-negative regular Borel measure $\mu$.  Assume that there exists a symmetric Dirichlet form $\cE$ on $(X,\mu)$ admitting a heat kernel $p$ such that for some $\alpha>0$,
\begin{equation}\label{eq:heatkernel}
p(x,y,t) = \frac{1}{(4 \pi t)^{\alpha/2}} e^{-\frac{\dist^2(x,y)}{4t}}
\end{equation}
holds for any $x,y \in X$ and any $t>0$. Then $\alpha$ is an integer, $(X,\dist)$ is isometric to $(\setR^\alpha,\dist_e)$ where $\dist_e$ stands for the classical Euclidean distance, and $\mu$ is the $\alpha$-dimensional Hausdorff measure. 
\end{theorem}

Then we show that this rigidity result can be turned quantitative via a suitable contradiction argument. Denoting by $\dist_{GH}$ the Gromov-Hausdorff distance and by $\setB^n_r$ any Euclidean ball in $\setR^n$ with radius $r>0$, we obtain the following:

\begin{theorem}\label{th:almostrigidity} For any $\upepsilon>0$, there exists $\updelta>0$ depending only on $\upepsilon$ and $n$ such that if $(X,\dist,\mu)$ is a complete metric measure space endowed with a symmetric Dirichlet form $\cE$ 
 admitting a heat kernel $p$ satisfying 
\begin{equation}\label{eq:almostheatkernel}
(1-\updelta) \frac{1}{(4 \pi t)^{n/2}} e^{-\frac{\dist^2(x,y)}{4(1-\updelta)t}}\le p(x,y,t) \le (1+\updelta) \frac{1}{(4 \pi t)^{n/2}} e^{-\frac{\dist^2(x,y)}{4(1+\updelta)t}}
\end{equation}
for any $x,y \in X$ and $t \in (0,T]$, for some given $T>0$, then for any $x \in X$ and $r \in (0,\sqrt{T})$,
\begin{equation}\label{eq:Reifenberg}
\dist_{\mathrm GH}\left( B_r(x), \setB^n_r\right)< \upepsilon r.
\end{equation}
\end{theorem}

The intrinsic Reifenberg theorem of Cheeger and Colding \cite[Theorem  A.1.1.]{CheegerColding} provides the following immediate topological consequence, where $\Psi(\cdot|n)$ is a function depending only on $n$ with $\Psi(r|n) \to 0$ when $r \to 0^+$.

\begin{corollary}\label{cor:topo} There exists $\updelta_n>0$ depending only on $n$ such that if $(X,\dist,\mu)$ is a complete metric measure space endowed with a symmetric Dirichlet form $\cE$ admitting a heat kernel $p$ such that for some numbers $\delta \in (0,\delta_n)$ and $T>0$,
$$
(1-\updelta)\frac{1}{(4 \pi t)^{n/2}} e^{-\frac{\dist^2(x,y)}{4(1-\updelta) t}}\le p(x,y,t) \le (1+\updelta)\frac{1}{(4 \pi t)^{n/2}} e^{-\frac{\dist^2(x,y)}{4(1+\updelta) t}}
$$holds for all $x, y \in X$ and $t \in (0,T)$, then  for any $x \in X$, there exists a topological embedding of $\mathbb{B}_{\sqrt{T}}^{n}$ into $B_{\sqrt{T}}(x)$ whose image contains $B_{(1-\Psi(\updelta|n))\sqrt{T}}(x)$.\end{corollary}

We point out the two previous results are also true in case $T=+\infty$. Moreover, Theorem \ref{th:almostrigidity} can be used to give an alternative proof of a celebrated result established by T-H. Colding \cite[Theorem 0.8]{Colding}, namely the almost rigidity of the volume for Riemannian manifolds with non-negative Ricci curvature. Let us recall this statement:

\begin{theorem}[Colding]\label{th:Colding}For any $\upepsilon>0$, there exists $\updelta>0$ depending only on $\upepsilon$ and $n$ such that if 
$(M^n,g)$ is a complete Riemannian manifold with non-negative Ricci curvature such that for any $x \in M$ and $r>0$,
\begin{equation}\label{eq:vol}\vol B_r(x) \ge \left(1-\updelta\right) \omega_n\, r^n,\end{equation}
then  for
any $x\in M$ and $r>0$, 
$$\dist_{\mathrm GH}\left( B_r(x), \setB^n_r\right)\le \upepsilon r.$$
\end{theorem}

This theorem is a direct consequence of our almost rigidity theorem coupled with an intermediary result, Theorem \ref{th:estimate}, which states, roughly speaking, that a complete Riemannian manifold satisfying the volume estimate \eqref{eq:vol} has necessarily an almost Euclidean heat kernel. Our proof of this result is based on previous works by J.~Cheeger and S.-T.~Yau \cite{CheegerYau}, P.~Li and S.-T.~Yau \cite{LiYau} and especially P.~Li, L.-F.~Tam and J.~Wang \cite{LiTamWang}.

Finally, in the last section of this paper, we investigate the case of a metric measure space equipped with a spherical heat kernel. To be precise, the sphere $\mathbb{S}^n$ has a heat kernel which can be written as
$$
K_t^{(n)}(\dist_{\mathbb{S}^n}(x,y))
$$
where $K_t^{(n)}$ is an explicit function and $\dist_{\mathbb{S}^n}$ is the classical round Riemannian distance. We show that if a metric measure space $(X,\dist,\mu)$ is equipped with a Dirichlet form admitting a heat kernel $p$ such that $$p(x,y,t)=K_t^{(n)}(\dist(x,y))$$ for all $x, y \in X$ and $t>0$, then $(X,\dist)$ is isometric to $(\mathbb{S}^n,\dist_{\mathbb{S}^n})$.

Let us spend some words to describe our proof of Theorem \ref{th:main}. A key point is the celebrated result of T.H.~Colding and W.P.~Minicozzi II asserting that on any complete Riemannian manifold satisfying the doubling and Poincaré properties, the space of harmonic maps with linear growth is finite-dimensional \cite{ColdingMinicozzi}. As already observed in non-smooth contexts \cite{Hua, HuaKellXia}, the proof of this result can be carried out on any complete metric measure spaces satisfying the doubling and Poincaré properties. It turns out that admitting a Dirichlet form with an Euclidean heat kernel forces the metric measure space to satisfy these two properties, see Proposition \ref{prop:important}.

Then we consider the functions
$$B(x,\cdot):=\frac{1}{2}(\dist^2(o,x)+\dist^2(o,\cdot)-\dist^2(x,\cdot)) \qquad x \in X$$
which are easily shown to have linear growth. When $(X,\dist,\mu)$ is equipped with a Dirichlet form $\cE$ satisfying the assumptions of Theorem \ref{th:main}, these functions are locally $L$-harmonic: this follows from establishing
$$
L\textbf{1}=0 \qquad \text{and} \qquad L\dist^2(x,\cdot) = 2 \alpha.
$$
Therefore, the vector space $\cV$ generated by the functions $B(x,\cdot)$ has a finite dimension $n$. Choosing a suitable basis $(h_1,\ldots,h_n)$ of this space, we embed $X$ into $\setR^n$ by setting
$$
H(x)=(h_1(x),\ldots,h_n(x))
$$
for any $x \in X$. More precisely, there exists $x_1,\ldots,x_n \in X$ such that $(\delta_{x_1},\ldots,\delta_{x_n})$ is a basis of $\cV^{*}$, and $(h_1,\ldots,h_n)$ is chosen as the dual of this basis. Setting $Q(\xi) := \sum_{i,j} B(x_i,x_j)\xi_i \xi_j$ for any $\xi=(\xi_1,\cdots,\xi_n) \in \setR^n$, we easily get
\begin{equation}\label{eq:isometry}
Q(H(x)-H(y))=\dist^2(x,y)
\end{equation}
for any $x, y \in X$, thus $H$ is an embedding.

To conclude, we establish $\alpha = n$ and show that $Q$ is non-degenerate, so that $\dist_Q(\xi,\xi')=\sqrt{Q(\xi-\xi')}$ defines a distance on $\setR^n$ that is isometric to the Euclidean distance: then \eqref{eq:isometry} shows that $H$ is an isometric embedding onto its image a final argument proves to be $\setR^n$. We prove these two concluding assertions by the study of asymptotic cones at infinity of $(X,\dist,\mu)$.

It is worth mentioning that in the case of $(X,\dist,\mu,\cE)$ equipped with a spherical heat kernel, we embed $X$ into $E_1:=\mathrm{Ker}(-L-\lambda_1 I)$, where $\lambda_1$ is the first non-zero eigenvalue of $-L$, and show that $H(X)$ is isometric to $\Sigma:=\{Q=1\}$ for some suitable quadratic form $Q$.

The paper is organized as follows. Our proof of Theorem \ref{th:main} relies on several notions and results from different areas that we collect in the preliminary Section 2. Then in Section 3 we establish simple rigidity results for metric measure spaces with an Euclidean heat kernel. We use these results in Section 4 which is dedicated to the proof of Theorem \ref{th:main}. Section 5 is devoted to the almost rigidity result, namely Theorem \ref{th:almostrigidity}, and Section 6 explains our new proof of Colding's volume almost rigidity theorem. Finally Section 7 contains our study of the case of metric measure spaces equipped with a spherical heat kernel.

\smallskip\noindent
\textbf{Acknowledgement.}
The first author thanks the Centre Henri Lebesgue ANR-11-LABX-0020-01 for creating an attractive mathematical environment; he was also  partially supported by the ANR grants: {\bf ANR-17-CE40-0034}: {\em CCEM} and {\bf ANR-18-CE40-0012}: {\em RAGE}.
The second author thanks S.~Honda for interesting remarks and questions at a late stage of this work and for the good working conditions in Tohoku University that grandly helped completing this article.

\section{Preliminaries}

\quad \, Throughout the article, we shall call metric measure space any triple $(X,\dist,\mu)$ where $(X,\dist)$ is a $\sigma$-compact metric space and $\mu$ is a non-negative $\sigma$-finite Radon measure on $(X,\dist)$ such that $\supp \mu = X$. Here $\supp \mu$ denotes the support of $\mu$. We shall keep fixed a number $\alpha>0$ and denote by $\omega_\alpha$ the quantity
\begin{equation}\label{eq:omega_n}
\omega_\alpha = \frac{\pi^{\alpha/2}}{\Gamma(\alpha/2+1)}
\end{equation}
where $\Gamma$ denotes the usual Gamma function $\{\mathrm{Re}>0\} \ni z \mapsto \int_0^{+\infty} t^{z-1}e^{-t} \di t$. Note that when $\alpha$ is an integer $n$, then $\omega_n$ is the volume of the unit Euclidean ball in $\setR^n$.

We shall use classical notations for the functional spaces defined on $(X,\dist, \mu)$, like $C(X)$ (resp.~$C_c(X)$) for the space of continuous (resp.~compactly supported continuous) functions, $\Lip(X)$ (resp.~$\Lip_c(X)$)) for the space of Lipschitz (resp.~compactly supported Lipschitz) functions, $L^p(X,\mu)$, where $p \in [1,+\infty)$, for the space of (equivalent classes of) $\mu$-measurable functions whose $p$-th power is $\mu$-integrable, $L^\infty(X,\mu)$ for the space of $\mu$-essentially bounded functions, and so on.
We shall write $\supp f$ for the support of a function $f$ and $1_A$ for the characteristic function of a set $A \subset X$.

A generic open ball in $(X,\dist)$ will be denoted by $B$, and we will write $\lambda B$ for the ball with same center as $B$ but radius multiplied by $\lambda>0$.

We will extensively make use of the following definition.

\begin{definition}
We say that a metric measure space $(X,\dist,\mu)$ has an $\alpha$-dimensional volume whenever $\mu(B)=\omega_\alpha r^\alpha$ for any metric ball $B\subset X$ with radius $r>0$.\\
\end{definition}

\textbf{Dirichlet forms.}

Let us recall some basic facts about Dirichlet forms, refering to e.g. \cite{FukushimaOshidaTakeda, Sturm1, KoskelaZhou} for more details. Let $(X,\tau)$ be a topological space equipped with a $\sigma$-finite Borel measure $\mu$. A Dirichlet form $\cE$ on $(X,\mu)$ is a non-negative definite bilinear map $\cE : \cD(\cE) \times \cD(\cE) \to \setR$, with $\cD(\cE)$ being a dense subset of $L^2(X,\mu)$, satisfying closedness, meaning that the space $\cD(\cE)$ is a Hilbert space once equipped with the scalar product
\begin{equation}\label{eq:scalE}
\langle f, g \rangle_{\cE} := \int_{X} f g \di \mu + \cE(f,g)\qquad \forall f, g \in \cD(\cE),
\end{equation}
and the Markov property: for any $f \in \cD(\cE)$, the function $f_{0}^{1} = \min ( \max(f,0),1)$ belongs to $\cD(\cE)$ and $\cE(f_{0}^{1},f_{0}^{1}) \le \cE(f,f)$. We denote by $|\cdot|_\cE$ the norm associated with $\langle \cdot, \cdot \rangle_\cE$.

We focus only on symmetric Dirichlet forms, i.e.~those $\cE$ for which $\cE(f,g)=\cE(g,f)$ holds for all $f,g \in \cD(\cE)$. Therefore, in the rest of the article, by Dirichlet form we will always tacitly mean \textit{symmetric} Dirichlet form.

Finally, let us recall that any Dirichlet form is associated with a non-negative definite self-adjoint operator $L$ with dense domain $\cD(L) \subset L^2(X,\mu)$ characterized by the following:
$$
\cD(L):=\left\{f \in \cD(\cE) \, : \, \exists h=:Lf \in L^2(X,\mu)\, \, \text{s.t.}\, \, \cE(f,g)= -\int_X h g \di \mu \, \, \, \forall g \in \cD(\cE)\right\}.
$$

\hfill
We now additionally assume that $(X,\tau)$ is locally compact and separable and that $\mu$ is a Radon measure such that $\supp \mu = X$. A Dirichlet form $\cE$ on $(X,\mu)$ is called \textit{strongly local} if $\cE(f,g)=0$ for any $f, g \in \cD(\cE)$ such that $f$ is constant on a neighborhood of $\supp g$, and \textit{regular} if $C_c(X) \cap \cD(\cE)$ contains a subset (called a \textit{core}) which is both dense in $C_c(X)$ for $\|\cdot\|_{\infty}$ and in $\cD(\cE)$ for $|\cdot|_\cE$. A celebrated result by A.~Beurling and J.~Deny \cite{BeurlingDeny} implies that any strongly local regular Dirichlet form $\cE$ on $(X,\mu)$ admits a non-negative definite symmetric bilinear map $\Gamma : \cD(\cE) \times \cD(\cE) \to \mathrm{Rad}$, where $\mathrm{Rad}$ denotes the set of signed Radon measures on $(X,\tau)$, such that
$$
\cE(f,g) = \int_X \di \Gamma(f,g) \qquad \forall f,g \in \cD(\cE),
$$
where $\int_X \di \Gamma(f,g)$ denotes the total mass of the measure $\Gamma(f,g)$. From now until the end of this paragraph, we assume that $\cE$ is strongly local and regular. 

Let us mention that the map $\Gamma$ is concretely given as follows: for any $f \in \cD(\cE) \cap L^\infty(X,\mu)$, the measure $\Gamma(f):=\Gamma(f,f)$ is defined by its action on test functions:
\begin{equation}\label{eq:2000}
\int_X \phi \di \Gamma(f) := \cE(f,f\phi) - \frac{1}{2}\cE(f^2, \phi) \qquad \forall \phi \in \cD(\cE) \cap C_c(X).
\end{equation}
Regularity of $\cE$ allows to extend \eqref{eq:2000} to the set of functions $\phi \in C_c(X)$, providing a well-posed definition of $\Gamma(f)$ by duality between $C_c(X)$ and $\Rad$. In case $f \in \cD(\cE)$ is not essentially bounded, $\Gamma(f)$ is obtained as the limit of the increasing sequence of measures $(\Gamma(f_{-n}^{n}))_{n\in\setN}$ where $f_{-n}^n:=\min(\max(f,-n),n)$ for any $n \in \setN$. The general expression of $\Gamma(f,g)$ for any $f,g \in \cD(\cE)$ is then obtained by polarization:
$$
\Gamma(f,g) := \frac{1}{4}(\Gamma(f+g,f+g) - \Gamma(f-g,f-g)).
$$

Strong locality of $\cE$ implies locality of $\Gamma$, that is
$$
\int_A \di \Gamma(u,w) = \int_A \di \Gamma(v,w)
$$
for any open set $A\subset X$ and any functions $u,v,w \in \cD(\cE)$ such that $u=v$ on $A$. This property allows to extend $\Gamma$ to the set $\cD_{loc}(\cE)$ made of those $\mu$-measurable functions $f$ for which for any compact set $K\subset X$ there exists $g \in \cD(\cE)$ such that $f=g$ $\mu$-a.e.~on $K$. Then $\Gamma$ satisfies the Leibniz rule:
\begin{equation}\label{eq:Leibniz}
\Gamma(fg,h)=f \Gamma(g,h) + g \Gamma(f,h) \qquad \forall u,v \in \cD_{loc}(\cE) \cap L^{\infty}_{loc}(X,\mu), \, \, \forall h \in \cD_{loc}(\cE),
\end{equation}
and the chain rule:
\begin{align}\label{eq:chain}
\Gamma(\eta \circ f,g)=(\eta'\circ f) \Gamma(f,g)  \qquad \qquad  & \forall \eta \in C^1_{b,bd}(\setR), \, \, \forall f \in \cD_{loc}(\cE),\nonumber \\
\text{and} \quad & \forall \eta \in C^1(\setR), \, \, \forall f \in \cD_{loc}(\cE)\cap L^{\infty}(X,\mu),
\end{align}
where $C^1_{b,bd}(\setR)$ stands for the set of bounded $C^1$ functions on $\setR$ with bounded derivative.

For our purposes, we also need to define $\cD_{loc}(\Omega,\cE)$ as the set of functions $f\in L^2_{\loc}(\Omega)$ for which for any compact set $K\subset \Omega$ there exists $g \in \cD(\cE)$ such that $f=g$ $\mu$-a.e.~on $K$; here $\Omega$ is an open subset of $X$.

The so-called \textit{intrinsic} extended pseudo-distance $\dist_\cE$ associated with $\cE$ is defined by:
\begin{equation}\label{eq:defdist}
\dist_\cE(x,y):=\sup \{|f(x)-f(y)| \, : \, f \in C(X) \cap \cD_{loc}(\cE) \, \, \, \text{s.t.} \, \, \Gamma(f) \le \mu\}\quad \forall x,y \in X.
\end{equation}
Here $\Gamma(f) \le \mu$ means that $\Gamma(f)$ is absolutely continuous with respect to $\mu$ with density lower than $1$ $\mu$-a.e.~on $X$, and ``extended'' refers to the fact that $\dist_\cE(x,y)$ may be infinite. When the topology $\tau$ is generated by a distance $\dist$ on $X$, we call asumption (A) the following statement:
\begin{equation}\label{eq:A}\tag{A}
\text{$\dist_\cE$ is a distance inducing the same topology as $\dist$}.
\end{equation}

A final consequence of strong locality and regularity is that the operator $L$ canonically associated to $\cE$ satisfies the classical chain rule:
\begin{equation}\label{eq:chainrule}
L(\phi \circ f) =  (\phi' \circ f) Lf + (\phi''\circ f) \Gamma(f) \qquad \forall f \in \mathbb{G}, \,\, \, \forall \phi \in C^\infty([0,+\infty),\setR),
\end{equation}
where $\mathbb{G}$ is the set of functions $f \in \cD(L)$ such that $\Gamma(f)$ is absolutely continuous with respect to $\mu$ with density also denoted by $\Gamma(f)$. In particular:
\begin{equation}\label{eq:chainrulesquare}
Lf^2 = 2fLf + 2 \Gamma(f) \qquad \forall f \in \mathbb{G}.
\end{equation}

\vspace{4mm}

\textbf{Heat kernel associated to a Dirichlet form.}

Let $(X,\tau)$ be a topological space equipped with a $\sigma$-finite Borel measure $\mu$. The spectral theorem (see e.g.~\cite[Th.~VIII.5]{ReedSimon}) implies that the operator $L$ associated to any Dirichlet form $\cE$ on $(X,\tau)$ defines an analytic sub-Markovian semi-group $(P_t)_{t>0}$ acting on $L^2(X,\mu)$ where for any $f \in L^2(X,\mu)$, the map $t \mapsto P_tf$ is characterized as the unique $C^1$ map $(0,+\infty)\to L^2(X,\mu)$, with values in $\cD(L)$, such that
$$
\begin{cases}
\frac{\di}{\di t} P_tf = L(P_t f) \qquad \forall t>0,\\
\lim\limits_{t \to 0} \|P_t f - f\|_{L^2(X,\mu)}=0.
\end{cases}
$$
One can then recover $\cD(L)$ and $L$ from $(P_t)_{t>0}$ in the following manner:
$$\cD(L)=\left\{f \in L^2(X,\mu) \, : \, \left(\frac{P_t f - f}{t}\right)_{t>0} \, \text{converges in the $L^2$-norm when $t\downarrow 0$}\right\}
,$$
\begin{equation}\label{eq:carL}
Lf = \lim\limits_{t \downarrow 0} \frac{P_t f-f}{t} \qquad \forall f \in \cD(L).
\end{equation}

We say that $\cE$ admits a heat kernel if there exists a family of $(\mu \otimes \mu)$-measurable functions $(p(\cdot, \cdot, t))_{t>0}$ on $X\times X$ such that for all $t>0$ and $f\in L^2(X,\mu)$, one has
$$P_t f (x) = \int_X p(x,y,t) f(y) \di \mu(y)  \qquad \text{for $\mu$-a.e.~$x \in X$};$$
the function $p=p(\cdot,\cdot,\cdot)$ is then called the heat kernel of $\cE$.
In this case, the semi-group property (namely $P_{s+t}=P_s \circ P_t$ for any $s,t>0$) implies that $p$ satisfies the so-called Chapman-Kolmogorov property:
\begin{equation}\label{eq:ChapmanKolmogorov}
\int_X p(x,z,t)p(z,y,s) \di \mu(z) = p(x,y,t+s), \qquad \forall x,y \in X, \, \, \forall s,t>0.
\end{equation}
Moreover, for any $t>0$, $p(\cdot,\cdot,t)$ is symmetric and uniquely determined up to a $(\mu\otimes \mu)$-negligible subset of $X\times X$.

When $\cE$ admits a heat kernel, the space $(X,\tau,\mu,\cE)$ is called stochastically complete whenever
$$
\int_X p(x,y,t) \di \mu(y) = 1 \qquad \forall x \in X, \forall t >0.
$$
Under stochastic completeness, one can show that the domain of $\cE$ coincides with
$$
\left\{f \in L^2(X,\mu) \, : \, t \mapsto \frac{1}{2t} \iint_{X\times X} (f(x)-f(y))^2 p(x,y,t) \di \mu(x) \di \mu(y) \, \, \text{is bounded} \right\}
$$
and that
\begin{equation}\label{eq:Grigor'yan}
\cE(f,g) = \lim\limits_{t \downarrow 0} \frac{1}{2t} \iint_{X\times X} (f(x)-f(y))(g(x)-g(y)) p(x,y,t) \di \mu(x) \di \mu(y)
\end{equation}
for any $f, g \in \cD(\cE)$ and
\begin{equation}\label{eq:Grigor'yan2}
\int_X \phi \di \Gamma(f) = \lim\limits_{t \downarrow 0} \frac{1}{2t} \iint_{X\times X} \phi(x)(f(x)-f(y))^2 p(x,y,t) \di \mu(x) \di \mu(y)
\end{equation}
for any $f \in \cD_{loc}(\cE)$ and $\phi \in C_c(X)$: see \cite[2.2]{Grigor'yan}, for instance.

As well-known, the classical Dirichet energy on $\setR^n$ admits the Gaussian heat kernel
$$
p(x,y,t) = \frac{1}{(4 \pi t)^{n/2}} e^{-\frac{\dist_e^2(x,y)}{4t}} \qquad \forall x,y \in X, \, \forall t >0,
$$
where $\dist_e$ is the usual Euclidean distance. This motivates the next definition.

\begin{definition}\label{def:Euclideanheatkernel}
Let $(X,\dist,\mu)$ be a metric measure space and $\cE$ a Dirichlet form on $(X,\mu)$. We say that $(X,\dist,\mu,\cE)$ has an $\alpha$-dimensional Euclidean heat kernel if $\cE$ admits a heat kernel $p$ such that:
$$
p(x,y,t) = \frac{1}{(4 \pi t)^{\alpha/2}} e^{-\frac{\dist^2(x,y)}{4t}}\qquad \forall x,y \in X, \, \forall t>0.$$
\end{definition}

\hfill

\textbf{Harnack inequalities.}

Let $(X,\dist,\mu)$ be a metric measure space equipped with a Dirichlet form $\cE$ with associated operator $L$. Let $\leb^1$ be the Lebesgue measure on $\setR$. In order to properly state what a Harnack inequality means for $(X,\dist,\mu,\cE)$, let us introduce some notions. We refer e.g.~to \cite{Sturm2} and the references therein for more details. Note first that any element $f \in L^2(X,\mu)$ uniquely defines a continuous linear form on $\cD(\cE)$, namely $g \mapsto \int_X fg \di \mu$. Thus $L^2(X,\mu)$ embeds into $\cD(\cE)^*$ whose norm we denote $|\cdot|_{\cE,*}$.

For any open interval $I \subset \setR$, we consider the following functional spaces:
\begin{itemize}
\item $L^2(I,\cD(\cE))$ is the space of $\leb^1$-measurable functions $u:I\to \cD(\cE)$, $u_t:=u(t)$, equipped with the Hilbert norm $\|u\|_{L^2(I,\cD(\cE))}:=(\int_I |u_t|_{\cE}^2 \di t)^{1/2}$;

\item $H^1(I,\cD(\cE)^*)$ is the space of $\leb^1$-measurable functions $u:I\to \cD(\cE)^*$ admitting a distributional derivative $\partial_t u \in L^2(I,\cD(\cE)^*)$ on $I$ equipped with the Hilbert norm $\|u\|_{H^1(I,\cD(\cE)^*)}:=(\int_I |u_t|_{\cE,*}^2 \di t + \int_I |(\partial_t u)_t  |_{\cE,*}^2 \di t  )^{1/2}$, where $(\partial_t u)_t  := \partial_t u(t)$;

\item $\cD_{par,I}(\cE):=L^2(I,\cD(\cE))\cap H^1(I,\cD(\cE)^*)$ equipped with the Hilbert norm $\|u\|_{par,I}:=(\int_I |u_t|_{\cE}^2 \di t + \int_I |(\partial_t u)_t  |_{\cE,*}^2 \di t  )^{1/2}$.
\end{itemize}
We can define a Dirichlet form $\cE_I$ on $\cD_{par,I}(\cE)$ by setting
$$
\cE_I(u,v):=\int_I\cE(u_t,v_t) \di t - \int_I (\partial_t u)_t  \cdot v_t \di t \qquad \forall u, v \in \cD_{par,I}(\cE).
$$
Let $\Omega \subset X$ be an open set. Denote by $Q$ the parabolic cylinder $I \times \Omega$. Let $\cD_{Q}(\cE)$ be the set of $(\leb^1 \otimes \mu)$-measurable functions defined on $Q$ such that for every relatively compact open set $\Omega' \Subset \Omega$ and every open interval $I' \Subset I$ there exists a function $u' \in \cD_{par,I}(\cE)$ such that $u=u'$ on $I' \times \Omega'$. We also define $\cD_{Q,c}(\cE)$ as the set of functions $u \in \cD_{Q}(\cE)$ such that for any $t \in I$, the function $u_t$ has compact support in $\Omega$.

\begin{definition}
We call local solution on $Q$ of the equation $(\partial_t+L)u = 0$ any function $u \in \cD_Q(\cE)$ such that $\cE_I(u,\phi)=0$ holds for any $\phi \in \cD_{Q,c}(\cE)$.
\end{definition}

When $\cE$ admits a heat kernel $p$, one can show that for any $x \in X$ and $t>0$ the function $p(x,\cdot,t)$ is a local solution of the equation $(\partial_t+L)u=0$.

The next important proposition is a combination of several famous results \cite{Grigor'yan92,Saloff-Coste,Sturm3}.

\begin{proposition}\label{prop:important}
Let $(X,\dist,\mu)$ be a metric measure space equipped with a strongly local and regular Dirichlet form $\cE$ satisfying assumption \eqref{eq:A}. Let $L$ be the operator canonically associated to $\cE$. Then the following statements are equivalent:
\begin{enumerate}
\item the combination of

a) the doubling property: there exists a constant $C_D>0$ such that  for any ball $B \subset X$,
\begin{equation}\label{eq:doubling}
\mu(2B)\le C_D \mu(B),
\end{equation}

b) the local Poincaré inequality: there exists a constant $C_P>0$ such that for any $f \in \cD(\cE)$ and any ball $B \subset X$ with radius $r>0$, setting $f_B := \mu(B)^{-1}\int_B f \di \mu$,
\begin{equation}\label{eq:Poincare}
\int_B |f-f_B|^2 \di \mu \le C_P r^2 \cE(f),
\end{equation}

\item the existence of a heat kernel $p$ for $\cE$ satisfying double-sided Gaussian estimates: there exists $C_G>0$ such that for any $x,y \in X$ and any $t>0$,
\begin{equation}\label{eq:LiYau}
\frac{C_G^{-1}}{\mu(B_{\sqrt{t}}(x))} e^{-\frac{\dist^2(x,y)}{4t}} \le p(x,y,t) \le \frac{C_G}{\mu(B_{\sqrt{t}}(x))} e^{-\frac{\dist^2(x,y)}{4t}},
\end{equation}

\item the parabolic Harnack inequality: there exists a constant $C_H>0$ such that {\color{blue} for any $s\in \setR$}, any ball $B$ with radius $r>0$ and any non-negative local solution $u$ on $(s-r^2,s)\times B$ of the parabolic equation $(\partial_t + L)u=0$, we have
\begin{equation}\label{eq:parabolicHarnack}
\essssup_{Q_-} (u) \le C_H \esssinf_{Q_+} (u)
\end{equation}
where $Q_-:=(s-(3/4)r^2,s-(1/2)r^2) \times (1/2)B$ and $Q_+:=(s-(1/4)r^2,s)\times (1/2)B$.
\end{enumerate}
\end{proposition}

Note that the parabolic Harnack inequality \eqref{eq:parabolicHarnack} implies the elliptic one introduced below in Lemma \ref{lem:ellipticHarnack}.\\

\textbf{Locally $L$-harmonic functions.}

Let $(X,\dist,\mu)$ be a metric measure space equipped with a  strongly local and regular Dirichlet form $\cE$ with associated operator $L$. We set $$
\cD_c(\cE):=\{ \phi \in \cD(\cE) \, \text{with compact support}\}.$$

\begin{definition}\label{def:localsol}  Let $\Omega \subset X$ be an open set.

1. We call local solution on $\Omega$ of the Laplace equation $Lu=0$ any function $u \in \cD_{loc}(\Omega,\cE)$ such that $\cE(u,\phi)=0$ holds true for any $\phi \in  \cD_c(\cE)$ with $\supp \phi \subset \Omega$.

2. We call locally $L$-harmonic function any function $u \in \cD(\cE)$ such that $\cE(u,\phi)=0$ holds true for any $\phi \in  \cD_c(\cE)$.

3. For any $f \in L^1_{loc}(X,\mu)$, we call local solution on $\Omega$ of the Poisson equation $Lu=f$ any function $u \in \cD_{loc}(\Omega, \cE)$ such that $\cE(u,\phi)=-\int_X f \phi \di \mu$ holds true for any $\phi \in \cD_c(\cE)$ with $\supp \phi \subset \Omega$.
\end{definition}

We shall often simply write ``$Lu=f$ on $\Omega$'' to mean that $u \in \cD_{loc}(\Omega, \cE)$ is a local solution on $\Omega$ of the equation $Lu=f$, and ``$Lv=0$'' to express that $v \in \cD_{loc}(\cE)$ is locally $L$-harmonic. Lastly, we point out that strong locality directly implies that constant functions are locally $L$-harmonic, i.e.
$$
L\mathbf{1}=0.
$$

Let us state a classical lemma (Liouville theorem under elliptic Harnack inequality) whose proof is omitted here (see e.g.~\cite[Lem.~6.3]{AldanaCarronTapie}).

\begin{lemma}\label{lem:ellipticHarnack}
Let $(X,\dist,\mu)$ be a metric measure space equipped with a Dirichlet form $\cE$ whose associated operator $L$ satisfies an elliptic Harnack inequality, meaning that there exists a constant $C_E>0$ such that for any ball $B \subset X$ and any non-negative local solution $h$ of $Lu=0$ on $B$, we have 
\begin{equation}\label{eq:ellipticHarnack}
\essssup_{(1/2)B} \, h \le C_E \esssinf_{(1/2)B} \, h.
\end{equation}
Then any non-negative locally $L$-harmonic function is constant.\\
\end{lemma}

\textbf{Strongly harmonic functions.}

Let $(X,\dist,\mu)$ be a metric measure space. Following the terminology adopted in \cite{GaczkowskiGorka, AGG19}, for any open set $\Omega \subset X$ we call strongly harmonic function on $\Omega$ any function $h:\Omega\to\setR$ satisfying the mean value property:
$$h(x) = \fint_{B_r(x)} h \di \mu \qquad \forall x \in \Omega, \, \, \forall \, r\in(0, \dist(x,^c \Omega)).$$
\begin{remark}\label{rem:defalternative}
It can easily be checked that a function $h:\Omega\to\setR$ is strongly harmonic if and only if for any $x\in \Omega$ and any $u\in C_c^1\left( \left[0, \dist(x,^c \Omega)\right]\right)$ with 
$\int_X u(\dist(x,y))\di \mu(y) = 1$ one has
$$h(x) = \int_{X}u\left(\dist(x,y)\right) h(y) \di \mu(y).$$
\end{remark}
Under mild assumptions on $(X,\dist,\mu)$, an elliptic Harnack inequality holds true for strongly harmonic functions, provided the doubling condition \eqref{eq:doubling} is satisfied: see \cite[Lemma 4.1]{AGG19}. The next lemma is an easy consequence of this fact. We recall that a metric space is called proper if any closed ball is compact, and that proper metric spaces are complete and locally compact.

\begin{lemma}\label{lem:Harnackstrongly}
Let $(X,\dist)$ be a proper metric space equipped with a regular Borel measure $\mu$ such that $0<\mu(B)<+\infty$ for any metric ball $B \subset X$. Assume that $(X,\dist,\mu)$ satisfies the doubling condition \eqref{eq:doubling}. Then any non-negative strongly harmonic function on $X$ is constant.
\end{lemma}

When $(X,\dist,\mu)$ has an $\alpha$-dimensional volume, strongly harmonic functions satisfy the following two properties:

\begin{lemma}\label{lem:0109}
Let $(X,\dist,\mu)$ be with an $\alpha$-dimensional volume and $h:X\to\setR$ be strongly harmonic. Then:
\begin{enumerate}
\item[(i)] if $h$ has linear growth -- meaning that there exists $C>0$ such that $|h|\le C(1+\dist(o,\cdot))$ for some $o \in X$ -- then $h$ is Lipschitz;
\item[(ii)] if $h$ is continuous and such that $\sup_{\partial B_{r_i}(o)} |h| =o(r_i)$ for some point $o \in X$ and some sequence $\{r_i\}_i \subset (0,+\infty)$ such that $r_i \to +\infty$, then $h$ is constant.
\end{enumerate}
\end{lemma}

\begin{proof}
Let us first prove $(i)$. Assuming $h$ to have linear growth, we know that there exists $o \in X$, $r_o>0$ and $M>0$ such that $|h(z)| \le M\dist(o,z)$ for all $z \in X \backslash B_{r_o}(o)$. Since $h$ is strongly harmonic, we have
$$
\mu(B_{r+d}(x)) h(x) - \mu(B_r(y)) h(y) = \int_{B_{r+d}(x) \backslash B_r(y)} h \di \mu
$$
for all $r>0$ and any given $x,y \in X$, where we have set $d:=\dist(x,y)$. Since $\mu(B_{r+d}(x)) = \omega_\alpha (r+d)^\alpha$ and $\mu(B_r(y)) = \omega_\alpha r^\alpha$, we obtain
\begin{align}\label{eq:D4.1}
|\omega_\alpha (r+d)^\alpha h(x) - \omega_\alpha r^\alpha h(y)| & \le \omega_\alpha ((r+d)^\alpha - r^\alpha)\, \mathrm{sup}_{B_{r+d}(x){\color{blue}\backslash B_r(y)}} \, |h| \nonumber \\
& \le \omega_\alpha ((r+d)^\alpha - r^\alpha) \, \mathrm{sup}_{B_{r+d+\dist(o,x)}(o){\color{blue}\backslash B_r(y)}} \, |h|
\end{align}
since $B_{r+d}(x) \subset B_{r+d+\dist(o,x)}(o)$. Choosing $r>r_o{\color{blue}+\dist(o,y)}$ in order to ensure that  $B_r(y)$ contains $B_{r_o}(o)$, we get $ \mathrm{sup}_{B_{r+d+\dist(o,x)}(o){\color{blue}\backslash B_r(y)}} \, |h| \le M(r+d+\dist(o,x)),$
hence $$|(1+d/r)^\alpha h(x) - h(y)| \le ((1+d/r)^\alpha-1)M(r+d+\dist(o,x)).$$ Letting $r \to +\infty$ and applying $(1+d/r)^\alpha-1 =\alpha d/r +o(1/r)$ yields to $|h(x) - h(y)| \le \alpha dM$.

To prove $(ii)$, apply \eqref{eq:D4.1} with $r=R_i:=r_i-d-\dist(o,x)$ to get
$$
\left| h(x) (1+d/R_i)^\alpha - h(y)\right| \le ((1+d/R_i)^\alpha-1) \,\mathrm{sup}_{B_{r_i}(o)} |h|.
$$
By the weak maximum principle \cite[Cor.~4.3]{AGG19}, we have $\sup_{B_{r_i}(o)} |h| = \sup_{\partial B_{r_i}(o)} |h|$. Since $(1+d/R_i)^\alpha-1= \alpha d/ R_i + o(1/R_i) = O(1/r_i)$  when $i\to +\infty$, then there exists $i_o$ and $C>0$ such that
$$\left| h(x) (1+d/R_i)^\alpha - h(y)\right| \le C r_i^{-1}\mathrm{sup}_{\partial B_{r_i}(o)} |h|
$$
for all $i \ge i_o$. This implies $h(x)=h(y)$ by letting $i$ tend to $+\infty$.
\end{proof}

\hfill

\textbf{Tangent cones at infinity.}

We refer to \cite{Gromov} for a definition of the Gromov-Hausdorff distance $\dist_{GH}$ between compact metric spaces and only mention here that a sequence of compact metric spaces $\{(X_i,\dist_i)\}$ converges to another compact metric space $(X,\dist)$ with respect to the Gromov-Hausdorff distance (what we denote by $\dist_{GH}(X_i,X) \to 0$) if and only if there exists an infinitesimal sequence $\{\eps_i\}_i \subset (0,+\infty)$ and functions $\phi_i : X_i \to X$ called $\eps_i$-isometries such that $|\dist(\phi_i(x),\phi_i(x'))-\dist_i(x,x')|\le \eps_i$ for any $x,x' \in X_i$ and any $i$. If $x_i \in X_i$ for any $i$ and $x \in X$ are such that $\dist(\phi_i(x_i),x) \to 0$, we write $x_i \stackrel{GH}{\to} x$.

When dealing with non-compact spaces, we say that a sequence of pointed metric spaces $\{(X_i, \dist_i, x_i)\}_i$ converges in the pointed Gromov-Hausdorff topology to $(X, \dist, x)$ if there exist sequences of positive numbers $\eps_i \downarrow 0$, $R_i \uparrow \infty$, and of Borel maps $\phi_i:B_{R_i}(x_i) \to X$, also called $\eps_i$-isometries, such that such that for any $i$ the ball $B_{R_i}(x)$ is included in the $\eps_i$-neighborhood of $\phi_i(B_{R_i}(x_i))$, namely $ \bigcup_{y \in \phi_i(B_{R_i}(x_i))} B_{\eps_i}(y)$, $|\dist_i(y, z)-\dist(\phi_i(y), \phi_i(z))|<\eps_i$ for any $y,\, z \in B_{R_i}(x_i)$, and $\dist(\phi_i(x_i),x)\to 0$ (which we also abbreviate to  $x_i \stackrel{GH}{\to} x$).

Pointed measured Gromov-Hausdorff convergence of a sequence of pointed metric measure spaces $\{(X_i, \dist_i, \mu_i, x_i)\}$ to $(X, \dist,\mu,x)$ is set as pointed Gromov-Hausdorff convergence of $\{(X_i, \dist_i, x_i)\}$ to $(X, \dist, x)$ with the additional requirement $(\phi_i)_{\sharp}\mu_i \stackrel{C_\bs(X)}{\rightharpoonup} \mu$ where $C_\bs(X)$ is the space of continuous functions with bounded support and $ f_\sharp $ is
the push forward operator between measures induced by a Borel map $f$.

A metric space $(X,\dist)$ is called metric doubling if there exists a positive integer $N$ such that any ball in $(X,\dist)$ can be covered by at most $N$ balls with half its radius. Whenever $(X,\dist)$ is a doubling space, for any $o \in X$, the family of pointed spaces $\{(X,r^{-1}\dist,o)\}_{r>1}$ satisfies the assumptions of Gromov's precompactness theorem \cite[Prop.~5.2]{Gromov}, henceforth it admits limit points in the pointed Gromov-Haudorff topology as $r\uparrow +\infty$. These pointed metric spaces are called tangent cones at infinity of $(X,\dist)$ in $o$.

It is well-known that when $(X,\dist,\mu)$ is satisfies the doubling property \eqref{eq:doubling}, then the metric space $(X,\dist)$ is metric doubling: see e.g.~\cite[Section 2.5]{AmbrosioColomboDiMarino}.

When a metric measure space $(X,\dist,\mu)$ has an $\alpha$-dimensional volume, a simple computation shows that it is measure doubling, with $C_D=2^\alpha$. Moreover, one can equip any of its tangent cones at infinity $(\uX,\udist,\uo)$ with a limit measure $\umu$ in the following way. Let $\{r_i\}_i$ be a sequence of positive real numbers diverging to $+\infty$ such that $(\uX,\udist,\uo)$ is the pointed Gromov-Hausdorff limit of $\{(X,r_i^{-1}\dist,o)\}_i$. Set $\mu_i:=r_i^{-\alpha}\mu$ for any $i$, and note that
$$
\mu_i(B_r^{\dist_i}(x)) = \mu_i(B_{r r_i}(x)) = \omega_\alpha r^\alpha \qquad \forall x \in X, r>0.
$$
Set $\underline{V}(x,r):=\omega_\alpha r^\alpha$ for any $x \in X$ and $r>0$. Then for any $\delta>0$ and any Borel set $A$ of $(\uX,\udist)$, setting
$$
\umu_\delta(A) := \inf\left\{ \sum_{i} \underline{V}(z_i,r) \, : \, \{B_{r_i}(z_i)\}_i\, \, \text{s.t. $A \subset \bigcup_i B_{r_i}(z_i)$ and $r_i \le \delta$} \right\}
$$
and then $\umu(A) = \lim\limits_{\delta \to 0} \umu_\delta(A)$ defines a metric outer measure $\umu$ on $(\uX,\udist)$ whose canonically associated measure, still denoted by $\umu$, is a Radon measure satisfying $\umu(B_r(\ux)) = \omega_\alpha r^\alpha$ for any $\ux \in \uX$ and $r>0$.
This shows that $(\uX,\udist,\umu)$ has an $\alpha$-dimensional volume. Moreover, we obviously have $\umu(B_r(\ux))=\lim_{i \to +\infty} \mu_i(B_r^{\dist_i}(x_i))$ for any $r>0$ and any sequence $x_i \stackrel{GH}{\to} x$; by density in $C_\bs(\uX)$ of the space spanned by the collection of characteristic functions of balls, this implies the pointed measured Gromov-Hausdorff convergence $(X,r_i^{-1}\dist,\mu_i,o) \to (\uX,\udist,\umu,\uo)$.\\

\textbf{Ascoli-Arzelà type theorems}

Let $\{(X_i,\dist_i,x_i)\}_i, (X,\dist,x)$ be pointed proper metric spaces such that $$(X_i,\dist_i,x_i) \to (X,\dist,x)$$ in the pointed Gromov-Hausdorff topology and $\phi_i:B_{R_i}(x_i) \to X$ be $\eps_i$-isometries, where $\{\eps_i\}_i, \{R_i\}_i \subset (0,+\infty)$ are such that $\eps_i \downarrow 0$ and $R_i \uparrow +\infty$. For any $i$, let $K_i$ be a compact subset of $X_i$, and assume that there exists $K \subset X$ compact such that $\dist_{GH}(K_i,K) \to 0$. We say that functions $f_i:X_i\to\setR$ converge to $f:X\to\setR$ uniformly over $K_i \to K$ if $\sup_{K_i} |f_i - f \circ \phi_i| \to 0$. Note that this definition depends on the choice of the $\eps_i$-isometries $\phi_i$ that we keep fixed for the rest of this paragraph.

\begin{remark}
In the rest of the article, whenever we consider a convergent sequence of pointed metric spaces $(X_i,\dist_i,x_i) \to (X,\dist,x)$, we always implicitly assume that sequences $\{\eps_i\}_i, \{R_i\}_i \subset (0,+\infty)$ with $\eps_i\downarrow 0, R_i \uparrow +\infty$ and $\eps_i$-isometries $\phi_i:B_{R_i}(x_i)\to X$ have been chosen a priori and that the statements ``$x_i \stackrel{GH}{\to} x$'' and ``$f_i \to f$ uniformly on compact sets'' are meant with these $\eps_i$-isometries.
\end{remark}

 In this context, we have the following Ascoli-Arzelà theorem:

\begin{proposition}\label{prop:AA1}
Let $\{(X_i,\dist_i,x_i)\}_i, (X,\dist,x)$ be as above, and $r>0$. For any $i$, let $f_i \in C(X_i)$ be such that:
\begin{itemize}
\item $\sup_i \|f_i\|_{L^{\infty}(\overline{B}_r(x_i))} < +\infty$,
\item the sequence $\{f_i\}_i$ is asymptotically uniformly continuous on $\overline{B}_r(x)$ (see \cite[Def.~3.2]{Honda}).
\end{itemize}
Then $\{f_i\}_i$ admits a subsequence $(f_{i(j)})_j$ which converges to $f$ uniformly over $\overline{B}_r(x_i) \to \overline{B}_r(x)$.
\end{proposition}

\begin{proof}
From \cite[Prop.~3.3]{Honda}, we know that for $\{f_i\}_i$ satisfying the above assumptions, there exists $f \in C(B_r(x))$ and a subsequence $(f_{i(j)})_j$ such that $f_{i(j)}(x_j) \to f(x)$ whenever $x_j \stackrel{GH}{\to} x \in B_r(x)$. With no loss of generality, we can assume that the subsequence is the whole sequence itself. By contradiction, assume that the uniform convergence $f_i\to f$ over $\overline{B}_r(x_i)\to \overline{B}_r(x)$ is not satisfied. Then there is some $\eps>0$ and  a subsequence $(f_{i(\ell)})_\ell$ such that $\inf_\ell \{\sup_{\overline{B}_r(x_{i(\ell)})} |f_{i(\ell)}-f \circ \phi_{i(\ell)}|\} \ge \eps$. Again, we can assume that the subsequence is the whole sequence itself. For any $i$, choose $y_i \in \overline{B}_r(x_i)$ such that $|f_i(y_i)-f \circ \phi_i(y_i)| \ge \eps/2$ and set $z_i:=\phi_i(y_i) \in \overline{B}_{r+\eps_i}(x)$. Properness of $X$ implies that the sequence $\{z_i\}_i$ converges to some $z \in \overline{B}_r(x)$, up to extraction. In particular, $y_i \stackrel{GH}{\to} z$. Then in
$$
\eps/2 \le |f_i(y_i)-f(z)| + |f(z)-f\circ \phi_i(y_i)|,
$$
the first term in the right-hand side goes to $0$ when $i$ tend to $+\infty$.  Since $f$ is continuous, we also have $|f(z)-f\circ \phi_i(y_i)|\to 0$ when $i$ tend to $+\infty$, hence a contradiction.
\end{proof}

Let $\{(X_i,\dist_i,x_i)\}_i, (X,\dist,x),\{\phi_i\}_i$ be as above. Let $(Y,\dist_Y)$ be another metric space. We say that $f_i : Y\to X_i$ converge to $f : Y\to X$ uniformly on compact subsets of $Y$ if $\sup_K \dist_i(\phi_i \circ f_i,f) \to 0$ for any compact set $K \subset Y$.

An Ascoli-Arzelà theorem is also available in this context. We state it with an equi-Lipschitz assumption which is enough for our purposes. The proof is omitted for brevity.

\begin{proposition}\label{prop:AA2}
Let $\{(X_i,\dist_i,x_i)\}_i, (X,\dist,x)$ be as above. Let $(Y,\dist_Y)$ be a metric space and $f_i : Y\to X_i$ be Lipschitz functions such that:
\begin{itemize}
\item $L:=\sup_i \Lip(f_i) < +\infty$,
\item there exists $y\in Y$ and $r>0$ such that $\dist_i(f_i(y),x_i)\le  r$ for any $i$.
\end{itemize} Then $\{f_i\}_i$ admits a subsequence converging uniformly on compact sets of $Y$ to some Lipschitz function $f:Y\to X$, and $\Lip(f) \le L$.
\end{proposition}

Let us conclude this paragraph with a stability result for strongly harmonic functions.

\begin{proposition}\label{prop:stabstrongharmonic}
Let $\{(X_i,\dist_i,\mu_i,x_i)\}_i, (X,\dist,\mu,x)$ be proper pointed metric measured spaces such that $(X_i,\dist_i,\mu_i,x_i) \to (X,\dist,\mu,x)$ in the pointed measured Gromov-Hausdorff topology. Let $f_i \in C(X_i)$ be converging to $f \in C(X)$ uniformly over $\overline{B}_r(x_i) \to \overline{B}_r(x)$ for any $r>0$. Assume that $f_i$ is strongly harmonic for any $i$. Then $f$ is strongly harmonic.
\end{proposition}

\begin{proof} 
By the characterization of strongly harmonic functions stated in Remark \ref{rem:defalternative}, it is enough to establish
\begin{equation}\label{eq:fstrongharm}
f(y) = \fint_{B_r(y)}u(\dist(y,z))f(z)\di \mu(z)
\end{equation}
for any given $r>0$, $y \in X$ and $u \in C^1_c([0,+\infty))$ such that $\int_X u(\dist(y,z)) \di \mu(z)=1$.
Let $y_i \in X_i$ for any $i$ be such that $y_i \stackrel{GH}{\to} y$. For any $i$, set $$u_i:=\frac{u}{\int_{X_i} u(\dist_i(y_i,z)) \di \mu_i(z)}$$ and note that
\begin{equation}
\int_{X_i} u_i(\dist_i(y_i,z)) \di \mu_i(z) = 1
\end{equation}
so that $f_i$ being strongly harmonic implies
\begin{equation}\label{eq:fistronharm}
f_i(y_i) = \fint_{B_r(y_i)}u_i(\dist_i(y_i,z))f_i(z)\di \mu_i(z).
\end{equation}
But
\begin{align*}
\int_{X_i} u(\dist_i(y_i,z)) \di \mu_i(z)  & = - \int_0^{+\infty} u'(r) \mu_i(B_r^{\dist_i}(y_i)) \di r\\
& \to - \int_0^{+\infty} u'(r) \mu(B_r^{\dist}(y)) \di r = \int_{X} u(\dist(y,z)) \di \mu(z)=1,
\end{align*}
so $u_i \to u$ uniformly on $(0,+\infty)$: this implies that the functions $u_i(\dist_i(y_i,\cdot))f_i \in C(X_i)$ converge uniformly over all compact sets to $u(\dist(y,\cdot))f \in C(X)$. Therefore, letting $i$ tend to $+\infty$ in \eqref{eq:fistronharm} provides \eqref{eq:fstrongharm}.
\end{proof}

\hfill

\textbf{Length structures.}

Let $(X,\dist)$ be a metric space. A path in $X$ is a continuous map $c:[0,1]\to X$. It is called rectifiable if its length
\[
L_\dist(c):=\sup \left\{ \sum_{i=1}^n \dist(c(t_i),c(t_{i-1})) \, \, : \, \, 0 = t_0 < \ldots < t_n = 1, \, \,  n \in \setN\backslash\{0\} \right\}
\]
is finite. $(X,\dist)$ is called length metric space if for any $x, y \in X$,
$$
\dist(x,y) = \inf\{L_\dist (c) \, : \, c \in \Omega_{xy}\},
$$
where $\Omega_{xy}$ is the set of rectifiable paths in $X$ such that $c(0)=x$ and $c(1)=y$. A geodesic space is a trivial example of length space. Equivalently, $(X,\dist)$ is length if $\dist$ coincides with its associated length distance $\overline{\dist}$ defined by:
$$
\overline{\dist}(x,y) := \inf\{L_\dist (c) \, : \, c \in \Omega_{xy}\}\qquad \forall x,y \in X,
$$
in which case we say that $\dist$ is a length distance. Note that we always have $\dist \le \overline{\dist}$ and $L_{\dist}(c)=L_{\overline{\dist}}(c)$ whenever $c$ is a rectifiable path in $X$. Moreover,
$$
L_\dist(c) = \lim\limits_{\alpha \to 0^+} L_{\dist,\alpha}(c)
$$
where
\[
L_{\dist,\alpha}(c) = \sup \left\{ \sum_{i=1}^n \dist(c(t_i),c(t_{i-1})) \,  :  \, 0 = t_0 < \ldots < t_n = 1, |t_i - t_{i+1}|<\alpha \, \,  \forall i,\, \,  n \in \setN\backslash\{0\} \right\}.
\]
In this context, we have the following lemma.

\begin{lemma}\label{lem:length}
Let $(X,\delta)$ be a length metric space. Assume that $\dist$ defined as $\dist:=2\sin(\delta/2)$ is a distance. Then its associated length distance $\overline{\dist}$ coincides with $\delta$.
\end{lemma}

\begin{proof}
First note that $(2/\pi) \delta \le \dist \le \delta$ because $(2\pi)x\le2\sin(x/2)\le x$ for any $x\ge 0$. In particular, a map $c:[0,1]\to X$ is continous for $\delta$ if and only if it is for $\dist$. Moreover, since $\delta$ is a length distance, $\dist \le \delta$ implies $\overline{\dist} \le \delta$, so we are left with proving the converse inequality. Let $c$ be a path in $X$. Being continuous, $c$ is also uniformly continuous: for any $\epsilon\in (0,1)$, there exists $\alpha>0$ such that for any $t,s \in [0,1]$,
$$
|t-s| < \alpha \quad \Rightarrow \quad \delta(c(t),c(s)) < \epsilon.
$$
Since $x- 2\sin(x/2)\le x^2$ for any $x\ge0$, then $$\delta - \dist \le \delta^2,$$ so that for any $t,s \in [0,1]$:
$$
|t-s| < \alpha \quad \Rightarrow \quad \delta(c(t),c(s)) - \dist(c(t),c(s)) \le \epsilon \delta(c(t),c(s)).
$$
This implies $L_{\delta,\alpha}(c) - L_{\dist,\alpha}(c) \le \epsilon L_{\delta,\alpha}(c)$ and thus $(1-\eps)L_\delta(c) \le L_\dist(c)$ by letting $\alpha$ tend to $0$. Letting $\eps$ tend to $0$ provides $L_\delta(c) \le L_\dist(c)$. This implies $\delta \le \overline{\dist}$.
\end{proof}

\hfill

\textbf{Busemann functions.}

Let $(X,\dist)$ be a metric space. A geodesic ray in $X$ is a continuous function $\gamma:[0,+\infty)\to X$ such that $\dist(\gamma(t),\gamma(s))=|t-s|$ for any $s,t\ge 0$. The Busemann function associated to a geodesic ray $\gamma$ is defined by
$$
b_\gamma(x)=\lim\limits_{t \to +\infty} t - \dist(x,\gamma(t)).
$$
Note that this limit is well-defined for any $x \in X$ since the function $t \mapsto t - \dist(x,\gamma(t))$ is non-decreasing and bounded from above by $\dist(o,x)$. Note also that $b_\gamma$ is $1$-Lipschitz, since for any $x, y \in X$ and any $t>0$, one has $t - \dist(x,\gamma(t)) - (t - \dist(y,\gamma(t)) \le \dist(x,y)$, henceforth $b_\gamma(x) - b_\gamma(y) \le \dist(x,y)$ by letting $t \to +\infty$. Moreover, for any $s>0$, one can easily check that
\begin{equation}\label{eq:Busemann}
b_\gamma(\gamma(s))=s.
\end{equation}

We shall need the following lemma.

\begin{lemma}\label{lem:preparatory}
Let $(X,\dist,\mu)$ be a metric measure space equipped with a Dirichlet form $\cE$ with associated operator $L$, and $\alpha \in \setR$. Assume
\[
L\textbf{1}=0, \qquad \dist(x,\cdot) \in \cD_{loc}(\cE) \quad \, \text{and} \quad \, L\dist(x,\cdot)=\alpha/\dist(x,\cdot) \, \, \text{on $X\backslash \{x\}$},
\] for any $x \in X$. Then any Busemann function on $(X,\dist)$ is  locally $L$-harmonic.
\end{lemma}

\begin{proof}
Let $b_\gamma$ be a Busemann function on $(X,\dist)$. Set $f_s:=s-\dist(\gamma(s),\cdot)$ for any $s>0$ and observe that the assumptions imply that $f_s$ is a local solution on $X \backslash \{x\}$ of $Lu = -\alpha/\dist(\gamma(s),\cdot)$. For any $\phi \in \cD_c(\cE)$, since $\gamma(s) \notin \supp \phi$ and thus $\dist(\gamma(s),\cdot)>0$ on $\supp \phi$ for $s$ large enough, then
$$
|\cE(f_s,\phi)|= \left|\int_X (Lf_s) \phi \di \mu\right| \le \frac{|\alpha|}{\dist(\gamma(s),\supp \phi)} \int_X |\phi| \di \mu \to 0 \quad \text{when $s \to +\infty$}.
$$
Moreover, as $(f_s)_s$ is increasing and converges pointwise to $b_\gamma$, then $f_s \to b_\gamma$ in $L^2(X,\mu)$. Since $\cD_{loc}(\cE)$ is a Fréchet space that can be equipped with the family of semi-norms $\{p_i(g):=(\|g\|_{L^2(K_i)}^2 +( \sup\{\cE(g,\phi):\phi \in \cD(\cE), \, \supp \phi \subset K_i\})^2)^{1/2}\}_i$, where $\{K_i\}_i$ is an exhaustion of $X$ by compact sets, then $b_\gamma \in \cD_{loc}(\cE)$ and $\cE(b_\gamma,\phi)=0$ for any $\phi \in \cD_c(\cE)$. In particular, $b_\gamma$ is locally $L$-harmonic.\\
\end{proof}

\textbf{Laplace transform.}

Let $F:[0,+\infty)\to \setR$ be a locally integrable function such that $F(t)=O(e^{\gamma t})$ when $t \to +\infty$ for some $\gamma \in \setR$. The Laplace transform of $F$ is the complex-valued function $\mathcal{L}\{F\}$ defined by:
$$
\mathcal{L}\{F\}(z) = \int_0^{+\infty} F(\xi)e^{-z\xi}\di \xi, \qquad  \forall z \in \{\mathrm{Re}(\cdot) > \gamma\}.
$$
Lerch's theorem asserts that if $F_1, F_2 : [0,+\infty)\to\setR$ are two continuous, locally integrable functions satisfying $F_1(t), F_2(t)=O(e^{\gamma t})$ when $t \to +\infty$ and $\mathcal{L}\{F_1\} = \mathcal{L}\{F_2\}$ on $\{\mathrm{Re}>\gamma\}$, then $F_1=F_2$ (see e.g.~\cite[Th.~2.1]{Cohen}). This provides the following lemma.

\begin{lemma}\label{lem:Laplace}
Let $F:(0,+\infty)\to \setR$ be a continuous and locally integrable function such that $F(t)=O(e^{\eps t})$ when $t \to +\infty$ for any $\eps>0$ and
\begin{equation}\label{eq:Laplace}
\mathcal{L}\{F\}(\lambda) = \lambda^{-\alpha-1} \qquad \forall \lambda >\lambda_o
\end{equation}
for some $\alpha>0$ and $\lambda_o \ge 0$. Then $F(\xi)=\xi^\alpha/\Gamma(\alpha+1)$ for any $\xi\ge0$.
\end{lemma}

\begin{proof}
Since $F$ is locally integrable, one can apply the classical theorem on holomorphy under the integral sign to get that $\mathcal{L}\{F\}$ is holomorphic on any compact subset of $\{\mathrm{Re}>0\}$. Therefore, by analytic continuation, \eqref{eq:Laplace} implies $\mathcal{L}\{F\}(z) = z^{-\alpha-1}$ for any $z \in \{ \mathrm{Re}(\cdot)>0\}$. Since the Laplace transform of $\xi \mapsto \xi^{\alpha}$ is $z\mapsto \Gamma(\alpha +1) z^{-\alpha-1}$, Lerch's theorem gives $F(\xi)=\xi^\alpha/\Gamma(\alpha+1)$ for any $\xi\ge0$.
\end{proof}

\section{First rigidity results for spaces with an Euclidean heat kernel}

\quad \, In this section, we establish several properties of metric measure spaces equipped with a Dirichlet form admitting an $\alpha$-dimensional Euclidean heat kernel. We shall use most of these results in the next section to prove Theorem \ref{th:main}. 


\subsection{Stochastic completeness and consequences}

We begin with stochastic completeness.

\begin{lemma}\label{lem:stochastic}
Let $(X,\dist,\mu,\cE)$ be with an $\alpha$-dimensional Euclidean heat kernel. 
Then $(X,\dist,\mu,\cE)$ is stochastically complete.
\end{lemma}

\begin{proof}
Take $t,s>0$ and $x \in X$. By \eqref{eq:ChapmanKolmogorov}, for any $y \in X$ we have
\begin{equation}\label{eq:Lem3.1}
\int_X p(x,z,t)e^{-\frac{\dist^2(z,y)}{4s}} \di \mu(z) = \left(\frac{s}{t+s}\right)^{\alpha/2} e^{-\frac{\dist^2(x,y)}{4(t+s)}}.
\end{equation}
Letting $s \to +\infty$ and applying the monotone convergence theorem, we get the result.\\
\end{proof}

As a consequence of Lemma \ref{lem:stochastic}, we can show that spaces with an $\alpha$-dimensional Euclidean heat kernel have an $\alpha$-dimensional volume.

\begin{lemma}\label{lem:euclideanvolume}
Let $(X,\dist,\mu,\cE)$ be with an $\alpha$-dimensional Euclidean heat kernel. Then $(X,\dist,\mu)$ has an $\alpha$-dimensional volume.
\end{lemma}

\begin{proof}
Take $x \in X$. By Lemma \ref{lem:stochastic}, we have:
$$
\int_X p(x,y,t) \di \mu(y) = 1 \qquad \forall \, t>0,
$$
so that the hypothesis on the heat kernel implies:
\begin{equation}\label{eq:001}
\int_X e^{-\frac{\dist^2(x,y)}{4t}} \di \mu(y) = (4\pi t)^{\alpha/2} \qquad \forall \, t>0.
\end{equation}
By Cavalieri's principle (see for instance \cite[Lemma 5.2.1]{AmbrosioTilli}), we have 
$$
\int_X e^{-\frac{\dist^2(x,y)}{4t}}\di \mu(y) = \int_0^{+\infty} \mu(\{e^{-\frac{\dist^2(x,\cdot)}{4t}}>s\}) \di s.
$$
Since for any $y \in X$, one has $e^{-\frac{\dist^2(x,y)}{4t}}>s$ if and only if $\dist^2(x,y) < - 4t \log(s)$, then
$$
\bigg\{e^{-\frac{\dist^2(x,\cdot)}{4t}}>s\bigg\}
= \begin{cases}
\emptyset & \text{if $s\ge 1$},\\
B_{\sqrt{-4t\log(s)}}(x) & \text{if $s<1$}.
\end{cases}
$$
Therefore, the change of variable $\xi=-4t\log(s)$ yields to
$$
\int_X e^{-\frac{\dist^2(x,y)}{4t}}\di \mu(y) = \frac{1}{4t} \int_0^{+\infty} e^{-\frac{\xi}{4t}} \mu(B_{\sqrt{\xi}}(x)) \di \xi.
$$
Coupled with \eqref{eq:001}, and setting $\lambda=1/(4t)$, this leads to:
$$
\int_0^{+\infty} e^{-\lambda \xi} \mu(B_{\sqrt{\xi}}(x)) \di \xi = \pi^{\alpha/2} \lambda^{-\alpha/2-1} \qquad \forall \lambda>0.
$$
Applying Lemma \ref{lem:Laplace} and \eqref{eq:omega_n} provides the result.
\end{proof}

A second consequence is that complete spaces with an $\alpha$-dimensional Euclidean heat kernel are proper; in particular, they are locally compact. Note that the space $(\setR^n \backslash \{0\}, \dist_e,\leb^n)$ shows that completeness is a non-removable assumption.

\begin{lemma}
Let $(X,\dist,\mu,\cE)$ be with an $\alpha$-dimensional Euclidean heat kernel and such that $(X,\dist)$ is complete. Then any closed ball in $X$ is compact.
\end{lemma}

\begin{proof}
Let $B$ be a closed ball in $X$ with center $o$ and radius $R$. Take $(x_k)_k \subset B$ and $t>0$. Set $u_k:=p(x_k,\cdot,t/2)$ for any $k$, and note that $\int_X u_k^2 \di \mu=p(x_k,x_k,t)=(4\pi t)^{-\alpha/2}$ by the Chapman-Kolmogorov property \eqref{eq:ChapmanKolmogorov}. Then the sequence $(u_k)_k$ is bounded in $L^2(X,\mu)$, so it weakly converges to some $u_\infty \in L^2(X,\mu)$. Let $L$ be the operator canonically associated to $\cE$. Set $v_k:=e^{-(t/2)L}u_k$ for any $k$ and $v_\infty = e^{-(t/2)L}u_\infty$. By the Chapman-Kolmogorov property, we have $v_k(y)=p(x_k,y,t)$ for any $k$ and $y \in X$, and since $\dist^2(y,o)/2 \le \dist^2(y,x_k) + \dist^2(x_k,o) \le \dist^2(y,x_k) + R^2$, then
$$
|v_k(y)| = (4\pi t)^{-\alpha/2}e^{-\frac{\dist^2(x_k,y)}{4t}} \le (4\pi t)^{-\alpha/2}e^{-\frac{\dist^2(o,y)}{8t}} e^{\frac{R^2}{4t}}=:w(y).
$$
The Chapman-Kolmogorov property implies easily that $w$ is in $L^2(X,\mu)$. Moreover, for any $y \in X$, the $L^2$ weak convergence $u_k \to u_\infty$ implies
$$
v_k(y)=\int_X p(y,z,t/2)u_k(z)\di \mu(z) \to \int_X p(y,z,t/2)u_\infty(z)\di \mu(z) = v_\infty(y).
$$
Then $v_k \to v_\infty$ in $L^2(X,\mu)$ by Lebesgue's dominated convergence theorem. Therefore, $(v_k)_k$ is a Cauchy sequence in $L^2(X,\mu)$. Since
\begin{align*}
\|v_k - v_l\|_{L^2}^2 & = \|v_k \|_{L^2}^2 + \|v_l\|_{L^2}^2 - 2 \int_X v_k v_l \di \mu \\ & = p(x_k,x_k,2t) + p(x_l,x_l,2t) - 2 p(x_k,x_l,2t) = \frac{2-2e^{-\frac{\dist^2(x_k,x_l)}{8t}}}{(8\pi t)^{\alpha/2}}
\end{align*}
for all $k,l$, we get that $(x_k)_k$ is a Cauchy sequence, hence the result.
\end{proof}

Let us conclude with an important lemma.

\begin{lemma}
Let $(X,\dist,\mu,\cE)$ be with an $\alpha$-dimensional Euclidean heat kernel such that $(X,\dist)$ is complete. Then $(X,\dist)$ is a geodesic space.
\end{lemma}

\begin{proof}
Let us begin with showing that any two points $x,y \in X$ admit a midpoint, i.e.~a point $m \in X$ such that
\begin{equation}\label{eq:middle}
\dist(x,m)=\dist(y,m)=\frac{\dist(x,y)}{2}\, \cdot
\end{equation}
Set $F(z)=\dist^2(x,z)+\dist^2(z,y)$ for any $z \in X$. Since $F(z) \to +\infty$ when $\dist(x,z), \dist(z,y) \to +\infty$, then there exists a ball $B \subset X$ such that $\inf_X F = \inf_B F$. From the previous lemma, we know that balls in $X$ are compact, so $\inf_B F$ is attained in some $m \in B$. Therefore, setting \[\lambda:=F(m)=\dist^2(x,m)+\dist^2(m,y),\] we have
$$
\|e^{-\frac{F}{4}}\|_{L^\infty(X,\mu)} = e^{-\frac{\lambda}{4}}.
$$
By the Chapman-Kolmogorov identity \eqref{eq:ChapmanKolmogorov}, we have for any $t>0$
$$
\int_X e^{-\frac{\dist^2(x,z)+\dist^2(z,y)}{4t}}\di \mu(z) = e^{-\frac{\dist^2(x,y)}{8t}}(2\pi t)^{\alpha/2}
$$
which can be raised to the power $t$ to provide
$$
\|e^{-\frac{F}{4}}\|_{L^{1/t}(X,\mu)} = e^{\frac{\dist^2(x,y)}{8}}(2\pi t)^{\alpha t/2}.
$$
Letting $t$ tend to $0$, this yields to $e^{-\frac{\lambda}{4}}=e^{-\frac{\dist^2(x,y)}{8}}$ hence $\lambda=\frac{\dist^2(x,y)}{2}$, thus \begin{equation}\label{eq:lambda2} \dist^2(x,m)+\dist^2(m,y) = \frac{\dist^2(x,y)}{2}
\end{equation}
by definition of $\lambda$. Since for any $z \in X$,
\begin{align*}
\dist^2(x,z)+\dist^2(z,y) & = \frac{1}{2}(\dist(x,z)+\dist(z,y))^2 + \frac{1}{2}(\dist(x,z)-\dist(z,y))^2\\
&\ge \frac{1}{2}(\dist(x,y))^2 + \frac{1}{2}(\dist(x,z)-\dist(z,y))^2,
\end{align*}
taking $z=m$ and using \eqref{eq:lambda2} implies $\dist(x,m)=\dist(m,y)$.

The existence of midpoints implies that $(X,\dist)$ is a length space, see \cite[Th.~2.4.16, 1.]{BuragoBuragoIvanov}. Then the result follows from \cite[Th.~2.5.23]{BuragoBuragoIvanov} and \cite[Th.~2.5.9]{BuragoBuragoIvanov}.
\end{proof}
\begin{remark}\label{rem:geodesicplus}
The previous proof can be adapted to show that if a complete proper metric measure space $(X,\dist,\mu)$ can be endowed with a symmetric Dirichlet form $\cE$ admitting a heat kernel $p$ such that for any $x,y\in X$,
$$\lim_{t\to 0+}-4t\log p(x,y,t)=\dist^2(x,y),$$ where the convergence holds locally uniformly, then $(X,\dist)$ is geodesic.
\end{remark}

\subsection{Strong locality and regularity of the Dirichlet form}

Let us show now that having an $\alpha$-dimensional Euclidean heat kernel forces a Dirichlet form to satisfy several properties. We start with the following.

\begin{lemma}
Let $(X,\dist,\mu,\cE)$ be with an $\alpha$-dimensional Euclidean heat kernel. Then $\Lip_c(X) \subset \cD(\cE)$.
\end{lemma}

\begin{proof}
Let $f \in \Lip_c(X)$ be with support $K$. Thanks to \eqref{eq:Grigor'yan}, we only need to show that
$$
F \, : \, (0,+\infty) \ni t \mapsto \frac{1}{2t} \iint_{X\times X} (f(x)-f(y))^2 \frac{1}{(4 \pi t)^{\alpha/2}} e^{-\frac{\dist^2(x,y)}{4t}} \di \mu(x) \di \mu(y)
$$
is a bounded function. Set $D(x,y)=f(x)-f(y)$ for any $(x,y) \in X \times X$. It is easily checked that $\supp(D) \subset (K \times X) \cup (K \times X)$, so for any $t>0$, using the symmetry in $x$ and $y$ of the integrand we get
\begin{align*}
F(t) & \le \frac{\Lip(f)^2}{t} \int_K\int_X \dist^2(x,y) \frac{1}{(4 \pi t)^{\alpha/2}} e^{-\frac{\dist^2(x,y)}{4t}} \di \mu(x) \di \mu(y)\\
& = \frac{4 \Lip(f)^2}{(4 \pi t)^{\alpha/2}} \int_K \int_X \frac{\dist^2(x,y)}{4t} e^{-\frac{\dist^2(x,y)}{4t}} \di \mu(x) \di \mu(y).
\end{align*}
For any measurable function $g:X\to [0,\infty)$ and any $C^1$ function $\phi : [0,\infty) \to[0,\infty)$ satisfying $\lim_{\lambda\to+\infty}\varphi(\lambda)=0$ and $\int_0^{+\infty} |\varphi'|(\lambda)\ \mu\left(\left\{g<\lambda\right\}\right)d\lambda<\infty$, writing $\varphi(g(x))=\int_{g(x)}^{+\infty} \varphi'(\lambda)d\lambda$ and applying Fubini's theorem leads to
\begin{equation}\label{eq:Fub}
\int_X \varphi(g(x))\di\mu(x)=-\int_0^{+\infty} \varphi'(\lambda)\ \mu\left(\left\{g<\lambda\right\}\right)\di\lambda.
\end{equation}
For any $y \in K$, using this fact with $g(x)=\dist^2(x,y)/(4t)$ and $\phi(\xi)=\xi e^{-\xi}$, we get
$$
F(t) \le \frac{4 \Lip(f)^2}{(4 \pi t)^{\alpha/2}} \int_K \int_0^{+\infty} (\lambda - 1) e^{-\lambda} \mu(B_{\sqrt{4 t \lambda}}(y)) \di \lambda \di \mu(y).
$$
Setting $C_o=C_o(\alpha):=4 \int_0^{+\infty}(\lambda-1)e^{-\lambda}\lambda^{\alpha/2} \di \lambda$ and recalling that $\mu(B_{\sqrt{4 t \lambda}}(y))= \omega_{\alpha} (4 t  \lambda)^{\alpha/2}$, we obtain $F(t) \le \Lip(f)^2 \mu(K) C_o \omega_{\alpha} \pi^{-\alpha/2}$, thus $F$ is bounded.
\end{proof}

\begin{corollary}\label{cor:1}
Let $(X,\dist,\mu,\cE)$ be with an $\alpha$-dimensional Euclidean heat kernel. Then $\Lip(X) \subset \cD_{loc}(\cE)\cap C(X)$ and for some constant $C_1$ depending only on $\alpha$, we have:
\begin{equation}\label{eq:Lip}
\Gamma(f) \le C_1 \Lip(f)^2 \mu \qquad \forall f \in \Lip(X).
\end{equation}
\end{corollary}

\begin{proof}
Take $f \in \Lip(X)$. For any compact set $K \subset X$, the function $\phi_K:=\max(1-\dist(\cdot,K),0)$ is a compactly supported Lipschitz function constantly equal to $1$ on $K$. Therefore, $f \phi_K$ coincides with $f$ on $K$, and thanks to the previous lemma, $f \phi_K$ belongs to $\cD(\cE)$. This shows that $f \in \cD_{loc}(\cE)$. Moreover, for any  non-negative $\phi \in C_c(X)$ and $t>0$, a direct computation like in the proof of the previous lemma implies
$$
\frac{1}{2t} \iint_{X\times X} \phi(x)(f(x)-f(y))^2 p(x,y,t) \di \mu(x) \di \mu(y) \le C_1 \Lip(f)^2 \int_X \phi(x) \di \mu(x)
$$
with $C_1$ depending only on $\alpha$, so that letting $t$ tend to $0$ and applying formula \eqref{eq:Grigor'yan2} yields to \eqref{eq:Lip}.
\end{proof}

We are now in a position to show the following crucial result.

\begin{proposition}\label{prop:slandreg}
Let $(X,\dist,\mu,\cE)$ be with an $\alpha$-dimensional Euclidean heat kernel. Then $\cE$ is strongly local and regular.
\end{proposition}

\begin{proof}
In the proof of \cite[Th.~4.1]{Grigor'yan}, it is shown that if $(X,\dist,\mu,\cE)$ admits a stochastically complete heat kernel $p$ satisfying $$t^{-\gamma/\beta} \Phi_1(\dist(x,y)t^{-1/\beta})\le p(x,y,t)\le t^{-\gamma/\beta} \Phi_2(\dist(x,y)t^{-1/\beta})$$ for any $x,y \in X$ and $t>0$, where $\beta$ and $\gamma$ are positive constants and $\Phi_1, \Phi_2$ are monotone decreasing functions from $[0,+\infty)$ to itself such that $\Phi_1>0$ and $\int^{+\infty}s^{\beta + \gamma-1}\Phi_2(s)\di s<+\infty$, then
\begin{equation}\label{eq:Grigor'yan}
\cE(f) \simeq \limsup_{t\to 0} t^{-\frac{(\beta + \gamma)}{2}} \iint_{\{\dist(x,y)<t^{1/2}\}}(f(x)-f(y))^2\di \mu(x)\di \mu(y)
\end{equation}
holds for all $f \in \cD(\cE)$, what in turn implies strong locality of $\cE$. Here we have used $A\simeq B$ to denote the existence of a constant $c>1$ such that $c^{-1}A\le B \le cA$. Choosing $\Phi_1(s)=\Phi_2(s)=e^{-s^2/4}$, $\beta=2$ and $\gamma = \alpha$, we can apply this result in our context to get strong locality of $\cE$.

To prove regularity, let us show that $\Lip_c(X)$ is a core for $\cE$.

\textit{Density in $C_c(X)$.} Let $f \in C_c(X)$ be with support $K$. For any $R>0$ and $x \in X$, set $f_R(x)=\inf_y \{f(y)+R\dist(x,y)\}$. Note that $f_R(x)\le f(x)$ for any $x \in X$, and since $f(y)+R\dist(x,y)\to +\infty$ when $\dist(x,y) \to +\infty$ and $(X,\dist)$ is proper, the infimum in the definition of $f_R$ is always attained at some $x' \in X$. Then $f_R(x)=f(x')+R\dist(x,x')$ implies
$$
\dist(x,x') \le \frac{2\|f\|_{\infty}}{R}.
$$
Being continuous with compact support, $f$ is uniformly continuous, so it admits a modulus of continuity $\omega$ which we can assume non-decreasing with no loss of generality.
Then for any $x \in X$,
$$
|f_R(x)-f(x)|=f(x)-f_R(x)=f(x)-f(x')+\underbrace{f(x')-f_R(x)}_{=-R\dist(x',x)\le 0}  \le \omega(\dist(x,x'))
$$
so that $$
\|f_R-f\|_{\infty} \le \omega\left( 2 \|f\|_{\infty}/R\right) \to 0$$ when $R \to +\infty$. Therefore, setting $\phi_K:=\max(1-\dist(\cdot,K),0)$ and $g_R:=\phi_K f_R$ for any $R>0$, we get a sequence of compactly supported Lipschitz functions $(g_R)_R$ converging uniformly to $f$. 


\textit{Density in $\cD(\cE)$.} Let $\Lip_o(X)$ be the set of Lipschitz functions $f$ on $X$ vanishing at infinity, i.e.~such that for some $o \in X$ one has $f(x)\to 0$ when $\dist(o,x)\to+\infty$. We are going to show that $\Lip_c(X)$ is dense in $\Lip_o(X)\cap \cD(\cE)$ for the norm $|\cdot|_{\cE}$. By \eqref{eq:Grigor'yan}, we know that that there exists a constant $C_\alpha>0$ such that if $f\in \cD(\cE)$, then
$$
 \frac{1}{C_\alpha}\limsup_{t\to 0}\int_X E(f,x,t)\di\mu(x)\le \cE(f)  \le C_\alpha\limsup_{t\to 0}\int_X E(f,x,t)\di\mu(x)$$ where
 $$E(f,x,t)= t^{-\frac{(\alpha+2)}{2}} \int_{B_{\sqrt{t}}(x)}(f(x)-f(y))^2\di \mu(y).$$ 
 
Let $f\in  \Lip_o(X)\cap \cD(\cE)$. For any $R>0$, we set 
 $$\varphi_R(x):=\left(1-\frac{\dist(x,B_R(o)}{R}\right)_+$$for any $x \in X$, and $f_R := f \varphi_R$.
 By monotone convergence, we have $\lim_{R\to +\infty} \|f-f_R\|_2=0.$ We look now at $E(f_R,x,t)$ and we distinguish 3 cases:
 \begin{itemize}
 \item if $x\in B_{R-\sqrt{t}}(o)$, then $E(f_R,x,t)=E(f,x,t)$;
 \item if $x\not\in B_{2R+\sqrt{t}}(o)$, then $E(f_R,x,t)=0$;
 \item if $x\in B_{2R+\sqrt{t}}(o)\setminus  B_{R-\sqrt{t}}(o)$,
 \end{itemize}
then using $f_R(x)-f_R(y)=\varphi_R(x)(f(x)-f(y))+f(y)(\varphi_R(x)-\varphi_R(y))$ we get
 \begin{align*}
 E(f_R,x,t)&\le 2 t^{-\frac{(\alpha+2)}{2}} \left(\int_{B_{\sqrt{t}}(x)}(f(x)-f(y))^2\di \mu(y)+\int_{B_{\sqrt{t}}(x)}f^2(y)(\varphi_R(x)-\varphi_R(y))^2\di\mu(y)\right)\\
 &\le 2E(f,x,t)+ 2 t^{-\frac{(\alpha+2)}{2}} \int_{B_{\sqrt{t}}(x)}f^2(y)\frac{t}{R^2}\di\mu(y)
 \end{align*}
where we have used the fact that $\varphi_R$ is $1/R$-Lipschitz. By Fubini's theorem,
\[
\int_X \int_{B_{\sqrt{t}}(x)}f^2(y)\di \mu(y) \di \mu(x) = \int_X f^2(y)\mu(B_{\sqrt{t}}(y))\di \mu(y) = \omega_\alpha t^{\frac{\alpha}{2}} \int_X f^2(y)\di \mu(y),
\]
thus
 \begin{align*}\cE(f_R)&\le C_\alpha \limsup_{t\to 0}\int_{X} E(f,x,t)\di\mu(x)+C_\alpha \frac{2}{R^2} \omega_\alpha \int_X f^2(y)\di\mu(y)\\
 &\le C_\alpha^2 \cE(f)+C_\alpha \frac{2}{R^2} \omega_\alpha \int_X f^2(y)\di\mu(y).
 \end{align*}
Hence $\{f_R\}_{R\ge 1}$ is bounded in $\cD(\cE)$. The fact that $\lim_{R\to +\infty} \|f-f_R\|_2=0$ implies that $f_R$ converges weakly to $f$ in $\cD(\cE)$ when $R\to+\infty$. By Mazur's lemma, there exists a sequence $\{u_\ell\}_\ell \subset \Lip_c(X)$ made of convex combination of $\{f_R\}_{R\ge 1}$ such that 
 $$\lim_{\ell\to +\infty} \|f-u_\ell\|_2=0.$$
Therefore, it is enough to show that $\Lip_o(X)$ contains a subset that is dense in $\cD(\cE)$ for $|\cdot|_\cE$. Let $L^2_c$ be the set of compactly supported functions $f$ in $L^2(X,\mu)$. Then $P_t(L^2_c) \subset \Lip_o(X)$ for any fixed $t>0$. Indeed, for any $f \in L^2_c$ and $x, y \in X$,
$$
|P_tf(x) - P_tf(y)|\le \frac{1}{(4\pi t)^{\alpha/2}} \int_X \Big|e^{-\frac{\dist^2(x,z)}{4t}}-e^{-\frac{\dist^2(y,z)}{4t}}\Big| |f(z)| \di \mu(z).
$$
Setting $\phi(s)=e^{-\frac{s^2}{4t}}$ and noting that $|\phi'(s)|\le |\phi'(\sqrt{2t})|=:c_t$ for all $s>0$, we get from the mean value theorem, the triangle inequality and Hölder's inequality, that $|P_tf(x) - P_tf(y)|\le  C(t,f) \dist(x,y)$ with $C(t,f):=c_t(4 \pi t)^{-\alpha/2} \mu(\supp f)^{1/2} \| f \|_{L^2}$. Moreover,
$$
|P_tf(x)| = \left|\frac{1}{(4\pi t)^{\alpha/2}} \int_{\supp f} f(y) e^{-\frac{\dist^2(x,y)}{4t}} \di \mu(y)\right|
\le \frac{ e^{-\frac{\dist^2(o,x)}{8t}} }{ (4\pi t)^{\alpha/2} } \int_{\supp f} |f(y)| e^{\frac{\dist^2(o,y)}{4t}} \di \mu(y) \to 0
$$
when $\dist(o,x) \to +\infty$.

Let us show now that $P_t(L^2_c)$ is dense in $\cD(\cE)$ by proving that its $\langle \cdot , \cdot \rangle_{\cE}$-orthogonal complement $F$ in $\cD(\cE)$ reduces to $\{0\}$. For any $v\in F$, we have:
\begin{equation}\label{eq:18.07}
\int_X v P_t f \di \mu + \cE(v,P_t f)=0 \qquad \forall f \in L^2_c.
\end{equation}
Since $P_t$ maps $L^2(X,\mu)$ into $\cD(L)$ and is self-adjoint, then
\begin{align*}
\cE(v,P_t f) & = - \int_X v L (P_t f) \di \mu =  -\int_X v \frac{\di}{\di t} P_t f \di \mu = -\frac{\di}{\di t} \int_X v P_t f \di \mu\\
 & = - \frac{\di}{\di t} \int_X (P_tv) f \di \mu
=  -\int_X \frac{\di}{\di t}(P_tv) f \di \mu = -  \int_X L(P_tv) f \di \mu =  \cE(P_tv,f)
\end{align*}
for any $f \in L^2_c$, so \eqref{eq:18.07} becomes:
$$
\int_X (P_tv) f \di \mu + \cE(P_t v,f)=0 \qquad \forall f \in L^2_c.
$$
This implies $P_t v \in \cD(L)$ with $L (P_t v) = P_t v$. Since $L$ is a non-positive operator, $1$ cannot be an eigenvalue of $L$, so we necessarily have $P_t v = 0$. This implies $v=0$ since the spectral theorem ensures that $0$ cannot be an eigenvalue of $P_t$.
\end{proof}

By the Beurling-Deny theorem, Proposition \ref{prop:slandreg} ensures the existence of a $\Gamma$ operator for any Dirichlet form $\cE$ with an $\alpha$-dimensional Euclidean heat kernel defined on a metric measure space $(X,\dist,\mu)$. We can then define the associated pseudo-distance $\dist_\cE$ as recalled in Section 2. It turns out that in this case, $\dist_\cE$ is equivalent to the initial distance $\dist$, as shown in the next proposition.

\begin{proposition}\label{eq:propA}
Let $(X,\dist,\mu,\cE)$ be with an $\alpha$-dimensional Euclidean heat kernel. Then there exists $C_2>0$ depending only on $\alpha$ such that $C_2 \dist \le \dist_{\cE} \le \dist$. In particular, the assumption \eqref{eq:A} is satisfied.
\end{proposition}

\begin{proof}
Let us first show that $C_2\dist \le \dist_\cE$ for some $C_2>0$. Set $\Lambda:=\{f \in \Lip(X) \, : \, C_1\Lip(f)^2 \le 1\}$ where $C_1$ is as in \eqref{eq:Lip}. It follows from Corollary \ref{cor:1} that $\Lambda$ is included in the set of test functions in \eqref{eq:defdist}. Noting that $f:=C_1^{-1/2}\dist(x,\cdot)$ is in $\Lambda$ for all $x \in X$ and that $|f(x)-f(y)|=C_1^{-1/2}\dist(x,y)$ for all $x,y \in X$, with $C_2:=C_1^{-1/2}$ we get: $$\dist_\cE(x,y)\ge C_2\dist(x,y) \qquad \forall x,y \in X.$$

Let us show now that $\dist_\cE \le \dist$. To this aim, we follow the lines of \cite{Grigor'yanEd}. Let $v \in \cD_{loc}(\cE)\cap C(X)$ be bounded and such that $\Gamma(v)\le \mu$. For any $a \in \setR$, $t>0$ and $x \in X$, set $\xi_a(x,t):=a v(x) - \frac{a^2}{2}t$.

\begin{claim}
 For any $f \in L^2(X,\mu)$, the quantity
$$
I(t):=\int_X f_t^2 e^{\xi_a(\cdot,t)} \di \mu,
$$
where $f_t:=P_t f$, does not increase when $t>0$ increases.
\end{claim}

Indeed, for any $t>0$, writing $\xi_a$ for $\xi_a(\cdot,t)$ and $\xi_a'$ for $\frac{\di}{\di t} \xi_a(\cdot,t)$, we have
\begin{align*}
\frac{\di}{\di t} (f_t^2e^{\xi_a}) & = 2 f_t \left(\frac{\di}{\di t} f_t \right)e^{\xi_a} + f_t^2\xi_a' e^{\xi_a} = 2 f_t Lf_t e^{\xi_a}-\frac{a^2}{2}f_t^2 e^{\xi_a}.
\end{align*}
Since $e^{\xi_a} \le e^{a \|v\|_\infty}$, this implies
$$
\left|\frac{\di}{\di t} (f_t^2e^{\xi_a}) \right| \le e^{a \|v\|_\infty}(2|f_t||Lf_t| + a^2|f_t|^2/2) \, \in \, L^1(X,\mu),
$$
so we can differentiate under the integral sign to get
$$
I'(t) = 2 \int_X f_t Lf_t e^{\xi_a} \di \mu - \frac{a^2}{2} \int_X f_t^2 e^{\xi_a} \di \mu.
$$
The Leibniz rule \eqref{eq:Leibniz} implies
\begin{align*}
\int_X f_t Lf_t e^{\xi_a} \di \mu = & \, \, - \cE(f_t,f_te^{\xi_a}) = - \int_X \Gamma(f_t,f_t e^{\xi_a}) \di \mu \\ = &  - \int_X f_t \underbrace{\Gamma(f_t,e^{\xi_a})}_{=\Gamma(f_t,e^{\xi_a/2}e^{\xi_a/2})} \di \mu - \int_X e^{\xi_a} \Gamma(f_t) \di \mu\\
= & - 2 \int_X f_t e^{\xi_a/2} \Gamma(f_t,e^{\xi_a/2})\di \mu - \int_X e^{\xi_a} \Gamma(f_t) \di \mu
\end{align*}
and, starting from $\Gamma(f_te^{\xi_a/2})$, $$- 2 f_t e^{\xi_a/2} \Gamma(f_t,e^{\xi_a/2}) = - \Gamma(f_te^{\xi_a/2}) + f_t^2 \Gamma(e^{\xi_a/2}) + e^{\xi_a} \Gamma(f_t),$$
so that
\begin{align*}
I'(t) & = 2 \int_X f_t^2 \Gamma(e^{\xi_a/2}) - 2 \int_X \Gamma(f_t e^{\xi_a/2})\di \mu - \frac{a^2}{2}\int_X f_t^2 e^{\xi_a} \di \mu \\ & \le \int_X f_t^2 \left( 2 \Gamma(e^{\xi_a/2})-\frac{a^2}{2} e^{\xi_a} \right) \di \mu.
\end{align*}
Since $v$ is bounded, we can apply the chain rule \eqref{eq:chain} with $\eta(\xi):=e^{\xi/2}$ to get $\Gamma(e^{\xi_a/2})=(1/4)e^{\xi_a}\Gamma(\xi_a)$ and thus
$$
2 \Gamma(e^{\xi_a/2}) - \frac{a^2}{2}e^{\xi_a} = \frac{1}{2}e^{\xi_a}\Gamma(\xi_a) - \frac{a^2}{2}e^{\xi_a} = \frac{1}{2}e^{\xi_a} a^2 \Gamma(v)-\frac{a^2}{2}e^{\xi_a}\le 0,
$$
so $I'(t) \le 0$.

\hfill

From now on, assume $a \ge 0$. Apply the claim to $f=1_A$, where $A$ is some Borel subset of $X$. Then for any $t>0$ and any Borel subset $B$ of $X$,
$$
\int_B f_t^2 e^{\xi_a(\cdot,t)} \di \mu \le \int_X f_t^2 e^{\xi_a(\cdot,t)} \di \mu =I(t)\le I(0) = \int_{A} e^{a v} \di \mu,
$$
hence
$$
\int_B f_t^2 e^{\xi_a(\cdot,t)} \di \mu \le \mu(A) e^{a \sup_A v}.
$$
Moreover, since the heat kernel is Euclidean, we have
\begin{align*}
\int_B f_t^2 e^{\xi_a(\cdot,t)} \di \mu & =  \int_{B} \left(\int_A p(x,y,t) \di \mu(y)\right)^2 e^{\xi_a(x,t)} \di \mu(x)\\ & \ge \frac{e^{-\frac{\sup_{A\times B} \dist^2}{2t}}}{(4\pi t)^\alpha} \mu(B) \mu(A)^2 e^{a \inf_B v - a^2 \frac{t}{2}}
\end{align*}
thus
\[
\frac{e^{-\frac{\sup_{A\times B} \dist^2}{2t}}}{(4\pi t)^\alpha} \mu(B) \mu(A) e^{a (\inf_B v - \sup_A v) - a^2 \frac{t}{2}} \le 1.
\]

Take $x,y \in X$. With no loss of generality we can assume $v(y)-v(x)>0$. Choose $t$ such that $\sqrt{t}<\dist(x,y)/3$. Apply the previous inequality with $A=B_{\sqrt{t}}(x)$ and $B=B_{\sqrt{t}}(y)$. In this case, $\sup_{A\times B} \dist^2 = \dist^2(x,y) + 2 \sqrt{t}$. Moreover, since $v$ is continuous, we have $\inf_B v - \sup_A v = v(y)-v(x) + \eps(t)$ where $\eps(t)\to 0$ when $t \to 0^+$. Then
\[
\frac{e^{-\frac{\dist^2(x,y) + 2\sqrt{t}}{2t}}}{(4\pi)^\alpha} \omega_\alpha^2 e^{a (v(y)-v(x) + \eps(t)) - a^2 \frac{t}{2}} \le 1
\]
for any $t \in (0,\dist^2(x,y)/9)$ and any $a \ge 0$. Now for $t \in (0,\dist^2(x,y)/9)$, choose $a=a(t)=(v(y)-v(x) + \eps(t))/t$ to get
\[
\frac{e^{-\frac{\dist^2(x,y) + 2\sqrt{t}}{2t}}}{(4\pi)^\alpha} \omega_\alpha^2 e^{\frac{1}{2t}(v(y)-v(x) + \eps(t))^2} \le 1
\]
Apply the logarithm function, multiply the resulting inequality by $2t$ and then add $\dist^2(x,y)$ to get
\[
-2\sqrt{t} + 2t \ln(\omega_\alpha^2/(4\pi)^\alpha) + (v(y)-v(x) + \eps(t))^2 \le \dist^2(x,y).
\]
Letting $t$ tend to $0$ gives
\begin{equation}\label{16sept2}
(v(x)-v(y))^2 \le \dist^2(x,y).
\end{equation}

Since for any $u \in \cD_{loc}(\cE)$ and any $R>0$, the function $u_R:=\max(u,R)$ is in $\cD_{loc}(\cE)$ with $\Gamma(u_R)\le \Gamma(u)$, approximating any possibly unbounded $v \in \cD_{loc}(\cE)\cap C_c(X)$ with $\Gamma(v) \le \mu$ by $(v_R)_{R>0}$ provides \eqref{16sept2} for any $x,y \in X$ for such a $v$. This implies $\dist_\cE \le \dist$.
\end{proof}

\begin{remark}
Though we will not use it in the sequel, let us point out that Proposition \ref{eq:propA} can be upgraded into $\dist=\dist_\cE$ provided a suitable technical assumption holds: see I. in \cite[Th.~2.5]{terElstRobinsonSikora}.
\end{remark}

\subsection{Evaluation of $L$ on squared distance functions}

Let us show now that the operator associated to the Dirichlet form of a space with an $\alpha$-dimensional Euclidean heat kernel behaves on squared distance functions as the Laplacian does in $\setR^n$.

\begin{lemma}\label{lem:Ld^2}
Let $(X,\dist,\mu,\cE)$ be with an $\alpha$-dimensional Euclidean heat kernel. Then 
$L\dist^2(x,\cdot)=2\alpha$, $L\dist(x,\cdot)=(\alpha-1)/\dist(x,\cdot)$ on $X\backslash \{x\}$ and $\Gamma(\dist(x,\cdot))=1$ $\mu$-a.e.~on $X$.
\end{lemma}

\begin{proof}
Take $x\in X$. Note first that Corollary \ref{cor:1} guarantees that $\dist^2(x,\cdot), \dist(x,\cdot) \in \cD_{loc}(\cE)$. For any $t>0$, a direct computation relying on the chain rule \eqref{eq:chainrule} and starting from the equation
$$
\left( \frac{\di}{\di t} - L \right) \frac{e^{-\frac{\dist^2(x,\cdot)}{4t}}}{(4 \pi t)^{\alpha/2}} = 0
$$
(recall the remark after Definition 2.3) provides
$$
\frac{\dist^2(x,\cdot)}{4t^2} + \frac{\alpha}{2t} - \frac{1}{4t}L\dist^2(x,\cdot) - \frac{1}{(4t)^2}\Gamma(\dist^2(x,\cdot)) =0.
$$
Multiplying by $(4t)^2$ and letting $t$ tend to $0$ gives $\Gamma(\dist^2(x,\cdot))=4\dist^2(x,\cdot)$ hence $\Gamma(\dist(x,\cdot))=1$ by \eqref{eq:chain}, while multiplying by $4t$ and letting $t$ tend to $+\infty$ gives $L\dist^2(x,\cdot)=2\alpha$, from which follows $L\dist(x,\cdot)=(n-1)/\dist(x,\cdot)$ by \eqref{eq:chainrulesquare}.

\end{proof}

As a consequence of Lemma \ref{lem:Ld^2}, we can show that locally $L$-harmonic functions on spaces with an $\alpha$-dimensional Euclidean heat kernel are necessarily strongly harmonic.

\begin{lemma}\label{lem:Lstrong} Let $(X,\dist,\mu,\cE)$ be with an $\alpha$-dimensional Euclidean heat kernel. Let $\Omega \subset X$ be an open set and $h$ a locally integrable local solution of $Lh=0$ on $\Omega$. Then for any $x \in \Omega$, the function defined on $(0,\dist(x,^{c}\Omega))$ by $$ r\mapsto \fint_{B_r(x)}h \di \mu$$ is a constant. Therefore, $h$ has a continuous representative strongly harmonic in $\Omega$.
\end{lemma}
\begin{proof}
Take $x \in \Omega$ and set $R:=\dist(x,^c\Omega)$. From $Lh=0$ we get that for any $\phi \in \cD(L)$ with compact support in $\Omega$,
$$\langle h,L\phi\rangle_{L^2}=0.$$
Take $\phi \in C_c^\infty((0,R))$ and set
$u=\phi \circ \dist(x,\cdot)$ on $X$. Then $u$ belongs to $\cD(L)$ and has compact support included in $\Omega$. The chain rule \eqref{eq:chainrule} and Lemma \ref{lem:Ld^2} yields to
\[
Lu = \chi\circ \dist(x,\cdot)
\]
where we have set
$$
\chi(r) := \phi''(r) + \frac{\alpha - 1}{r} \phi'(r) = r^{1-\alpha} \left(r^{\alpha-1} \phi'\right)'
$$
for any $r \in (0,R)$.
Since $\chi \circ \dist(x,\cdot)=-\int_{\dist(x,\cdot)}^R \chi'(s)ds$, we have
\begin{align*}
-\int_0^R \chi'(s)\left(\int_{B_s(x)}h\di\mu\right)\di s & = -\int_0^R \chi'(s)\left(\int_X h1_{B_s(x)}\di\mu\right)\di s\\
& = \int_X \left( - \int_0^R \chi'(s) 1_{(\dist(x,\cdot),+\infty)}(s) \di s \right) h \di \mu\\
&  = \int_X (\chi \circ \dist(x,\cdot)) h \di \mu = \langle h,Lu\rangle_{L^2} = 0.
\end{align*}
This implies that the function 
$s\mapsto I(s):= \int_{B_s(x)}h\di \mu$ satisfies the equation
$$ \left[r^{\alpha-1}\left(r^{1-\alpha}y'\right)'\right]'=0$$
in the distributional sense on $(0,R)$. Then there exists real-valued constants $a$, $b$ and $c$ such that for any $s \in (0,R)$,
$$I(s)=as^\alpha+bs^2+c.$$
Since $h$ is locally integrable, $c=\lim_{s\to 0} I(s) = 0$. Then $s^{-\alpha}I(s)= a+b s^{2-\alpha}$ for any $s \in (0,R)$. Using test functions $\phi$ that are constantly equal to $1$ in a neighborhood of $0$, we can get
$(2-\alpha)b=0$, from which follows that $s \mapsto s^{-\alpha} I(s)$ is a constant. Since $(X,\dist,\mu)$ has an $\alpha$-dimensional volume, this provides the result.\\
\end{proof}

\subsection{Spaces of locally $L$-harmonic functions with polynomial growth}

Let us conclude with a result that shall be crucial in the next section. We recall that for any positive integer $m$, a function $h:X \to \setR$ has polynomial growth of rate $m$ if there exists $C>0$ such that $|h|\le C(1+\dist^m(o,\cdot))$ holds for some $o \in X$. The case $m=1$ corresponds to a linear growth. Note that functions with a fixed polynomial growth rate form a vector space.

\begin{proposition}\label{prop:finitedim}
Let $(X,\dist,\mu,\cE)$ be with an $\alpha$-dimensional Euclidean heat kernel. Then for any $m \in \setN\backslash \{0\}$, the space of locally $L$-harmonic functions $h:X\to\setR$ with polynomial growth of rate $m$ is finite dimensional.
\end{proposition}

\begin{proof}
Having an $\alpha$-dimensional Euclidean heat kernel, $(X,\dist,\mu,\cE)$ trivially satisfies the Gaussian estimate \eqref{eq:LiYau}. Moreover, we know from Proposition \ref{eq:propA} that it satisfies the assumption \eqref{eq:A}. Consequently, by Proposition \ref{prop:important}, $(X,\dist,\mu,\cE)$ has the doubling and Poincaré properties. Therefore, the arguments of \cite{ColdingMinicozzi} carry over.
\end{proof}

\section{Construction of the isometry}

\quad \, In this section, we construct an isometry between a given metric measure space $(X,\dist,\mu)$ equipped with a Dirichlet form $\cE$ admitting an $\alpha$-dimensional Euclidean heat kernel and an Euclidean space $\setR^l$ equipped with a distance $\dist_Q$ associated to a suitable quadratic form $Q$.

Let us recall that a quadratic form on a $\setR$-vector space $V$ is a map $Q : V \to \setR$ for which there exists a bilinear symmetric form $\beta : V \times V \to\setR$ such that $Q(u)=\beta(u,u)$ for any $u \in V$, in which case one has
$\beta(u,v)=\frac{1}{2}(Q(u+v)-Q(u)-Q(v))$ and then
$Q(u+v)=Q(u) +2\beta(u,v) + Q(v)$ for any $u,v \in V$. Moreover, when $Q$ is positive definite, setting\begin{equation}\label{eq:dist_Q}\dist_Q(u,v):=\sqrt{Q(u-v)}\end{equation} for any $u,v \in V$ defines a distance on $V$ canonically associated to $Q$. When $V=\setR^l$, Sylvester's law of inertia states that $Q$ can be transformed into $(v_1,\ldots,v_l) \mapsto \sum_i v_i^2$ via a suitable change of basis. This implies that $(\setR^l,\dist_Q)$ and $(\setR^l,\dist_e)$ are isometric, so that the construction made in this section proves Theorem \ref{th:main}.

\subsection{The quadratic form $Q$ and the coordinate function $H$}

Let us explain how to define $Q$ in our context. We first fix a base point $o \in X$ and set $$B(x,y):=\frac{1}{2}(\dist^2(o,x)+\dist^2(o,y)-\dist^2(x,y))$$for any $x,y \in X$. Note that
\begin{equation}\label{eq:D6.1}
\dist^2(x,y)= B(x,x)+B(y,y)-2B(x,y).
\end{equation}

\begin{remark}
In case $(X,\dist)=(\setR^l,\dist_{e})$ and $o$ is the origin in $\setR^l$, the law of cosines gives $B(x,y)=\scal{x}{y}$ for any $x,y \in \setR^l$, where $\scal{\cdot}{\cdot}$ is the usual Euclidean scalar product in $\setR^l$.
\end{remark}

For any $x \in X$, it follows from Lemma \ref{lem:Ld^2} and the fact that constant functions are locally $L$-harmonic that $B(x,\cdot)$ is locally $L$-harmonic. Moreover, for any $x,y \in X$, since $\dist^2(o,y)-\dist^2(x,y) = (\dist(o,y)-\dist(x,y))(\dist(o,y)+\dist(x,y))$, $\dist(o,y)-\dist(x,y)\le \dist(o,x)$ and $\dist(x,y)\le \dist(x,o)+\dist(o,y)$, we have
\begin{align*}
B(x,y) & \le \frac{1}{2}(\dist^2(o,x)+\dist(o,x)(\dist(o,x)+2\dist(o,y)))\\
& = \dist^2(o,x)+\dist(o,x)\dist(o,y) \le C_x(1+\dist(o,y))
\end{align*}
with $C_x:=\max(\dist^2(o,x),\dist(o,x))$. This shows that $B(x,\cdot)$ has linear growth for any $x \in X$. Then $\cV:=\Span\{B(x,\cdot): x \in X\}$ is a subspace of the space of locally $L$-harmonic functions with linear growth. Using Proposition \ref{prop:finitedim}, we know that this space has a finite dimension, so $\cV$ has a finite dimension which we denote by $l$.

Let us then consider the subspace $\mathcal{D}:=\Span\{\delta_x : x \in X\}$ of the algebraic dual $\cV^*$ of $\cV$. If $f \in \cV$ is such that $\theta(f)=0$ for any $\theta \in \cD$, then $f=0$; since the duality pairing $\cV\times\cV^* \to\setR$ is non-degenerate, this implies $\cD = \cV^*$. Therefore, there exist $x_1, \cdots, x_l \in X$ such that $\{\delta_{x_1},\cdots,\delta_{x_l}\}$ is a basis of $\mathcal{V}^*$. Let $\{h_1,\cdots,h_l\}$ be the associated basis of $\mathcal{V}$. Then for any $x \in X$,
\begin{equation}\label{eq:basis}
B(x,\cdot) = \sum_{i=1}^l \delta_{x_i}(B(x,\cdot)) h_i = \sum_{i=1}^l B(x,x_i) h_i
\end{equation}
and for any $i \in \{1,\cdots,l\}$,
$$
B(x,x_i) = B(x_i,x) = \sum_{j=1}^l \delta_{x_j}(B(x_i,\cdot)) h_j(x)= \sum_{j=1}^l B(x_i,x_j)h_j(x).
$$
Therefore, we have
\begin{equation}\label{eq:D6.2}
B(x,y) = \sum_{i,j=1}^l B(x_i,x_j)h_j(x)h_i(y)
\end{equation}
for any $x,y \in X$. We now define $Q$ on $\setR^l$ by setting:
$$
Q(\xi) := \sum_{i,j=1}^lB(x_i,x_j)\xi_i \xi_j \qquad \forall \, \xi=(\xi_1,\cdots,\xi_l) \in \setR^l.
$$
Then $Q$ is a quadratic form whose  associated symmetric form $\beta$ is given by
$$\beta(\xi,\xi') = \sum_{i,j=1}^l B(x_i,x_j)\xi_i \xi_j'
$$
for any $\xi, \xi' \in \setR^l$. Note that $\beta$ is non-degenerate. Indeed, if $\xi \in \setR^l $ is such that $\beta(\xi,\cdot)=0$, then for any $y \in X$ we have $\sum_{i=1}^l \xi_i B(x_i,y)=0$ because
$$\sum_{i=1}^l \xi_i B(x_i,y)=\sum_{i,j=1}^l \xi_i h_j(y)B(x_i,x_j)=\beta(\xi,(h_1(y),\ldots,h_l(y)))=0.$$
But $\{B(x_i,\cdot)\}_i$ is a basis of $\cV$ since
$$
B(x,y)=\sum_{i=1}^l \delta_{x_i}(B(\cdot,y))h_i(x)= \sum_{i=1}^l B(x_i,y)h_i(x)
$$
for any $x,y \in X$, thus $\xi=0_l$, where $0_l$ denotes the origin in $\setR^l$.


We are now in a position to introduce our ``coordinate'' function $H:X\to\setR^l$ which we define as
$$
H:=(h_1,\cdots,h_l).
$$
This function $H$ is continuous because $h_1,\cdots,h_l$ are so. Moreover, for any $x,y \in X$, we have
\begin{equation}\label{eq:D6.3}
\beta(H(x),H(y)) = B(x,y)
\end{equation}
thanks to \eqref{eq:D6.2} and
\begin{equation}\label{eq:D6.4}
\dist^2(x,y) = Q(H(x)-H(y))
\end{equation}
thanks to \eqref{eq:D6.1}. Note that $H(o)=0_l$ because $B(x,o)=0$ for any $x \in X$. Moreover, \eqref{eq:D6.4}, the continuity of $H$ and the completeness of $(X,\dist)$ imply that $H(X)$ is a closed set of $\setR^l$.

\begin{claim}\label{claim}
$H$ is an injective map. Moreover, $\Span(H(X)) = \setR^l$ -- in fact, the closed convex hull of $H(X)$ is $\setR^l$.
\end{claim}

\begin{proof}
If $H(x)=H(y)$ then \eqref{eq:D6.4} gives $\dist^2(x,y)=Q(0)=0$, so $x=y$. Then $H$ is injective. For the second statement, let us recall that the closed convex hull $\overline{\mathrm{conv}}(E)$ of a closed set $E \subset \setR^l$ is defined as the smallest convex subset of $\setR^l$ containing $E$; moreover, $\overline{\mathrm{conv}}(E)$ coincides with
$$
\bigcap_{\lambda \in \mathcal{A}(E)} \{\lambda \ge 0\}
$$
where $\mathcal{A}(E)$ is the set of affine functions on $\setR^l$ that are non-negative on $E$.  Note that being closed and convex, $\Span(E)$  contains $\overline{\mathrm{conv}}(E)$. Take $\lambda \in \mathcal{A}(H(X))$. Then $\lambda \circ H : X \to \setR$ is an affine combination of locally $L$-harmonic functions, hence it is a locally $L$-harmonic function too. Since $\lambda \circ H$ is non-negative on $X$, Lemma \ref{lem:ellipticHarnack} implies that it is a constant. Therefore, $\overline{\mathrm{conv}}(H(X))=\setR^l$, what brings the result.
\end{proof}

Note that the right-hand side in \eqref{eq:D6.4} does not define any squared distance on $\setR^l$ unless $Q$ is shown to be positive definite, see Subsection 4.3.

\subsection{Conical structure of tangent cones at infinity}


Let $(\uX,\udist,\uo)$ be a tangent cone at infinity of $(X,\dist)$ at $o$. We denote by $\{(X_i,\dist_i:=r_i^{-1}\dist,o)\}_i$, where $\{r_i\}_i \subset (0,+\infty)$ converges to $+\infty$, the sequence of rescalings of $(X,\dist,o)$ converging in the pointed Gromov-Hausdorff topology to $(\uX,\udist,\uo)$. Note that whenever $\ux \in \uX$, there exists a sequence $\{x_i\}_i\subset X$ such that $x_i \stackrel{GH}{\to} \ux$; in particular, $\dist_i(o,x_i) \to \udist(\uo,\ux)$ and $\dist(o,x_i)\to +\infty$.\\

\textbf{\underline{Step 1.}} [Construction of a Busemann function $h_\infty$ associated with a divergent sequence]

Let $\{x_i\}_i \subset X$ be a sequence such that $\dist(o,x_i) \to +\infty$. For any $i$, setting $$D_i:=\dist(o,x_i),$$ we define $$h_i(y):=D_i-\dist(x_i,y)$$ 
for any $y \in X$ and call $c_i:[0,D_i] \to X$ a minimizing geodesic joining $o$ to $x_i$.

On one hand, the triangle inequality implies that the functions $h_i$ are all $1$-Lipschitz, so by the Ascoli--Arzelà theorem, up to extracting a subsequence, we can assume that the sequence $\{h_i\}_i$ converges uniformly on compact subsets of $X$ to a $1$-Lipschitz function $h_\infty$. On the other hand, the minimizing geodesics $c_i$ being $1$-Lipschitz too, we can use again the Ascoli-Arzelà theorem to assume, up to extraction, that they converge uniformly on compact sets of $[0,+\infty)$ to a geodesic ray $\gamma$.

\begin{claim}\label{claim2}
The function $h_\infty$ constructed as above coincides with the Busemann function $b_\gamma$ associated with $\gamma$.
\end{claim}

\begin{proof}
Thanks to Lemma \ref{lem:Ld^2} and the fact that constant functions are locally $L$-harmonic, we know that for any $i$ we have $h_i \in \cD_{loc}(\cE)$ and
$$
Lh_i = \frac{\alpha-1}{\dist(x_i,\cdot)} \qquad \text{on $X \backslash \{x_i\}$}.
$$
Therefore, for any $R\in(0,\dist(o,x_i))$, since $\dist(x_i,y)\ge D_i-\dist(y,o)\ge D_i-R$ holds for any $y \in B_R(o)$, we get
\begin{equation}\label{eq:D5.1}
|Lh_i| \le \frac{\alpha-1}{D_i-R} \qquad \text{on $B_R(o)$.}
\end{equation}
Then 
\begin{equation}\label{eq:wesh}
|\cE(h_i,\phi)| = |\langle \phi, L h_i\rangle_{L^2}| \le \frac{\alpha-1}{D_i-R}\|\phi\|_1 \to 0 \qquad \text{when $i \to +\infty$}
\end{equation}
 for any $\phi \in \cD_c(\cE)$. Since $h_i \to h_\infty$ uniformly on compact sets, then $h_i \to h_\infty$ in $L^2_{loc}(X,\mu)$. As in the proof of Lemma \ref{lem:preparatory}, this implies $h_\infty \in \cD_{loc}(\cE)$ with $Lh_\infty=0$. Moreover, $Lb_\gamma=0$ by Lemma \ref{lem:preparatory}, so $h_\infty - b_\gamma$ is a locally $L$-harmonic function. Let us show that it is non-negative.
For any $i$ and $s \in [0,D_i]$, set: $$h_{i,s}(y)=s-\dist(y,c_i(s)) \qquad \forall y \in X.$$
Since $\dist(x_i,y) \le \dist(x_i,c_i(s))+\dist(y,c_i(s))=D_i -s + \dist(y,c_i(s))$ for all $y \in X$, we have
$$
h_{i,s} \le h_i.
$$
As the curves $c_i$ pointwise converge to $\gamma$, the functions $h_{i,s}$ pointwise converge to $g_s:y \mapsto s - \dist(y,\gamma(s))$, so that letting $i$ tend to $+\infty$ provides
$$
g_s \le h_\infty
$$
and then letting $s$ tend to $+\infty$ gives
$$
b_\gamma \le h_{\infty}.
$$
Then $h_\infty - b_\gamma$ is a  non-negative locally $L$-harmonic function on $X$ hence it is a constant because of Lemma \ref{lem:ellipticHarnack}. But $b_\gamma(o)=0=h_i(o)$ for any $i$, so this constant is equal to $0$.
\end{proof}


\hfill

\textbf{\underline{Step 2.}} [Behavior of $H$ in the convergence $(X,\dist_i) \to (\uX,\udist)$ and link with $h_\infty$]

Recall that for any $1 \le j \le n$, the function $h_j$ has linear growth: $|h_j(x)| \le C_j (1+\dist(o,x))$ for any $x \in X$, where $C_j>0$ is some constant. Then the rescalings $h_j^{i}:=r_i^{-1}h_j$ are such that $|h_j^{i}(x)| \le C_j(r_i^{-1}+ \dist_i(o,x))$ for any $x \in X$, hence
$$
\|h_j^{i}\|_{L^\infty(B_r^{\dist_i}(o))} \le C_j(r_i^{-1}+r)\le C_j(1+r)
$$
holds for any $r>0$ and any $i$ such that $r_i>1$. Moreover, since $h_j$ is locally $L$-harmonic, it is strongly harmonic by Lemma \ref{lem:Lstrong} and then Lipschitz by Lemma \ref{lem:0109}: there is some constant $C_j'>0$ such that 
$$
|h_j(x)-h_j(y)|\le C_j' \dist(x,y)
$$
for any $x,y \in X$. This implies that the sequence $\{h_j^{i}\}_i$ is asymptotically uniformly continuous on $\overline{B}_r(x)$. It is immediate to check that the rescalings $h_j^{i}$ are also strongly harmonic in $(X,\dist_i,\mu_i)$ where $\mu_i:=r_i^{-\alpha}\mu$.  Then Proposition \ref{prop:AA1} and Proposition \ref{prop:stabstrongharmonic} imply that up to extracting a subsequence, we can assume that for any $j=1,\dots, l$, the functions $h_j^{i}$ converge uniformly on all compact sets to a strongly harmonic function $\uh_j : \uX \to \setR$. We set
$$
\uH := (\uh_1,\ldots,\uh_l) : \uX \to \setR^l.
$$

\begin{claim}
For any given $\ux \in \uX$, the function $X \ni y \mapsto \beta(\uH(\ux),H(y))$ is a multiple of a Busemann function.
\end{claim}

\begin{proof}
Let $\{x_i\}_i \subset X$ be such that $x_i \stackrel{GH}{\to} \ux$. Denote by $h_\infty$ the Busemann function associated to $\{x_i\}_i$ as in the previous step. Then for any $y \in X$,
\begin{align*}
\beta(\uH(\ux),H(y)) & = \lim\limits_{i \to +\infty} \beta(H_i(x_i),H(y))\\ & = \lim\limits_{i \to +\infty} \frac{1}{r_i} \beta(H(x_i),H(y))\\
& = \lim\limits_{i \to +\infty} \frac{1}{2r_i} (\dist^2(o,x_i)+\dist^2(o,y)-\dist^2(x_i,y)) \qquad \text{by \eqref{eq:D6.3}}\\
& = \lim\limits_{i \to +\infty} \frac{(\dist(o,x_i)-\dist(x_i,y))(\dist(o,x_i)+\dist(x_i,y))}{2r_i}\\
& = h_\infty(y) \left(\frac{\udist(\uo,\ux)}{2} + \lim\limits_{i \to +\infty} \frac{\dist(x_i,y)}{2r_i} \right)
\end{align*}
since $\dist(o,x_i)-\dist(x_i,y) \to h_\infty(y)$. Now
$$
\underbrace{\frac{\dist(x_i,o)-\dist(o,y)}{r_i}}_{\to \udist(\ux,\uo)} \le \frac{\dist(x_i,y)}{r_i} \le \underbrace{\frac{\dist(x_i,o)+\dist(o,y)}{r_i}}_{\to \udist(\ux,\uo)}\, ,
$$
so
\begin{equation}
\beta(\uH(\ux),H(y)) = \udist(\uo,\ux) h_\infty(y).
\end{equation}
\end{proof}

Note that we also have the following.

\begin{claim}\label{uclaim}
$\uH$ is an injective map. Moreover, $\Span(\uH(\uX)) = \setR^l$.
\end{claim}

\begin{proof}
Dividing $\eqref{eq:D6.4}$ by $r_i^2$ and taking $i\to+\infty$ implies $\udist^2(x,y) = Q(\uH(x)-\uH(y))$ for any $x,y \in X$, hence the injectivity of $\uH$. To prove the second part of the statement, let us show that $\uH(\uX)$ is contained in no hyperplan of $\setR^l$. Take a linear form $\lambda :\setR^l\to\setR$ vanishing on $\uH(\uX)$. Considering the convergent sequence $(X_i,\dist_i,o_i)\to(\uX,\udist,\uo)$, Proposition \ref{prop:AA1} implies that up to extraction the equi-Lipschitz functions $r_i^{-1} \lambda \circ H : X \to \setR$ converge to $0=\lambda \circ \uH : \uX \to \setR$ over $\overline{B}^{\dist_i}_r(o) \to\overline{B}_r(o)$ for any $r>0$. Therefore, we have $$\sup_{\partial B_{r_i}(o)} |\lambda \circ H| = \sup_{\partial B_1^{\dist_i}(o)} |\lambda \circ H|= o(r_i).$$ Being a linear combination of locally $L$-harmonic functions, $\lambda \circ H$ is locally $L$-harmonic hence strongly harmonic by Lemma \ref{lem:Lstrong}. Thus Lemma \ref{lem:0109} implies that $\lambda \circ H$ is constantly equal to $0$. Since $\Span(H(X))=\setR^l$, this implies $\lambda = 0$.
\end{proof}

\hfill

\textbf{\underline{Step 3.}} [Construction of the bijection]

Our goal now is to construct a natural bijection between $\uX \backslash \{\uo\}$ and $\uS \times (0,+\infty)$, where $\uS:=\{\ux \in \uX : \udist(\uo,\ux)=1\}$.\\

Let us start with some heuristics. For $\ux \in \uX \backslash \{\uo\}$ given, we look for $\usigma \in \uS$ and $t\in(0,+\infty)$ uniquely determined by $\ux$. Here is how we are going to proceed:

1. prove that there exists only one minimizing geodesic $\uc$ joining $\uo$ to $\ux$;

2. show that $\uc$ extends in an unique way into a geodesic ray $\ugamma$.

\noindent Indeed, the unique geodesic ray $\ugamma$ that we are going to construct necessarily crosses $\uS$ at a single point $\usigma$ (otherwise $\ugamma$ would fail to be a minimizing geodesic) and it is such that $\ugamma(t)=\ux$ for a unique time $t>0$. Conversely, a pair $(\usigma,t)$ would uniquely determine a point $\ugamma(t) \in \uX$.\\

Let us proceed now with the construction. Take $\ux \in \uX \backslash \{\uo\}$. Let $\{x_i\}_i \subset X$ be such that $x_i \stackrel{GH}{\to} \ux$. For any $i$, let $c_i$ be the minimizing $\dist_i$-geodesic joining $o$ to $x_i$. As done previously, up to extracting a subsequence we can assume that $\{c_i\}_i$ converges uniformly on compact subsets of $[0,+\infty)$ to a geodesic ray $\gamma : [0,+\infty) \to X$. We know from Claim \ref{claim2} and the previous step that
\begin{equation}\label{eq:asabove}
\beta(\underline{H}(\ux),H(y)) = \udist(\uo,\ux) b_\gamma(y)
\end{equation}
holds for any $y \in X$, where $b_\gamma$ is the Busemann function associated with $\gamma$. Consider now a minimizing geodesic $\uc$ in $(\uX,\udist)$ joining $\uo$ to $\ux$ and set $$\uD:=\udist(\uo,\ux).$$ For any $s \in [0,\uD]$, acting as we did to establish \eqref{eq:asabove}, we can prove that $$\beta(\underline{H}(\uc(s)),H(y)) = s b_\gamma(y)$$ holds for any $y \in X$. Subtracting \eqref{eq:asabove} to this latter equality yields to
\begin{equation}\label{eq:draft23}
\beta\left( \underline{H}(\uc(s)) - \frac{s}{\uD}\uH(\ux) , H(y)\right) = 0
\end{equation}
for any $y \in X$. By Claim \ref{claim}, this implies 
\begin{equation}\label{eq:draft23}
\beta\left( \underline{H}(\uc(s)) - \frac{s}{\uD}\uH(\ux) , \xi \right) = 0
\end{equation}
for any $\xi \in \setR^n$ and then
\begin{equation}\label{eq:H}
\underline{H}(\uc(s)) = \frac{s}{\uD}\uH(\ux)
\end{equation}
since $\beta$ is non-degenerate. Uniqueness of $\uc$ follows: if $\uc_1$ and $\uc_2$ are two minimizing geodesics joining $\uo$ to $\ux$, for any $s \in [0,\uD]$ one has
$$
\uH(\uc_1(s)) = \frac{s}{\uD}\uH(\ux) =\uH(\uc_2(s))
$$
and thus $\uc_1(s)=\uc_2(s)$ since $\uH$ is injective.\\

Let us show now that $\uc$ extends in an unique way into a geodesic ray. Our argument is inspired by the analysis done by J. Cheeger about generalized linear functions \cite[Section 8]{CheegerRademacher}. For any $i$, set $D_i:=\dist(o,x_i)$ and write $\gamma_i:[0,+\infty) \to X$ for the geodesic ray in $(X,\dist_i)$ defined by:
$$
\gamma_i(s) = \gamma(sr_i + D_i) \qquad \forall s>0.
$$
On one hand, by Proposition \ref{prop:AA2}, we know that up to extracting a subsequence we can assume that $\{\gamma_i\}_i$ converges uniformly on compact subsets of $[0,+\infty)$ to a geodesic ray $\tilde{\gamma}:[0,+\infty] \to \uX$ whose associated Busemann function we denote by $b_{\tilde{\gamma}}.$ On the other hand, if we write $b_{\gamma_i}$ for the Busemann function associated with $\gamma_i$, we can proceed as in Step 2 with $\dist_{i}, H_{i},\gamma_{i}$ in place of $\dist, H, \gamma$ respectively to get
\begin{equation}
b_{\gamma_{i}}(y) = \frac{\beta(\uH(\ux),H_{i}(y))}{\uD}
\end{equation}
for any $y \in X$. Then the sequence $(b_{\gamma_i})$ converges pointwise to the function
$$
F:\uX \ni \uy \mapsto \frac{\beta(\uH(\ux),\uH(\uy))}{\uD}
$$
and we have the following:
\begin{claim}
\begin{equation}\label{eq:uniqueness}
F=\underline{b}_{\tilde{\gamma}}+\uD.
\end{equation}
\end{claim}

\begin{proof}
Observe first that $F$ is strongly harmonic since it is a linear combination of the strongly harmonic functions $\uh_1,\ldots,\uh_l$. Let us show that $b_{\tilde{\gamma}}$ is strongly harmonic too. For any $i$, set $$
p_i(x,y,t):=\frac{1}{(4\pi t)^{\alpha/2}} e^{-\frac{\dist_i^2(x,y)}{4t}}= r_i^\alpha p(x,y,r_i^2 t)$$
for any $x,y \in X$ and $t>0$, and
$$
\underline{p}(\ux,\uy,t) := \frac{1}{(4 \pi t)^{\alpha/2}} e^{-\frac{\udist^2(\ux,\uy)}{4t}}
$$
for any $\ux,\uy \in X$ and $t>0$. Then for any $x,y \in X$ and $s,t>0$,
\begin{align*}
p_i(x,y,t+s) & = r_i^\alpha p(x,y,r_i^2 t + r_i^2 s) = r_i^\alpha \int_X p(x,z,r_i^2 t)p(z,y,r_i^2 s) \di \mu(z)\\
& = \int_X p_i(x,z,t)p(z,y,r_i^2s) r_i^\alpha \frac{\di \mu(z)}{r_i^\alpha} = \int_X p_i(x,z,t)p_i(z,y,s) \di \mu_i(z).
\end{align*}
For any $\ux, \uy \in \uX$ and $\{x_i\}_i, \{y_i\}_i \in X$ such that $x_i \stackrel{GH}{\to} \ux$ and $y_i \stackrel{GH}{\to} \uy$, the convergence $\dist_i(x_i,y_i) \to \udist(\ux,\uy)$ implies $p_i(x_i,y_i,t) \to \up(\ux,\uy,t)$ for any $t>0$, hence:
$$
\underline{p}(\ux,\uy,t+s) = \int_X \underline{p}(\ux,\uz,t) \underline{p}(\uz,\uy,s) \di \umu(\uz) \qquad \forall \ux, \uy \in \uX, \, \forall \, t,s>0.
$$
By a standard procedure described for instance in \cite[Section 2]{Grigor'yan}, we can construct a Dirichlet form $\underline{\cE}$ on $(\uX,\udist,\umu)$ admitting a heat kernel given by $\underline{p}$. In particular, $(\uX,\udist,\umu,\underline{\cE})$ has an $\alpha$-dimensional Euclidean heat kernel. Writing $\underline{L}$ for the associated self-adjoint operator, we deduce from Lemma \ref{lem:preparatory} that $b_{\tilde{\gamma}}$ is locally $\underline{L}$-harmonic, then Lemma \ref{lem:Lstrong} implies that $b_{\tilde{\gamma}}$ is strongly harmonic.

Let us show now that $F - \underline{b}_{\tilde{\gamma}} \ge 0$.  Take $i\in \setN$, $s>0$ and $y \in X$. Then $b_\gamma(y) \ge r_i s + D_i - \dist(\gamma(r_i s+D_i),y)$ by definition of a Busemann function, hence $r_i^{-1}b_\gamma(y) \ge s+ r_i^{-1} D_i  - \dist_i(\gamma_i(s),y)$. Since
\begin{align*}
r_i^{-1}b_\gamma(y) & = \lim\limits_{s \to +\infty} r_i^{-1} s - r_i^{-1} \dist(\gamma(s),y)\\
& = D_i r_i^{-1} + \lim\limits_{s' \to +\infty} s' - \dist_i(\gamma(D_i+r_i s'),y)\qquad \text{where $s'=(s-D_i)r_i^{-1}$}\\ & = D_i r_i^{-1} + b_{\gamma_i}(y),
\end{align*}
we get $b_{\gamma_i}(y) \ge s - \dist_i(\gamma_i(s),y)$. Letting $i$ tend to $+\infty$ provides $F(y)\ge s - \udist(\tilde{\gamma}(s),y)$, after what letting $s$ tend to $+\infty$ gives $F\ge b_{\tilde{\gamma}}$. 

By Lemma \ref{lem:ellipticHarnack}, we get that $F-b_{\tilde{\gamma}}$ is a constant function. Since $F(\ux)=\uD=\udist(\uo,\ux) $ and $b_{\tilde{\gamma}}(\ux)=0$, the claim is proved.
\end{proof}

Let $\ugamma:[0,+\infty)\to \uX$ be the concatenation of $\uc$ and $\tilde{\gamma}$, i.e.
$$
\ugamma(t):=
\begin{cases}
\uc(t) & \text{if $0<t\le \uD$,}\\
\tilde{\gamma}(t-\uD) & \text{if $t\ge\uD$.}
\end{cases}
$$
By construction, $\ugamma$ is $1$-Lipschitz: $\udist(\ugamma(t),\ugamma(s))\le |t-s|$ for any $s,t>0$. Moreover \eqref{eq:H} implies $F(\ugamma(t))=t$ when $0<t\le \uD$ while \eqref{eq:uniqueness} implies $F(\ugamma(t))=t$ when $t\ge \uD$.
Since the function $F$ is $1$-Lipschitz we get $|t-s| \le \udist(\ugamma(t),\ugamma(s))$ for any $s,t>0$, thus $\ugamma$ is a geodesic ray that extends $\uc$.

Let us show that this extension $\ugamma$ is unique. By \eqref{eq:uniqueness}, we have
\begin{equation}\label{eq:un}
\beta(\uH(\ux),\uH(\uy)) = \uD( \ub_{\tilde{\gamma}}(\uy)-\uD)
\end{equation}
for any $\uy \in \uX$ and we can obtain
\begin{equation}
\beta(\uH(\tilde{\gamma}(t)),\uH(\uy)) = t( \ub_{\tilde{\gamma}}(\uy)-\uD)
\end{equation}
for any $\uy \in \uX$ and $t>\uD$ by a similar reasoning. Then if $\ugamma'$ is another extension of $c$ obtained from a geodesic ray $\tilde{\gamma}'$ emanating from $\ux$, we get
$$
\beta(\uH(\tilde{\gamma}(t))-\uH(\tilde{\gamma}'(t)),\uH(\uy)) =0
$$
for any $\uy \in \uX$ and $t>\uD$, from which Claim $\ref{uclaim}$ yields $\tilde{\gamma}(t)=\tilde{\gamma}'(t)$.

\begin{remark}
Note that \eqref{eq:un} implies $\beta(\uH(\tilde{\gamma}(t)),\uH(\uy)) = t\beta(\uH(\tilde{\gamma}(1)),\uH(\uy))$ for all $\uy \in \uX$, hence
\begin{equation}\label{eq:conicalstructure}
\uH(\tilde{\gamma}(t)))= t\uH( \tilde{\gamma}(1)).
\end{equation}
\end{remark}
\hfill

\textbf{\underline{Step 4.}} [Construction of the isometry]

Let $\Phi$ be the inverse of the bijection constructed in the previous step, i.e.~$$\Phi : \begin{array}{ccr}
(0,+\infty) \times \uS &  \to & \uX \backslash\{\uo\} \\
(t,\usigma) & \mapsto & \ugamma_{\usigma}(t),
\end{array}
$$
where $\ugamma_{\usigma}$ is the geodesic ray obtained by extending the minimizing geodesic joining $\uo$ to $\usigma$. Note that \eqref{eq:conicalstructure} implies
\begin{equation}\label{eq:conicalstructure2}
\uH(\Phi(t,\usigma))= t\uH( \Phi(1,\usigma))
\end{equation}
for any $(t,\usigma)\in (0,+\infty) \times \uS$. Let $\dist_C$ be the cone distance on $(0,+\infty)\times \uS$ defined by
$$
\dist_C^2((t,\usigma),(t',\usigma')):=(t-t')^2 + 2 t t' \sin^2\left(\frac{\udist_{\uS}(\usigma,\usigma')}{2}\right) 
$$
for any $(t,\usigma),(t',\usigma') \in (0,+\infty) \times \uS$, where $\udist_{\uS}$ is the length distance associated with the distance on $\uS$ obtained by restricting $\udist$ to $\uS \times \uS$. We are going to establish
\begin{equation}\label{eq:isometry}
\udist(\ux,\ux') = \udist_C((t,\usigma),(t',\usigma'))
\end{equation}
for any $\ux = \Phi(t,\usigma), \, \ux' = \Phi(t',\usigma') \in \uX \backslash \{\uo\}$.

\begin{claim}\label{claim:isometrywithdelta}
There exists $\delta(\usigma,\usigma') \in [0,\pi]$ such that
\begin{equation}\label{eq:coeur2}
\udist^2(\ux,\ux') = (t-t')^2 + 4 tt' \sin^2\left(\frac{\delta(\usigma,\usigma')}{2}\right)\,\cdot
\end{equation}
\end{claim}

\begin{proof}
Choose $\{x_i\}_i, \{x_i'\}_i \subset X$ such that $x_i \stackrel{GH}{\to} \ux$ and $x_i' \stackrel{GH}{\to} \ux'$. For any $i$, divide \eqref{eq:D6.4} by $r_i^2$ to get $
\udist_i^2(x_i,x_i') = Q(H_i(x_i)-H_i(x_i')).$
Letting $i$ tend to $+\infty$ implies $
\udist^2(\ux,\ux') = Q(\uH(\ux)-\uH(\ux'))$, hence
$$
\udist^2(\ux,\ux')
= Q(t \uH(\Phi(1,\usigma)) - t' \uH(\Phi(1,\usigma')))
$$
thanks to \eqref{eq:conicalstructure2}. To compute $Q(t \uH(\Phi(1,\usigma)) - t' \uH(\Phi(1,\usigma')))$, let us use $\underline{h}_i(\usigma)$ as a shorthand for $\underline{h}_i(\Phi(1,\usigma))$. Then:
\begin{align}
\udist^2(\ux,\ux') & = Q(t\uh_1(\usigma) - t' \uh_1(\usigma'), \ldots, t\uh_l(\usigma) - t' \uh_l(\usigma')) \nonumber \\
& = \sum_{i,j=1}^l B(x_i,x_j) (t\underline{h}_i(\usigma)-t'\underline{h}_j(\usigma'))(t \underline{h}_j(\usigma)-t'\underline{h}_j(\usigma'))\nonumber \\
& = \left(\sum_{i,j=1}^l B(x_i,x_j) \uh_i(\usigma)\uh_j(\usigma)\right) t^2 +  \left(\sum_{i,j=1}^l B(x_i,x_j) \uh_i(\usigma')\uh_j(\usigma')\right) (t')^2\nonumber\\
&  - 2 t t' \left(\sum_{i,j=1}^l B(x_i,x_j) \uh_i(\usigma')\uh_j(\usigma)\right)\nonumber\\
& = Q(\uH(\usigma)) t^2 + Q(\uH(\usigma')) (t')^2 - 2 t t' \beta(\uH(\usigma),\uH(\usigma'))\nonumber\\
& = \udist^2(\usigma,\uo)t^2 + \udist^2(\usigma',\uo) (t')^2 - 2 t t' \beta(\uH(\usigma),\uH(\usigma'))\nonumber\\
& = t^2 + (t')^2 - 2 t t' \beta(\uH(\usigma),\uH(\usigma').
\end{align}
Set $\uB(\ux,\ux'):=\frac{1}{2}(\udist^2(\uo,\ux)+\udist^2(\uo,\ux')-\udist^2(\ux,\ux'))$. Write \eqref{eq:D6.3} with $x=x_i$, $x'=x_i'$, divide by $r_i^2$ and let $r_i$ tend to $+\infty$ to get $\ubeta(\uH(\ux),\uH(\ux'))=\uB(\ux,\ux')$ and then
\begin{equation}\label{eq:coeur}
\udist^2(\ux,\ux') = t^2 + (t')^2 - 2 t t' \uB(\usigma,\usigma').
\end{equation}
Assuming $t=t'=1$ provides $\uB(\usigma,\usigma')=1-\frac{1}{2}\udist^2(\usigma,\usigma')$. The triangle inequality implies $\udist^2(\usigma,\usigma')\le 4$, thus $\uB(\usigma,\usigma') \in [-1,1]$, so we can set
$$
 \cos(\delta(\usigma,\usigma')):=\uB(\usigma,\usigma')
$$
for some $\delta(\usigma,\usigma') \in [0,\pi]$.
\end{proof}

In particular, \eqref{eq:coeur2} implies
$\udist(\usigma_0,\usigma_1)=2\sin(\delta(\usigma_0,\usigma_1)/2)$
for any $\usigma_0,\usigma_1 \in \uS$.

\begin{claim}\label{claim:geodesicdelta}
The function $\delta$ defines a geodesic distance on $\uS$.
\end{claim}

\begin{proof}
Let us first show that $\delta$ defines a distance on $\uS$. We only prove the triangle inequality since the two other properties are immediate. For given $\usigma_0, \usigma_1,\usigma_2 \in \uS$, let us set $\alpha := \delta(\usigma_0,\usigma_1)$, $\beta:=\delta(\usigma_1,\usigma_2)$ and $\gamma:=\delta(\usigma_0,\usigma_2)$. We can assume $\alpha + \beta \le \pi$ because otherwise we would have $\alpha + \beta > \pi \ge \gamma$, thus nothing to prove. For any $t,s,r>0$, the triangle inequality for $\udist$ written with \eqref{eq:coeur2} gives
$$
\sqrt{(t-s)^2 + 4 ts\sin^2(\alpha/2)} + \sqrt{(s-r)^2 + 4 sr\sin^2(\beta/2)} \ge \sqrt{(t-r)^2 + 4tr\sin^2(\gamma/2)}.
$$
Considering the three complex numbers $z_0=t$, $z_1 = se^{i\alpha}$ and $z_2 = re^{i(\alpha + \beta)}$, this can be rewritten as
$$
|z_0-z_1| + |z_1-z_2| \ge \sqrt{(t-r)^2 + 4tr\sin^2(\gamma/2)}.
$$
Choosing $s$ so that $z_0, z_1, z_2$ are aligned implies $|z_0-z_2| = |z_0-z_1| + |z_1 - z_2|$ thus
$$
\sqrt{(t-r)^2 + 4tr\sin^2((\alpha+\beta)/2)} = |z_0-z_2| \ge \sqrt{(t-r)^2+ 4tr\sin^2(\gamma/2)}
$$
which yields to $\alpha + \beta \ge \gamma$.

Let us show now that $\delta$ is geodesic. For given $\usigma_0, \usigma_1 \in \uS$ with $\usigma_0 \neq \usigma_1$, we aim at finding $\usigma_m \in \uS$ such that
$$\delta(\usigma_0,\usigma_m)=\delta(\usigma_m,\usigma_1)=\frac12\delta(\usigma_0,\usigma_1).$$
Let $c:[0,\udist(\usigma_0,\usigma_1)]\to X$ be the minimizing $\udist$-geodesic between $\usigma_0$ and $\usigma_1$. Assume first $\delta(\usigma_0,\usigma_1)<\pi$ so that $c(\udist(\usigma_0,\usigma_1)/2)\neq 0$. Then $c(\udist(\usigma_0,\usigma_1)/2)$ writes as $\Phi(s,\usigma_m)$ for some $(s,\sigma_m) \in (0,1) \times \uS$. We have
$$
\udist(\usigma_0,\usigma_m) = \udist(\usigma_m,\usigma_1) = \frac{1}{2}\udist(\usigma_0,\usigma_1) 
$$
from which follows
\begin{equation}\label{eq:star}
(1-s)^2 + 4 s \sin^2\left(\frac{\alpha_0}{2}\right)=(1-s)^2 + 4 s \sin^2\left(\frac{\alpha_1}{2}\right)= \sin^2\left(\frac{\beta}{2}\right)
\end{equation}
thanks to \eqref{eq:coeur2}, where we have set $\alpha_0:=\delta(\usigma_0,\usigma_m)$, $\alpha_1:=\delta(\usigma_m,\usigma_1)$ and $\beta:=\delta(\usigma_0,\usigma_1).$
Note first that \eqref{eq:star} immediately implies $\alpha_0=\alpha_1$. Moreover, for any $t>0$,
$$
\udist\left(\usigma_0,\Phi(t,\sigma_m)\right)+\udist\left(\usigma_1,\Phi(t,\sigma_m)\right)\ge \udist\left(\usigma_0,\usigma_1\right)
$$
implies
$$
\sqrt{(1-t)^2 + 4t \sin^2(\alpha_0/2)} + \sqrt{(1-t)^2 + 4t \sin^2(\alpha_1/2)} \ge 2 \sin\left(\frac{\beta}{2}\right)
$$
thus
$$
(1-t)^2 + 4t \sin^2(\alpha_0/2) \ge \sin^2\left(\frac{\beta}{2}\right).
$$
Therefore, the polynomial function $F : t \mapsto (1-t)^2 + 4t\sin^2(\alpha_o/2) - \sin^2(\beta/2)$
is non-negative and vanishes only at $t=s$, so $F'(s)=0$ hence
$$
2(1-s) = 4 \sin^2\left( \frac{\alpha_0}{2}\right).
$$
Plugging this in \eqref{eq:star} leads to $\sin^2(\beta/2) = \sin^2(\alpha_0/2)$
hence $\alpha_0=\beta_2$.
\end{proof}


Claim \ref{claim:geodesicdelta} and Lemma \ref{lem:length} implies $\delta=\udist_{\uS}$ which yields \eqref{eq:isometry} by Claim \ref{claim:isometrywithdelta}.

\subsection{Equality $l=\alpha$ and positive definiteness of $Q$}

Since $\beta$ is non-degenerate, we can write
\begin{equation}\label{eq:decomposition}
\setR^l = E_+\oplus E_-
\end{equation}
where $E_+$ is a subspace of $\setR^l$ with maximal dimension where $\beta$ is positive definite and $E_-$ is its $\beta$-orthogonal complement; $\beta$ is negative definite on  $E_-$. We call $p_+$ the dimension of $E_+$ and $p_-$ the dimension of $E_-$. Note that $l = p_+ + p_-$ so in particular, $l \ge p_+$. Let us prove $p_+ \ge \alpha$, then $l = \alpha$, in order to reach our conclusion that is $l=p_+=\alpha$. \\

\textbf{\underline{Step 1.}} [$p_+ \ge \alpha$]

Let us write $\uH=(\uH_+,\uH_-)$ where $\uH_+:=\mathrm{proj}_{E_+} \circ \uH$ and $\uH_-:=\mathrm{proj}_{E_-} \circ \uH$, and $\mathrm{proj}_{E_+}, \mathrm{proj}_{E_-}$ are the projections associated to the decomposition \eqref{eq:decomposition}. Moreover, we set $q_+(v_+):=\beta(v_+,v_+)$ for any $v_+ \in E_+$ and $q_-(v_-):=\beta(v_-,v_-)$ for any $v_- \in E_-$. Then for any $\ux, \uy \in \uX$,
$$
Q(\uH(\ux)-\uH(\uy)) = q_+(\uH_+(\ux)-\uH_+(\uy)) + q_-(\uH_-(\ux)-\uH_-(\uy)),
$$
thus
\begin{equation}\label{eq:D9.1}
\udist^2(\ux,\uy) - q_-(\uH_-(\ux)-\uH_-(\uy)) = q_+(\uH_+(\ux)-\uH_+(\uy)).
\end{equation}
Since $q_- \le 0$, it follows from \eqref{eq:D9.1} that $q_+(\uH_+(\ux)-\uH_+(\uy)) \ge \udist^2(\ux,\uy)$. Moreover, $- q_-(\uH_-(\ux)-\uH_-(\uy))$ is bounded from above by $\lambda \udist^2(\ux,\uy)$ where $\lambda$ is the largest modulus an eigenvalue of $Q$ can have, so $q_+(\uH_+(\ux)-\uH_+(\uy)) \le (1+\lambda)\udist^2(\ux,\uy)$. Finally, since $\uH$ is injective, then $\uH_+$ is injective too. Therefore, the map $\uH_+$ is a bi-Lipschitz embedding of $(\uX,\udist)$ into $(E_+,\dist_{q_+})$ where $\dist_{q_+}(v_+,v_+'):=\sqrt{q_+(v_+-v_+')}$ for any $v_+, v_+' \in E_+$. This implies that $p_+$ is greater than or equal to the local Hausdorff dimension of $\uX$ which is equal to $\alpha$.\\

\textbf{\underline{Step 2}.} [$l=\alpha$]

Set $\tilde{\mu} := (\Phi^{-1})_{\#}(\umu \measrestr \uX \backslash \{\uo\})$. Then $\tilde{\mu}$ is a Borel measure on $(0,+\infty)\times \uS$ equipped with $\dist_C$. We complete $(0,+\infty)\times \uS$ by adding the point $\uo$ corresponding to the tip of this metric cone. Let $\di t \otimes \nu_t$ be the disintegration of $\tilde{\mu}$ with respect to the first variable $t$ (we refer to \cite[2.5]{AmbrosioFuscoPallara} for the definition of disintegration of a measure). Since for any $\lambda,r>0$, we have $\umu(h_{\lambda}(B_r(\uo)))=\lambda^{\alpha} \umu(B_r(\uo))$ where $h_\lambda : (t,\usigma)\mapsto (\lambda t, \usigma)$, then $\di \nu_t=t^{\alpha-1}\nu_1$ for any $t>0$ and $\unu_1(\uS)=\alpha \omega_\alpha$. Let us write $\unu$ instead of $\unu_1$.

\begin{claim}
For any $\usigma \in \uS$ and $\uh \in  \ucV:=\Span(\uh_1,\cdots,\uh_l)$,
\begin{equation}\label{proK}
\uh(\usigma) = \frac{\alpha}{\unu(\uS)} \int_{\uS} \cos(\udist_{\uS}(\usigma,\uphi)) \uh(\uphi) \di \unu(\uphi).
\end{equation}
\end{claim}

\begin{proof}
Take $\uh \in \ucV$ and $t>0$. Since $\uh_1,\cdots,\uh_l$ are locally $\uL$-harmonic, then for any $\ux \in \uX$,
\begin{equation}\label{eq:13.56}
\uh(\ux) = \int_{\uX} \up(\ux,\uy,t)\uh(\uy)\di\umu(\uy).
\end{equation}
Use the notation $\ux = \Phi(r,\usigma)$ and $\uy=\Phi(s,\uphi)$ and note that $\uh(\Phi(t,\usigma)) = r \uh(\usigma)$ and $\uh(\Phi(s,\uphi)) = s \uh(\uphi)$ thanks to \eqref{eq:conicalstructure}. Then \eqref{eq:13.56} writes
\begin{equation}\label{eq:draft21}
r \uh(\usigma)=\frac{1}{(4\pi t)^{\alpha/2}}\int_0^{+\infty} \int_{\uS} e^{\frac{-r^2-s^2+2rs\cos(\udist_{\uS}(\usigma,\phi))}{4t}}\uh(\uphi)s^\alpha\di s \di \unu(\uphi).
\end{equation}
Since
$$
\frac{\di}{\di r} \left( e^{\frac{-r^2+2rs\cos(\udist_{\uS}(\usigma,\uphi))}{4t}}\right) = \left( - \frac{r}{2t} + \frac{s\cos(\udist_{\uS}(\usigma,\uphi)}{2t}\right)e^{\frac{-r^2+2rs\cos(\udist_{\uS}(\usigma,\uphi))}{4t}},
$$
differentiating \eqref{eq:draft21} with respect to $r$ and evaluating at $r=0$ gives
$$
\uh(\usigma) = \frac{1}{(4\pi t)^{\alpha/2}} \int_0^{+\infty} e^{-\frac{s^2}{4t}}\frac{s^{\alpha+1}}{2t} \di s \int_{\uS} \cos(\udist_{\uS}(\usigma,\uphi)) \uh(\uphi) \di \unu(\uphi).
$$
A direct computation using the change of variable $\xi=\frac{s^2}{4t}$ shows that
$$
\frac{1}{(4\pi t)^{\alpha/2}} \int_0^{+\infty} e^{-\frac{s^2}{4t}}\frac{s^{\alpha+1}}{2t} \di s = \frac{1}{\omega_\alpha} = \frac{\alpha}{\unu(\uS)}\, .
$$
\end{proof}
Set $\cW:=\{\uh_{|S} \, : \, \uh \in \ucV\}$. Note that \eqref{eq:conicalstructure} implies that the restriction map $\ucV\rightarrow \cW$ is a bijection, hence $\dim \cW=l$.
We introduce the operator 
$\mathcal{K}\colon L^2(\uS,\di\unu)\rightarrow L^2(\uS,\di\unu)$ defined by
$$\mathcal{K}(f)(\usigma):= \frac{\alpha}{\unu(\uS)} \int_{\uS} \cos(\udist_{\uS}(\usigma,\uphi)) f(\uphi) \di \unu(\uphi)
$$
for any $f \in L^2(\uS,\di\unu)$ and $\umu$-a.e.~$\usigma \in \uS$. Since for any $\usigma,\usigma' \in \uS$, \begin{equation}\label{expcos}\cos(\dist_S(\usigma,\usigma')=\uB(\usigma,\usigma')=\beta\left(\uH(\Phi(1,\usigma)),\uH(\Phi(1,\usigma'))\right)=\sum_{i,j} B(x_i,x_j) \uh_i(\usigma)\uh_j(\usigma'),\end{equation}
then the image of $\mathcal{K}$ is contained in $\cW$ and according to \eqref{proK}, we have $$\mathcal{K}f=f \qquad \text{for every $f \in \cW$.}$$
Hence $\mathcal{K}$ is the orthogonal projection onto $\cW$ and 
if $\underline{k}_1,\ldots, \underline{k}_l$ form an orthonormal basis of $\cW$ for the $L^2(\uS,\unu)$ scalar product, then for any $f\in L^2(\uS,\di\nu):$
$$\mathcal{K}(f)(\usigma)= \sum_{i=1}^l \underline{k}_i(\usigma)  \int_{\uS} \underline{k}_i(\uphi) f(\uphi) \di \unu(\uphi)
$$ This implies
\begin{equation}\label{eq:0409}
\frac{\alpha}{\unu(\uS)} \cos(\udist_{\uS}(\usigma,\uphi)) = \sum_{i=1}^l \underline{k}_i(\usigma) \underline{k}_i(\uphi)
\end{equation}
for $\unu\otimes \unu$-a.e.~$(\usigma,\uphi) \in \uS \times \uS$. Since for any $i$, the function $\underline{k}_i$ admits a continuous representative -- still denoted by $\underline{k}_i$ -- defined by
$$
\underline{k}_i (\usigma) = \int_{\uS} \cos(\udist_{\uS}(\usigma,\uphi)) \uh(\uphi) \di \unu(\uphi)
$$
for any $\usigma \in \uS$, then \eqref{eq:0409} holds for all $(\usigma,\uphi) \in \uS \times \uS$. In particular, we can take $\usigma = \uphi$ in \eqref{eq:0409} to get
$$
\frac{\alpha}{\unu(\uS)} = \sum_{i=1}^l \underline{k}_i(\usigma)^2.
$$
Integrating over $\uS$ with respect to $\unu$ gives $\alpha = l.$\\

\subsection{Conclusion}

From the previous subsections, we get that $H$ is an isometric embedding of $(X,\dist)$ into $(\setR^l,\dist_Q)$ or, as explained at the beginning of this section, into $(\setR^l,\dist_e)$. Therefore, $H(X)$ equipped with the restriction of $\dist_e$ is geodesic. Minimizing geodesics in $(\setR^l,\dist_e)$ being straight lines, this implies that $H(X)$ is convex. Being also closed, $H(X)$ is equal to its closed convex hull that is equal to $\setR^l$ by the proof of Claim \ref{claim}, hence Theorem \ref{th:main} is proved.

\section{Almost rigidity result for the heat kernel}

In this section, we show how our rigidity result (Theorem \ref{th:main}) provides an almost rigidity result (Theorem \ref{th:almostrigidity}). We fix a positive constant $T>0$, a positive integer $n$, and we recall that $\setB^n_r$ stands for an Euclidean ball in $\setR^n$ with radius $r>0$ (where this ball is centered as no importance), and $\dist_{GH}$ for the Gromov-Hausdorff distance.

We begin with the following lemma:
\begin{lemma}\label{lem:voleu}  If $(X,\dist,\mu)$ is a complete metric measure space endowed with a symmetric Dirichlet form $\cE$ 
 admitting a heat kernel $p$ such that for some $\gamma >1$, 
\begin{equation}\label{eq:estheatkernel}
 \frac{\upgamma ^{-1}}{(4 \pi t)^{n/2}} e^{-\upgamma\frac{\dist^2(x,y)}{4t}}\le p(x,y,t) \le \frac{\upgamma }{(4 \pi t)^{n/2}} e^{-\frac{\dist^2(x,y)}{4\upgamma t}}
\end{equation}
for all $x, y \in X$ and $t \in (0,T]$, then there exists positive constants $c(n,\upgamma), C(n,\upgamma)$ such that for any $x\in X$ and $r\le \sqrt{T}$,
$$c(n,\upgamma)\, r^n\le \mu(B_r(x)))\le C(n,\upgamma)\, r^n.$$
\end{lemma}
\begin{remark} The upper bound is quite classical, the novelty is the lower bound which was nonetheless known  for stochastically complete spaces (see \cite[Th.~2.11]{Grigor'yan}).
\end{remark}
\proof For any $x \in X$ and $r>0$, integrating the lower bound in \eqref{eq:estheatkernel} gives
$$e^{-\upgamma\frac{r^2}{4t}} \mu\left(B_r(x)) \right)\le \int_{B_r(x)} e^{-\upgamma\frac{\dist^2(x,y)}{4t}}\di\mu(y)\le  \gamma (4 \pi t)^{n/2}$$
hence $\mu\left(B_r(x)) \right)\le   \gamma (4 \pi t)^{n/2} e^{\upgamma\frac{r^2}{4t}}$ for any $t \in (0,T]$. Consequently, when $r\le \sqrt{T}$, choosing $t=r^2$ provides
\begin{equation}\label{eq:uppervolume} \mu\left(B_r(x) \right)\le e^{\frac \upgamma 4} \upgamma(4 \pi )^{n/2}\, r^n,
\end{equation}
while when $r\ge \sqrt{T}$, choosing $t=T$ gives
\begin{equation}\label{eq:uppervolumeT}
\mu\left(B_r(x) \right)\le \upgamma\,(4 \pi T)^{n/2}\,  e^{\upgamma\frac{r^2}{4T}}\,.
\end{equation}
Note that \eqref{eq:uppervolume} is the desired upper bound. Take $t \in (0,T/2]$. Combining \eqref{eq:estheatkernel} with the Chapman-Kolmogrov formula, we get
$$
\frac{\upgamma^{-1}}{(8\pi t)^{n/2}} \le p(x,x,2t) = \int_X p(x,y,t)^2 \di \mu(y) \le \frac{\upgamma^2}{(4 \pi t)^n} \int_X e^{-\frac{\dist^2(x,y)}{2 \gamma t}} \di \mu(y)
$$
hence
\begin{equation}\label{eq:est4}
\upgamma^{-3}\,(2\pi)^{n/2}\, t^{n/2} \le  \int_X e^{-\frac{\dist^2(x,y)}{2\upgamma t}}\di\mu(y) \le \mu(B_r(x))+\int_{X\setminus B_r(x)} e^{-\frac{\dist^2(x,y)}{2\upgamma t}}\di\mu(y).
\end{equation}
From now on, assume $r \le \sqrt{T}$. By Cavalieri's principle and the estimates \eqref{eq:uppervolume} and \eqref{eq:uppervolumeT}, we get
\begin{align}\label{eq:est1}
& \int_{X\setminus B_r(x)} e^{-\frac{\dist^2(x,y)}{2\upgamma t}}\di\mu(y)=\int_r^{+\infty} e^{-\frac{\rho^2}{2\upgamma t}}\frac{\rho}{\upgamma t}\mu(B_\rho(x))\di\rho \nonumber\\
\le & \, \,  (4 \pi )^{n/2}e^{\frac{\upgamma}{4}} \int_r^{\sqrt{T}} e^{-\frac{\rho^2}{2\upgamma t}}\frac{\rho}{ t} \, \rho^n \di \rho+(4 \pi T)^{n/2}\int_{\sqrt{T}}^{+\infty} e^{-\frac{\rho^2}{2\upgamma t}}\frac{\rho}{ t}\,  e^{\upgamma\frac{\rho^2}{4T}}\di\rho.\end{align}
A direct computation shows that for any $n \in \setN$, there exists $C_0>0$ depending only on $n$ such that for any $A\ge 1$,
$$
\int_A^{+\infty} e^{-\frac{\xi^2}{2}} \xi^{n+1} \di \xi \le C_0 A^n e^{-\frac{A^2}{2}}.
$$
Therefore, using the change of variable $\xi= \rho/\sqrt{\upgamma t}$ to get
$$
\int_r^{\sqrt{T}} e^{-\frac{\rho^2}{2\upgamma t}}\frac{\rho}{ t} \rho^n \di \rho\le \int_r^{+\infty} e^{-\frac{\rho^2}{2\upgamma t}}\frac{\rho}{ t}\, \rho^n \di\rho =  \upgamma^{\frac n2+1}  t ^{n/2} \int_{r/\sqrt{\upgamma t}}^{+\infty} e^{-\frac{\xi^2}{2}}\xi^{n+1} \di \xi,
$$
we obtain that $r\ge \sqrt{\upgamma t}$ implies
\begin{equation}\label{eq:est2} \int_r^{\sqrt{T}} e^{-\frac{\rho^2}{2\upgamma t}}\frac{\rho}{\ t}\, \rho^n \di \rho\le C_0 \upgamma r^n e^{-\frac{r^2}{2\upgamma t}}.
\end{equation}
To bound the second term in \eqref{eq:est1}, assume $t\le T/\upgamma^2$. Then a straightforward computation shows that $-\frac{\rho^2}{2\gamma t} + \frac{\gamma \rho^2}{4T} \le - \frac{\rho^2}{4 \gamma t}$ holds, thus
\begin{equation}\label{eq:est3}
\int_{\sqrt{T}}^{+\infty} e^{-\frac{\rho^2}{2\upgamma t}}\frac{\rho}{ t}\,  e^{\upgamma\frac{\rho^2}{4T}}\di \rho \le \int_{\sqrt{T}}^{+\infty} e^{-\frac{\rho^2}{4\upgamma t}}\frac{\rho}{ t}\, \di \rho = 2\upgamma e^{-\frac{T}{4\upgamma t}}.
\end{equation}
Combining \eqref{eq:est1}, \eqref{eq:est2} and \eqref{eq:est3}, we get existence of a constant $C>0$ depending only on $n$ such that if
$r^2 \ge \gamma^2 t$ (this implies both $r\ge \sqrt{\gamma t}$ and $t \le T/\gamma^2$), then
$$\int_{X\setminus B_r(x)} e^{-\frac{\dist^2(x,y)}{2\upgamma t}}\di\mu(y)\le C\left(\upgamma e^{\frac{\upgamma}{4}}r^n e^{-\frac{r^2}{2\upgamma t}}+\upgamma  T^{n/2} e^{-\frac{T}{4\upgamma t}}\right).$$
Then \eqref{eq:est4} implies 
\begin{equation}\label{eq:lem}
\upgamma^{-3}\,(2\pi)^{n/2}\, t^{n/2} \le \mu(B_r(x))+C\left(\upgamma e^{\frac{\upgamma}{4}}r^n e^{-\frac{r^2}{2\upgamma t}}+\upgamma  T^{n/2} e^{-\frac{T}{4\upgamma t}}\right)
\end{equation}
for any $t \in (0,r^2/\gamma^2)$, what can be rewritten as
$$
c'(n,\gamma)t^{n/2}\le \mu(B_r(x))+Ct^{n/2}\left(\upgamma e^{\frac{\upgamma}{4}} F(r^2/t)+\upgamma  G(T/t)\right)
$$
where $c'(n,\gamma):=\upgamma^{-3}\,(2\pi)^{n/2}$ and $F(s):=s^{n/2}e^{-\frac{s}{2\gamma}}$, $G(s):=s^{n/2}e^{-\frac{s}{4\gamma}}$ for any $s\ge 0$. The function $G$ is decreasing on $(2 n\gamma,+\infty)$ so if $r^2/t \ge 2 n\gamma$, since $r^2 \le T$, we get $G(T/t) \le G(r^2/t)$. As $\lim\limits_{s \to +\infty}F(s)=\lim\limits_{s \to +\infty} G(s)=0$, then there exists $s(n,\gamma)>0$ such that if $s\ge s(n,\gamma)$,
$$
C\left(\upgamma e^{\frac{\upgamma}{4}} F(s)+\upgamma  G(s)\right) \le \frac{c'(n,\gamma)}{2}\, \cdot
$$
Then for any $t >0$ such that $r^2/t \ge \max(\gamma^2,2 n\gamma,s(n,\gamma))=:\theta(n,\gamma)$, we get
$$
\frac{c'(n,\gamma)}{2}t^{n/2} \le \mu(B_r(x)).
$$
Choosing $t=t(r)$ such that $\theta(n,\gamma) t \le r^2 \le 2 \theta(n,\gamma)t$, we get
$$
\frac{c'(n,\gamma)}{2^{n/2+1}\theta(n,\gamma)^{n/2}} r^{n} \le \mu(B_r(x)).
$$

\endproof

We shall also need the next proposition.
\begin{proposition}\label{prop:5.3} Let $(X,\dist,\mu)$ be a measure metric space satisfying the local doubling condition, namely there exists $r_o>0$ and $C_D>0$ such that $\mu(B_{2r}) \le C_D \mu(B_r)$ for any $r \in (0,r_o)$, and such that for some $\alpha >0$, we have
$$\int_X \frac{1}{(4 \pi t)^{\alpha/2}} e^{-\frac{\dist^2(x,z)}{4t}}\frac{1}{(4 \pi s)^{\alpha/2}} e^{-\frac{\dist^2(z,y)}{4s}} \di \mu(z)=\frac{1}{(4 \pi (t+s))^{\alpha/2}} e^{-\frac{\dist^2(x,y)}{4(t+s)}}$$
 for all $x,y \in X$ and $t,s \in (0,T)$. Then there exists a symmetric Dirichlet form $\cE$ on $(X,\dist,\mu)$ admitting an $\alpha$-Euclidean heat kernel.
\end{proposition}
\proof By \cite[Lem.~3.9]{Carron}, the space $(X,\dist,\mu)$ satisfies $
\mu(B_R(x))/\mu(B_r(x)) \le c_o e^{c_1 R/r}$ for any $x \in X$, $r\in(0,r_o)$ and $R\ge r$, where $c_o$ and $c_1$ depend only on $C_D$. For any $C>0$ and $z \in X$, applying \eqref{eq:Fub} with $\phi(\lambda)=\lambda^2 e^{-C\lambda^2}$ and $g(y)=\dist(z,y)$ yields to
\begin{align}\label{eq:exp}
\int_X e^{-C\dist^2(z,y)}\di \mu(y) & = \int_0^{+\infty} 2 C \lambda e^{-C\lambda^2} \mu(B_\lambda(y))\di \lambda \nonumber \\
& \le \left(\int_0^r 2 C \lambda e^{-C\lambda^2} \di \lambda + \int_r^{+\infty} 2 C \lambda c_o e^{-C\lambda^2 + c_1 \lambda/r} \di \lambda\right)\mu(B_r(y)) \nonumber \\
& \le c_2 \mu(B_r(y))
\end{align}
for any $r \in(0,r_o)$, where $c_2$ depends only on $C_D$ and $C$. For any $x,y\in X$ and $t\in \setC\setminus \{0\}$,
we set 
$$\bP_\alpha(x,y,t):=\frac{1}{(4 \pi t)^{\alpha/2}} e^{-\frac{\dist^2(x,y)}{4t}}\, .$$
Take $x,y \in X$ and $t\in(0,T)$. By assumption, the identity
$$\int_X \bP_\alpha(x,z,t)e^{-\frac{\dist^2(z,y)}{4s}}\di \mu(z)= \left(\frac{s}{t+s}\right)^{\frac \alpha2}e^{-\frac{\dist^2(z,y)}{4(t+s)}}$$
is valid for any $s \in (0,T-t)$. However both expressions are holomorphic in $s\in \setC_+:=\{z\in \setC, \mathrm{Re} z>0\}$ (the left-hand side can be proved holomorphic by a suitable application of the dominated convergence theorem using \eqref{eq:exp}), thus the identity holds for any $s\in \setC_+$. Freezing $s$ and letting $t$ be variable, we can apply the same reasoning to get the identity valid for any $s, t \in \setC_+$. In particular, we obtain:
$$\int_X \bP_\alpha(x,z,t) \bP_\alpha(z,y,s)\di\mu(z)= \bP_\alpha(x,y,t+s) \qquad \forall s,t >0.$$
Thus for any $x \in X$ and $t>0$,
\begin{align*}
\int_X \bP_\alpha(x,z,t) \di \mu(z) & = (4 \pi t)^{\alpha/2} \int_X \left( \frac{1}{(4 \pi t)^{\alpha/2}} e^{-\frac{\dist^2(x,z)}{8t}} \right)^2 \di \mu(z)\\
& = (16 \pi t)^{\alpha/2} \int_X \bP_\alpha(x,z,2t)^2 \di \mu(z)\\ & =(16 \pi t)^{\alpha/2}  \bP_\alpha(x,x,4t) = 1.
\end{align*}
This easily implies that for any $f\in L^2(X,\mu)$, if
 $f_t(x)=\int_X \bP_\alpha(x,z,t)f(z) \di\mu(z)$, then 
 $$\lim_{t\to 0+} \|f_t-f\|_{L^2}=0.$$ 
Then by a standard procedure described for instance in \cite[Section 2]{Grigor'yan}, we can build a symmetric Dirichlet form whose heat kernel is $\bP_\alpha$.
\endproof

We can now prove Theorem \ref{th:almostrigidity}.
\proof The metric spaces considered in this proof are all complete. Assume that the result is not true. Then there exists some $\upepsilon>0$  such that for any $\updelta>0$ we can find:\begin{itemize}
\item $T_\updelta>0$,
\item a  metric measure space $(X_\updelta,\dist_\updelta, \mu_\updelta)$ endowed with a symmetric Dirichlet form $\cE_\updelta$ 
 admitting a heat kernel $p_\updelta$ satisfying
$$(1-\updelta) \frac{1}{(4 \pi t)^{n/2}} e^{-\frac{\dist_\updelta^2(x,y)}{4(1-\updelta)t}}\le p_\upepsilon(x,y,t) \le (1+\updelta) \frac{1}{(4 \pi t)^{n/2}} e^{-\frac{\dist_\updelta^2(x,y)}{4(1+\updelta) t}} $$
for any $x,y \in X_\updelta$ and $t\in (0,T_\updelta]$,
\item $x_\updelta\in X_\updelta$ and $r_\updelta\in (0,\sqrt{T_\updelta}]$ such that $\dist_{\mathrm GH}\left( B_{r_\updelta}(x_\updelta), \setB^n_{r_\updelta}\right) \ge \upepsilon r_\updelta$.
\end{itemize}
By a rescaling of the distance and of the measure, we can assume that $r_\updelta=1$ and $T_\updelta=1$.
It follows from Lemma \ref{lem:voleu} that the set of pointed metric measure space
$$\left\{(X_\updelta,\dist_\updelta,\mu_\updelta,x_\updelta)\right\}_{\updelta\in (0,1/2)}$$ satisfies a uniform local doubling condition, thus it is precompact for the pointed measure Gromov-Hausdorff topology. Therefore, we can consider an infinitesimal sequence $\{\updelta_\ell\}_\ell \subset (0,1/2)$ and a sequence of pointed metric measure spaces $$\left\{(X_\ell,\dist_\ell,\mu_\ell,x_\ell)\right\}_{\ell}$$ converging to some pointed metric measure space $(X_\infty,\dist_\infty,\mu_\infty,x_\infty)$ such that for any $\ell$:
\begin{itemize}
\item the space $(X_\ell,\dist_\ell,\mu_\ell)$ is endowed with a symmetric Dirichlet form $\cE_\ell$ 
 admitting a heat kernel $p_\ell$ satisfying
\begin{equation}\label{eq:14.17}
(1-\updelta_\ell) \frac{1}{(4 \pi t)^{n/2}} e^{-\frac{\dist_\ell^2(x,y)}{4(1-\updelta_\ell)t}}\le p_\ell(x,y,t) \le (1+\updelta_\ell) \frac{1}{(4 \pi t)^{n/2}} e^{-\frac{\dist_\ell^2(x,y)}{4(1+\updelta_\ell)t}}\end{equation}
for any $x,y \in X_\ell$ and $t\in (0,1]$,
\item $\dist_{\mathrm GH}\left( B_1(x_\ell), \setB^n_1\right) \ge \upepsilon. $
\end{itemize}
In particular, letting $l$ tend to $+\infty$ gives:
\begin{equation}\label{disHG}\dist_{\mathrm GH}\left( B_1(x_\infty), \setB^n_1\right) \ge \upepsilon.\end{equation}
Since for any $\ell$ we have
$$\int_{X_\ell}p_\ell(x,z,t)p_\ell(z,y,s)d\mu_\ell(z)=p_\ell(t+s,x,y)$$
for all $x,y \in X_\ell$ and $t,s>0$, we deduce from \eqref{eq:14.17} that when $t+s<1$,
$$\frac{(1-\updelta_\ell)^{\frac n2+1}}{(1+\updelta_\ell)^{n+1}}\ \bP_n\left(x,y,\frac{1-\updelta_\ell}{1+\updelta_\ell}(t+s)\right)\le \int_{X_\ell}\bP_n(x,z,t)\bP_n(z,y,s)d\mu_\ell(z)$$
and
 $$\int_{X_\ell}\bP_n(x,z,t)\bP_n(z,y,s)d\mu_\ell(z)\le \frac{(1+\updelta_\ell)^{\frac n2+1}}{(1-\updelta_\ell)^{n+1}}\ \bP_n\left(x,y,\frac{1+\updelta_\ell}{1-\updelta_\ell}(t+s)\right).$$
From this, we obtain for any $x,y\in X_\infty$ and any $t,s>0$ with $t+s<1$,
$$\int_{X_\infty}\bP_n(x,z,t)\bP_n(z,y,s)d\mu_\infty(z)= \bP_n(x,y,t+s).$$
Then Proposition \ref{prop:5.3} and Theorem \ref{th:main} imply that $(X_\infty,d_\infty)$ is isometric to $(\setR^n,\dist_e)$. But this is in contradiction with \eqref{disHG}.
\endproof

\section{A new proof of Colding's almost rigidity theorem}

In this section, we show how our almost rigidity result (Theorem \ref{th:almostrigidity}) can be used to give an alternative proof of the almost rigidity theorem for the volume of Riemannian manifolds with non-negative Ricci curvature (Theorem \ref{th:Colding}). Here again $n$ is a fixed positive integer and $\mathbb{B}_r^n$ is an Euclidean ball in $\setR^n$ with radius $r>0$.

We recall that whenever $(M^n,g)$ has non-negative Ricci curvature, the Bishop-Gromov comparison theorem states that the function $r\mapsto \omega_n^{-1} r^{-n}\vol (B_r(x))$ is non-increasing for any $x \in M$
and the quantity
\begin{equation}\label{eq:theta}
\uptheta=\lim_{r\to +\infty} \frac{\vol (B_r(x))}{\omega_n\, r^n}
\end{equation}
does not depend on $x$. When $\uptheta>0$, we say that $(M^n,g)$ has Euclidean volume growth, in which case one has
\begin{equation}\label{eq:volumelowerbound}
\vol (B_r(x)) \ge \uptheta \omega_n r^n
\end{equation}
for any $x \in X$ and $r>0$. Note that a manifold satisfying \eqref{eq:vol} has Euclidean volume growth with $\uptheta \ge 1-\delta$. Our proof of Theorem \ref{th:Colding} is a direct application of Theorem \ref{th:almostrigidity} together with the following heat kernel estimate.
\begin{theorem}\label{th:estimate} There exists a function $\upgamma\colon [0,1]\to [1,\infty)$ satisfying $\lim_{\theta\to 1^-} \upgamma(\theta)=1$ such that 
whenever $(M^n,g)$ is a complete Riemannian manifold with non-negative Ricci curvature and Euclidean volume growth, then the heat kernel $p$ of $(M^n,g)$ satisfies
$$\frac{1}{(4\pi t)^{\frac n2}}e^{-\frac{d^2(x,y)}{4t}}\le p(x,y,t)\le \upgamma(\uptheta) \frac{1}{(4\pi t)^{\frac n2}}e^{-\frac{d^2(x,y)}{\upgamma(\uptheta)4t}}
$$
for all $x,y \in M$ and $t>0$, where $\uptheta$ is given by \eqref{eq:theta}.
\end{theorem} 

\begin{remark} Our proof of the above heat kernel upper bound follows the arguments of P. Li, L-F. Tam  and J. Wang \cite{LiTamWang}.
\end{remark}

\begin{proof}
The lower bound is the comparison theorem of J.~Cheeger and S-T.~Yau \cite{CheegerYau}: for any $t>0$ and $x,y\in M$, we have
\begin{equation}\label{eq:CheegerYau}\bP_n(x,y,t)\le p(x,y,t)\end{equation} where
$\displaystyle \bP_n(x,y,t)=(4\pi t)^{-\frac n2}e^{-\frac{d^2(x,y)}{4t}}.$ Consequently we only need to prove the upper bound.

Take $x,y \in X$ and $t>0$. We shall need the following estimates from P. Li and S-T. Yau (see \cite[Formula (2.1)]{LiTamWang}): for any $r,\tau>0$,
\begin{equation}\label{eq:LY1}\int_{B_r(x)}p(x,z,\tau)\di\vol(z)\ge \int_{\mathbb{B}^n_r}\frac{1}{(4\pi \tau)^{\frac n2}}e^{-\frac{\|\xi\|^2}{4\tau}}\di \xi
\end{equation}
and 
\begin{equation}\label{eq:LY2}
\int_{M\setminus B_r(x)}p(x,z,\tau)\di \vol(z)\le \int_{\setR^n\setminus \mathbb{B}_r^n}\frac{1}{(4\pi \tau)^{\frac n2}}e^{-\frac{\|\xi\|^2}{4\tau}}\mathrm{\di}\xi.
\end{equation}

For $\delta>0$ to be precisely chosen later, set $r:=(1+\delta)^{-1}\dist(x,y)$ and $\tau := (1+\delta)t$. Note that $B_{r}(x) \cap B_{\delta r}(y) = \emptyset$. By the Harnack inequality of P.~Li and S-T.~Yau \cite{LiYau}, we have
\begin{equation}\label{LY}
p(x,y,t)\le \left(\frac{\tau}{t}\right)^{\frac n2}\, e^{\frac{\dist(z,y)^2}{4(\tau-t)}} p(x,z,\tau)\end{equation}
for every $z\in M$, so that averaging over the ball $B_{\delta r}(y)$ gives
\begin{align}\label{eq:Colding111}p(x,y,t)&\le e^{\frac{\delta^2r^2}{4(\tau-t)} } \left(\frac{\tau}{t}\right)^{\frac n2}\fint_{B_{\delta r}(y)}p(x,z,\tau)\di\vol(z) \noindent \\
&\le e^{\frac{\delta d^2(x,y)}{4(1+\delta)^2t}} \left(\frac{\tau}{t}\right)^{\frac n2}\fint_{B_{\delta r}(y)}p(x,z,\tau)\di\vol(z).
\end{align}
Now 
\begin{align*}
\int_{B_{\delta r}(y)}p(x,z,\tau)\di \vol(z)&=\int_{M\setminus B_r(x)}p(x,z,\tau)\di \vol(z)- \int_{M\setminus (B_r(x)\cup B_{\delta r}(y))} p(x,z,\tau)\di \vol(z)\\
&\le \int_{\setR^n\setminus \mathbb{B}_r^n}\frac{1}{(4\pi \tau)^{\frac n2}}e^{-\frac{\|\xi\|^2}{4\tau}}\di \xi - \int_{M\setminus (B_r(x)\cup B_{\delta r}(y)}\mathbb{P}_n(x,z,\tau)\di \vol(z)
\end{align*}
thanks to \eqref{eq:LY2} and \eqref{eq:CheegerYau}. Continuing,
\begin{align*}
\int_{B_{\delta r}(y)}p(x,z,\tau)\di \vol(z) \le \int_{\setR^n\setminus \setB_r^n}\frac{1}{(4\pi \tau)^{\frac n2}}e^{-\frac{\|\xi\|^2}{4\tau}}\di \xi & - \int_{M\setminus B_r(x)}\bP_n(x,z,\tau)\di \vol(z)\\
&+ \int_{B_{\delta r}(y)}\bP_n(x,z,\tau)\di \vol(z)\\
\le \int_{\setR^n\setminus \setB_r^n}\frac{1}{(4\pi \tau)^{\frac n2}}e^{-\frac{\|\xi\|^2}{4\tau}}\di \xi & - \int_{M\setminus B_r(x)}\bP_n(x,z,\tau)\di \vol(z)\\
&+\vol(B_{\delta r}(y)) \frac{1}{(4\pi \tau)^{\frac n2}}e^{-\frac{(\dist(x,y)-\delta R)^2}{4\tau}}.
\end{align*}
By Cavalieri's principle and \eqref{eq:volumelowerbound}, we have
\begin{align*}
\int_{M\setminus B_r(x)}\bP_n(x,z,\tau)\di \vol(z) & =\int_{r}^{+\infty} \frac{1}{(4\pi \tau)^{\frac n2}}e^{-\frac{s^2}{4\tau}}\frac{s}{2\tau}\vol(B_s(x))\di s\\
& \ge \int_{r}^{+\infty} \frac{1}{(4\pi \tau)^{\frac n2}}e^{-\frac{s^2}{4\tau}}\frac{s}{2\tau} \theta \omega_n s^n \di s = \theta \int_{\setR^n\setminus \setB_r^n}\frac{1}{(4\pi \tau)^{\frac n2}}e^{-\frac{\|\xi\|^2}{4\tau}}\di \xi,
\end{align*}
hence
\begin{equation}\label{eq:Colding222}
\int_{B_{\delta r}(y)}p(x,z,\tau)\di \vol(z) \le (1-\theta)\int_{\setR^n\setminus \setB_r^n}\frac{1}{(4\pi \tau)^{\frac n2}}e^{-\frac{\|\xi\|^2}{4\tau}}\di \xi +\vol(B_{\delta r}(y)) \frac{1}{(4\pi \tau)^{\frac n2}}e^{-\frac{(\dist(x,y)-\delta r)^2}{4\tau}}\, \cdot
\end{equation}
As pointed out in \cite[Formula (2.6)]{LiTamWang}, direct computations show that there exists a constant $C=C(n)>0$ such that
$$
\int_{\setR^n\setminus \setB_r^n}\frac{1}{(4\pi \tau)^{\frac n2}}e^{-\frac{\|\xi\|^2}{4\tau}}\di \xi \le C\left(1+\left(\frac{r}{\sqrt{4\pi\tau}}\right)^n\right)e^{-\frac{r^2}{4\tau}}.
$$
This together with \eqref{eq:Colding222} and  \eqref{eq:Colding111} yields to
\begin{align*}
p(x,y,t)&\le (1-\uptheta) C\left(1+\left(\frac{r}{\sqrt{4\pi\tau}}\right)^n\right)e^{-\frac{r^2}{4\tau}}e^{\frac{\delta \dist^2(x,y)}{4(1+\delta)^2t}} \left(\frac{\tau}{t}\right)^{\frac n2}\frac{1}{\vol (B_{\delta r}(y))}\\
& \qquad \qquad \qquad \qquad \qquad \qquad \qquad +\frac{1}{(4\pi t)^{\frac n2}}e^{-\frac{(\dist(x,y)-\delta r)^2}{4\tau}}e^{\frac{\delta \dist^2(x,y)}{4(1+\delta)^2t}} \, .
\end{align*}
It is easily checked that $
-\frac{r^2}{4\tau} =-\frac{(\dist(x,y)-\delta r)^2}{4\tau} = - \frac{\dist^2(x,y)}{4(1+\delta)^3 t}$
hence
\begin{align*}
p(x,y,t)&\le\left[(1-\uptheta) C\left(1+\left(\frac{r}{\sqrt{4\pi\tau}}\right)^n\right)\left(1+\delta\right)^{\frac n2}\underbrace{\frac{1}{\vol (B_{\delta r}(y))}}_{\le 1/(\uptheta \omega_n \delta^n r^n)}+\frac{1}{(4\pi t)^{\frac n2}}\right]e^{-\frac{(1-\delta-\delta^2)\dist^2(x,y)}{4(1+\delta)^3t}}\\
& \le\left[(\uptheta^{-1}-1) C\left(1+\left(\frac{r}{\sqrt{4\pi\tau}}\right)^n\right)\frac{(4\pi \tau)^{\frac n2}}{(4\pi t)^{\frac n2}}\frac{1}{\omega_n \delta^n r^n}+\frac{1}{(4\pi t)^{\frac n2}}\right]e^{-\frac{(1-\delta-\delta^2)\dist^2(x,y)}{4(1+\delta)^3t}}\\
& = \left[(\uptheta^{-1}-1)\left(1+\left(\frac{r}{\sqrt{4\pi\tau}}\right)^n\right) \frac{C}{\omega_n}\frac{(4\pi \tau)^{\frac n2}}{\delta^n r^n} +1\right]\frac{1}{(4\pi t)^{\frac n2}}e^{-\frac{(1-\delta-\delta^2)\dist^2(x,y)}{4(1+\delta)^3t}} \, .
\end{align*}
Now we distinguish two cases. According to \cite[Formula (2.4)]{LiTamWang}, if $\dist(x,y)\le \delta \sqrt{t}$, then
\begin{equation}\label{esti1}p(x,y,t)\le \frac{1}{\uptheta} \frac{1}{(4\pi t)^{\frac n2}} e^{-\frac{d^2(x,y)}{4t}} e^{\frac{\delta^2}{4}}.\end{equation}
If $\dist(x,y)\ge \delta \sqrt{t}$, then $\frac{r}{\sqrt{\tau}}\ge \frac{\delta}{(1+\delta)^{2}}$, thus
\[
\left(1+\left(\frac{r}{\sqrt{4\pi\tau}}\right)^n\right) \frac{(4\pi \tau)^{\frac n2}}{\delta^n r^n} = \left( \frac{(4\pi \tau)^{\frac n2}}{r^n} + 1 \right)\delta^{-n} \le (4\pi)^{\frac n2}\left(\frac{\delta+ 1}{\delta}\right)^{2n}+\left(\frac{1}{\delta}\right)^n.
\]
Therefore, if $\delta<1/2$, we get
$$\left(1+\left(\frac{r}{\sqrt{4\pi\tau}}\right)^n\right) \frac{(4\pi \tau)^{\frac n2}}{\delta^n r^n} \le C'\delta^{-2n}$$ where $C'$ depends only on $n$, which yields to
$$
p(x,y,t) \le \left[(\uptheta^{-1}-1)\Lambda \delta^{-2n}+1\right]\frac{1}{(4\pi t)^{\frac n2}}e^{-\frac{(1-\delta-\delta^2)\dist^2(x,y)}{4(1+\delta)^3t}}.
$$
where $\Lambda:=CC'/\omega_n$ depends only on $n$. Now we choose \begin{equation}\label{eq:delta}
\delta=\delta(\uptheta):=\min\left\{ \frac 12\, ,  \left(\left(\uptheta^{-1}-1\right)\Lambda\right)^{\frac{1}{2n+1}}\right\},\end{equation}
so that when $\left(\left(\uptheta^{-1}-1\right)\Lambda\right)^{\frac{1}{2n+1}} < 1/2$ then $$\delta(\uptheta)= \left(\uptheta^{-1}-1\right)\Lambda\delta(\uptheta)^{-2n},$$ hence
\begin{equation}\label{eq:end}
p(x,y,t)\le \frac{\delta(\uptheta)+1}{(4\pi t)^{\frac n2}}e^{-\frac{(1-\delta(\uptheta)-\delta(\uptheta)^2)\dist^2(x,y)}{4(1+\delta(\uptheta))^3t}},\,\end{equation}
and when $\left(\left(\uptheta^{-1}-1\right)\Lambda\right)^{\frac{1}{2n+1}} \ge 1/2$ -- which corresponds to the case $\uptheta \le 1-\eps_n$ with $\eps_n:=(1+2^{2n+1}\Lambda)^{-1}$ depending only on $n$ -- then $\delta(\theta)=1/2$ implies
\begin{equation}\label{eq:end2}
p(x,y,t)\le \frac{(\uptheta^{-1} - 1)\Lambda 2^{2n} + 1}{(4\pi t)^{\frac n2}}e^{-\frac{(1-\delta(\uptheta)-\delta(\uptheta)^2)\dist^2(x,y)}{4(1+\delta(\uptheta))^3t}}\,.\end{equation}
Note that $\delta(\uptheta) \to 0$ when $\uptheta \to 1$. Therefore, setting $F(\uptheta):=(1+\delta(\uptheta))^3/(1-\delta(\uptheta)-\delta(\uptheta)^2)$ and
\[
\gamma(\uptheta):=
\begin{cases}
\max(1+\delta(\uptheta),F(\uptheta)) & \text{if $1-\eps_n<\uptheta<1$,}\\
\max(2^{2n}\Lambda (\uptheta^{-1} - 1)+1,F(\uptheta)) & \text{if $0<\uptheta \le 1-\eps_n$,}
\end{cases}
\]
we get the result.

\end{proof}

For completeness, let us provide a short proof of Theorem \ref{th:Colding}.

\begin{proof}
Take $\upepsilon>0$. By Theorem \ref{th:almostrigidity}, there exists $\delta'=\delta'(n,\upepsilon)>0$ such that if $(M^n,g)$ is complete and satisfying $\Ric \ge 0$ and \eqref{eq:almostheatkernel} with $\delta$ replaced by $\delta'$, then any ball with radius $r$ in $M$ is $(\upepsilon r)$-GH close from a ball with same radius in $\setR^n$. But Theorem \ref{th:estimate} implies that there exists $\delta=\delta(n,\delta')=\delta(n,\upepsilon)>0$ such that if $1-\delta\le \theta$ holds, then $\gamma(\theta)-1 \le \delta$ and thus \eqref{eq:almostheatkernel} is true. The result follows.
\end{proof}

\section{Case of a Spherical heat kernel}

In this section, for any Riemannian manifold $(M^n,g)$, we define the operator $L$ acting on $L^2(M)$ as the Friedrich extention of the operator $\tilde{L}$ defined by the formula:
\[
- \int_M (\tilde{L}u)v = \int_M \langle \nabla u, \nabla v \rangle \qquad \forall u,v \in C^\infty_c(M).
\]
The spectral theorem implies that $L$ generates a semi-group $(e^{tL})_{t>0}$ which admits a smooth heat kernel.

The heat kernel of $(e^{tL})_{t>0}$ on the sphere $\mathbb{S}^n$ equipped with the canonical spherical metric $g_{\mathbb{S}^n}$ admits a well-known expression, namely
$$
K_t^{(n)}(\dist_{\mathbb{S}^n}(x,y)) 
$$
for any $x,y \in \mathbb{S}^n$ and $t>0$, where $\dist_{\mathbb{S}^n}$ is the Riemannian distance canonically associated with $g_{\mathbb{S}^n}$ and \begin{equation}\label{eq:K_t}
K_t^{(n)}(r):=\sum_{i=0}^{+\infty} e^{\lambda_i t} C_i^{(n)}(r)\end{equation} for any $r>0$, with $\lambda_i := -i(i+n-1)$ and $
C_i^{(n)}(\cdot):=(2i+n-1)(n-1)^{-1}\sigma_n^{-1} G_i^{\frac{n-1}{2}}(\cos(\cdot))$ for any  $i \in \setN$. Here the functions $G_i^\alpha$ are the Gegenbauer polynomials (see e.g.~\cite{AtkinsonHan}). For our purposes, it is worth mentioning that $$C_0^{(n)}(r)=\frac{1}{\sigma_n} \quad \text{and} \quad C_1^{(n)}(r)=\frac{n+1}{\sigma_n}\cos(r)$$ for any $r>0$. Moreover, the sum in \eqref{eq:K_t} converges uniformly in $C([0,+\infty))$.


\begin{theorem}
Let $(X,\dist,\mu)$ be a complete metric measure space equipped with a Dirichlet form $\cE$ admitting a spherical heat kernel $p$, that is
\begin{equation}\label{eq:spher}
p(x,y,t) = K_t^{(n)}(\dist(x,y))
\end{equation}
for any $x, y \in X$ and $t>0$. Then $(X,\dist)$ is isometric to $(\mathbb{S}^n,\dist_{\mathbb{S}^n})$.
\end{theorem}

\begin{proof}
Let $L$ be the self-adjoint operator canonically associated with $\cE$ and $(P_t)_{t>0}$ the associated semi-group. Assumption \eqref{eq:spher} implies that for any $t>0$ and $f \in L^2(X,\mu)$,
$$
\int_{-\infty}^{+\infty} e^{t \lambda} \di (f,E_\lambda f) = \sum_{i=0}^{+\infty} e^{\lambda_i t} \iint_{X\times X} C_i(\dist(x,y)) f(x) f(y) \di \mu(x) \di \mu(y)
$$
holds, where $(f,E_\lambda f)$ is the projection-valued measure of $L$ associated with $f$, see e.g.~\cite[p.~262-263]{ReedSimon}. Uniqueness of the map $f \mapsto (f,E_\lambda f)$ implies that the spectrum of $L$ is given by $\lambda_0, \lambda_1, \lambda_2, \ldots$ and that the projection operators $P_i : L^2(X,\mu) \to E_i:=\mathrm{Ker}(L-\lambda_i \mathrm{Id})$, for any $i \in \setN$, have a kernel $p_i$ such that for any $x,y \in X$,
$$
p_i(x,y)=C_i(\dist(x,y)).
$$
Since $P_i$ commutes with $L$ for any $i$, we have $P_i Lg = \lambda_i P_i g$ for any $g \in \cD(L)$, thus $
\langle p_i(x,\cdot), Lg \rangle_{L^2} = \lambda_i \langle p_i(x,\cdot), g \rangle_{L^2}$
for any $x \in X$. This implies $p_i(x,\cdot) \in \cD(L)$ with $$Lp_i(x,\cdot)=\lambda_i p_i(x,\cdot)$$
for any $x \in X$. In case $i=0$, as $\lambda_0=0$ and $p_0(x,y)=C_0(\dist(x,y))=1/\sigma_n$ for any $x, y \in X$, we get $L\textbf{1}=0$ thus $P_0 \textbf{1}=\textbf{1}$. This implies $\int_Xp_0(x,y) \di \mu(y)=1$ for any $x \in X$, hence
\begin{equation}\label{eq:volsphere}
\mu(X) = \sigma_n\, .
\end{equation}
In case $i=1$, we have $\lambda_1=-n$ and $p_1(x,y)=C_1(\dist(x,y))=\frac{n+1}{\sigma_n} \cos(\dist(x,y))$
for any $x,y \in X$, hence
\begin{equation}\label{eq:cosdist}
L_x\cos(\dist(x,y)) = - n \cos(\dist(x,y)).
\end{equation}
Let $\phi_1,\ldots, \phi_l$ be continuous functions forming a $L^2(X,\mu)$-orthogonal basis of $E_1$. Observe that
\begin{equation}\label{eq:p1}
P_1f(x) = \int_X p_1(x,y) f(y) \di \mu(y) = \int_X\frac{n+1}{\sigma_n} \cos(\dist(x,y)) f(y) \di \mu(y)
\end{equation}
and
$$
P_1f(x)=\sum_{j=1}^l \left(\int_X \phi_i(y) f(y) \di \mu(y)\right)\phi_i(x)= \int_X \left[\sum_{j=1}^l \phi(y)\phi(x)\right] f(y)\di \mu(y)
$$
holds for any $f \in L^2(X,\mu)$ and $x \in X$. This implies
$$\sum_{i=1}^l\phi(x)^2 = \frac{n+1}{\sigma_n}$$ for any $x \in X$, hence integration over $X$ and \eqref{eq:volsphere} provides $$l=n+1.$$ Setting $$\cV:=\Span\{\cos(\dist(x,\cdot))  :  x \in X\}$$ we get $\cV \subset E_1$ thanks to \eqref{eq:cosdist}. Since $E_1$ is the image of $L^2(X,\mu)$ by $P_1$, the reverse inclusion follows from \eqref{eq:p1}, hence $$\cV=E_1.$$
Acting as in Subsection $4.1$, we can show that there exist $x_1, \ldots, x_{n+1} \in X$ such that $\{\delta_{x_1},\ldots,\delta_{x_{n+1}}\}$ is a basis of $\cV^*$ whose associated basis $\{h_1,\ldots,h_{n+1}\}$ of $\cV$ permits to write
\begin{equation}\label{eq:cos1}
\cos(\dist(x,y)) = \sum_{i,j=1}^{n+1} c_{ij} h_i(x) h_j(y)
\end{equation}
for any $x,y \in X$, where $c_{ij}:=\cos(\dist(x_i,x_j))$ for any $i,j$. Let $\beta$ be the bilinear form defined by $$
\beta(\xi,\xi') = \sum_{i,j=1}^{n+1} c_{ij}\xi_i \xi_j'$$ for any $\xi=(\xi_1,\cdots,\xi_{n+1}), \xi'=(\xi_1',\cdots,\xi_{n+1}') \in \setR^{n+1}$ and $Q$ the associated quadratic form. Set
$$
H:
\begin{array}{ccl}
X & \to & \setR^{n+1}\\
x & \mapsto & (h_1(x),\ldots,h_{n+1}(x)).
\end{array}
$$
Then \eqref{eq:cos1} writes as
\begin{equation}\label{eq:cos2}
\cos(\dist(x,y)) = \beta(H(x),H(y)).
\end{equation}
Choosing $y=x$ implies $H(x) \in \Sigma:=\{\xi \in \setR^{n+1} \, : \, \beta(\xi,\xi) =1\}$, so $H(X)$ is a subset of $\Sigma$. A direct computation provides:
\begin{equation}\label{eq:cos3}
Q(H(x)-H(y))=4\sin^2\left( \frac{\dist(x,y)}{2}\right) \qquad \forall x,y \in X,
\end{equation}
from which follows that $H$ is an injective map. Writing $
\setR^{n+1} = E_+ \oplus E_- \oplus \mathrm{Ker} \beta$ where $E_+, E_-$ are subspaces of $\setR^{n+1}$ where $\beta$ is positive definite and negative definite respectively, we can proceed as in 4.3, Step 1 (using the same notations) to get that $H_+$ is a bi-Lipschitz embedding of $(X,\dist)$  onto its image in $(E_+,q_+)$. Therefore, $\dim(E_+)$ is greater than the Hausdorff dimension of $X$.

\begin{claim}
The Hausdorff dimension of $X$ is $n$.
\end{claim}

\begin{proof}
The short-time expansion of the heat kernel on Riemannian manifolds \cite{MinakshisundaramPleijel} and the Cheeger-Yau estimate \cite{CheegerYau} implies that for some $C>0$ and $t_o>0$,
$$
\frac{1}{(4\pi t)^{n/2}} e^{-\frac{r^2}{4t}} \le K_t^{(n)}(r) \le \frac{C}{(4\pi t)^{n/2}} e^{-\frac{r^2}{5t}}
$$
holds for any $r \in (0,\pi)$ and $t\in (0,t_0)$. Therefore, proceeding as in the proof of Lemma \ref{lem:voleu}, we get existence of a positive constant $C$ such that for any $x \in X$ and any $r\in (0,\sqrt{t_0})$,
$$C^{-1}r^n\le \mu(B_r(x))\le C r^n.$$
Hence the claim is proved.
\end{proof}

Thus $n+1 \ge \dim(E_+)> n$, so $\dim(E_+)=n+1$. This shows that $\beta$ is positive definite, thus the distance $\dist_Q$ is well-defined. The associated length distance $\delta$ on $\Sigma$ is then given by:
$$
\dist_Q(\xi,\xi') = 2 \sin\left( \frac{\delta(\xi,\xi')}{2} \right) \qquad \forall \xi, \xi' \in \Sigma,
$$
so that one eventually has:
$$
\delta(H(x),H(y)) = \dist(x,y) \qquad \forall x,y \in X,
$$
i.e.~$H$ is an isometric embedding of $(X,\dist)$ into $\Sigma$ equipped with $\delta$.
Since $$\lim_{t\to 0^+} -4t \log K_t^{(n)}(r)=r^2,$$ we get from Remark \ref{rem:geodesicplus} that
$(X,\dist$) is a geodesic space. Then $H(X)$ is a closed totally geodesic subset of $\Sigma$, meaning that minimizing geodesics joining two points in $H(X)$ are all contained in $H(X)$. We assume that there exists $p \in \Sigma \backslash H(X)$ and set $r:=\delta(p,H(X))$.

\begin{claim}
We have $r< \pi/2$.
\end{claim}

\begin{proof}
Assume $r\ge \pi/2$. Then $H(X)$ is contained in the hemisphere $\{\sigma \in \Sigma : \beta(\sigma,p)\le 0\}$. Set $\lambda(\xi)=\beta(\xi,p)$ for any $\xi \in \setR^{n+1}$. Then $\lambda \circ H:X\to\setR$ is non-positive, and $\lambda \circ H(x)=0$ if and only if $H(x)=0$, which is impossible, so $\lambda \circ H$ is actually negative. But $\lambda \circ H$ is a linear combination of $h_1,\ldots,h_n$ thus it is an element of $\cV$. Since functions in $\cV = E_1$ are $L^2$-orthogonal to constant functions, we reach a contradiction, namely $\int_X \xi \circ H \di \mu = 0$.
\end{proof}

In fact, the same reasoning can be used to prove that $H(X)$ is contained in no hemisphere of $\Sigma$.

We are now in a position to conclude. Since $H(X)$ is closed there exists $q\in H(X)$ such that $\delta(p,q)=r$. The convexity of $H(X)$ implies that any miminizing geodesic of length $<\pi$ starting at $q$ and passing through the open ball $B_r(p)$ cannot meet $H(\Sigma)$. But the union of these minimizing geodesics is an open hemisphere, so $H(X)$ is contained in the complementary hemisphere, hence a contradiction.
\end{proof}


\begin{thebibliography}{GMS13}

\bibitem[AH12]{AtkinsonHan}
      \textsc{K. Atkinson, W. Han}:
      \textit{Spherical Harmonics and Approximations on the Unit Sphere: An Introduction,}
      Lecture Notes in Mathematics \textbf{2044}, Springer, 2012.


\bibitem[AGG19]{AGG19}
      \textsc{T. Adamowicz, M. Gaczkowski, P. G\'orka}:
      \textit{Harmonic functions on metric measure spaces,}
      Revista Matematica Complutense {\bf 32}(1) (2019), 141--186.

\bibitem[ACT18]{AldanaCarronTapie}
      \textsc{C. Aldana, G. Carron, S. Tapie}:
      \textit{$A_\infty$ weights and compactness of conformal metrics under $L^{n/2}$ curvature bounds,}
      ArXiV preprint: 1810.05387 (2018).

\bibitem[ACDM15]{AmbrosioColomboDiMarino}
      \textsc{L. Ambrosio, M. Colombo, S. Di Marino}:
      \textit{Sobolev spaces in metric measure spaces: reflexivity and lower semicontinuity of slope,}
      Adv. Studies in Pure Math. {\bf 67} (2015), 1--58.
           
\bibitem[AFP00]{AmbrosioFuscoPallara}
	\textsc{L. Ambrosio, N. Fusco, D. Pallara}:
	\textit{Functions of bounded variations and free discontinuity problems,} Oxford University Press (2000).
           
           
           
\bibitem[AT03]{AmbrosioTilli}
{\sc L.~Ambrosio, P.~Tilli:} {\it Topics on Analysis in Metric Spaces,}
Oxford Lecture Series in Mathematics and Its Applications, OUP Oxford, 2003.

	

\bibitem[BD59]{BeurlingDeny}
          \textsc{A. Beurling, J. Deny}:
          \textit{Dirichlet spaces,}
          Proc. Natl. Acad. Sci. USA \textbf{45} (1959), 208--215.
          
\bibitem[BBI01]{BuragoBuragoIvanov}
	\textsc{D. Burago, Y. Burago, S. Ivanov}:
	\textit{A course in metric geometry},
	Graduate Studies in Mathematics, 33, American Mathematical Society, Providence, RI, xiv+45 pp (2001).

\bibitem[Ca16]{Carron}
        \textsc{G. Carron}:
        \textit{Geometric inequalities for manifolds with Ricci curvature in the Kato class},
        ArXiV preprint: 1612.03027 (2016).
          
   \bibitem[Ch99]{CheegerRademacher}
 \textsc{J. Cheeger}:
  Differentiability of Lipschitz functions on metric measure spaces,
  \emph{ Geom. Funct. Anal.}
  \textbf{9} (1999), 428–517.

\bibitem[CC97]{CheegerColding}
 \textsc{J. Cheeger, T.H. Colding}:
  On the structure of spaces with Ricci curvature bounded below. I.,
  \emph{J. Differential Geom.}
  \textbf{46} (1997), no. 3, 406--480.

 \bibitem[ChY81]{CheegerYau}
  \textsc{J. Cheeger and S-T. Yau}:
  A lower bound for the heat kernel,\emph{ Comm. Pure Appl. Math.} \textbf{34} (1981), 465--480.


\bibitem[Co07]{Cohen}
{\sc A. M. Cohen:} {\it Numerical Methods for Laplace Transform Inversion}, Numerical Methods and Algorithms, Springer, Boston, MA, 2007.

   \bibitem[Co97]{Colding}
  \textsc{T.H. Colding}:
  Ricci curvature and volume convergence,
\emph{Ann. of Math. (2)} \textbf{145}(3) (1997), 477--501. 

\bibitem[CM97]{ColdingMinicozzi}
         \textsc{T. H. Colding, W. P. Minicozzi II}: \textit{Harmonic functions on Manifolds},
         Ann. of Math. \textbf{146}(3) (1997), 725--747.
         
         

\bibitem[ERS07]{terElstRobinsonSikora}
	\textsc{A. F. M. ter Elst, D. W. Robinson, A. Sikora}:
	\textit{Small time asymptotics of diffusion processes},
	J. Evol. Equ. \textbf{7} (2007), 79--112.

\bibitem[FOT10]{FukushimaOshidaTakeda}
        \textsc{M. Fukushima, Y. Oshima, M. Takeda:}
        \textit{Dirichlet Forms and Symmetric Markov Processes},
         De Gruyter Studeis in Mathematics, 19, 2010.

\bibitem[GG09]{GaczkowskiGorka}
      \textsc{M. Gaczkowski, P. G\'orka}:
      \textit{Harmonic Functions on Metric Measure Spaces: Convergence and Compactness,}
      Potential Anal. {\bf 31} (2009), 203--214.
      
 \bibitem[Gr92]{Grigor'yan92}
	\textsc{A. Grigor'yan}:
	\textit{The heat equation on non-compact Riemannian manifolds},
	 Math. USSR Sb. \textbf{72}(1) (1992), 47--77. Translated from the Russian version in Matem. Sbornik \textbf{182}(1) (1991), 55--87. 
	 
 \bibitem[Gr94]{Grigor'yanEd}
	\textsc{A. Grigor'yan}:
	\textit{Integral maximum principles and its applications},
	 Proc. of Edimburgh Royal Soc. \textbf{124} (1994), 353--362. 
      
\bibitem[Gr10]{Grigor'yan}        
        \textsc{A. Grigor'yan}:
        \textit{Heat kernels on metric measure spaces with regular volume growth}, in "Handbook of Geometric Analysis (Vol. II)" ed. L. Ji, P. Li, R. Schoen, L. Simon, Advanced Lectures in Math. \textbf{13}, International Press, 2010. 1-60. 

\bibitem[Gro07]{Gromov}
      \textsc{M. Gromov}:
      \textit{Metric structures for Riemannian and non-Riemannian spaces,} 
      Modern Birkh\"auser Classics, Birkh\"auser, Boston, MA, english ed., 2007. 

        
\bibitem[H15]{Honda}
        \textsc{S. Honda}:
        \textit{Ricci curvature and $L^p$-convergence,}
        J. Reine Angew Math. \textbf{705}, 85--154 (2015).

\bibitem[H11]{Hua}        
        \textsc{B. Hua}:
        \textit{Harmonic functions of polynomial growth on singular spaces with nonnegative Ricci curvature}, Proc. of the Amer. Math. Soc. \textbf{139}(6) (2011).
        
\bibitem[HKX16]{HuaKellXia}        
        \textsc{B. Hua, M. Kell, C. Xia}:
        \textit{Harmonic functions on metric measure spaces}, ArXiV preprint: 1308.3607v2 (2016).
        
\bibitem[KZ12]{KoskelaZhou}
	\textsc{P. Koskela, Y. Zhou}:
	\textit{Geometry and analysis of Dirichlet forms},
	  Adv. in Math. \textbf{231} (2012), 2755-2801.
	  
	  
\bibitem[LTW97]{LiTamWang}
  \textsc{P. Li, L-F. Tam  and J. Wang}:
  Sharp bounds for the Green's function and the heat kernel, \emph{Math. Res. Lett.} \textbf{4} (1997), no. 4, 589--602.
  
\bibitem[LY86]{LiYau}
\textsc{P. Li, S-T. Yau}:
On the parabolic kernel of the Schr\"odinger operator,
\emph{ Acta Math.} \textbf{156} (1986), no. 3-4, 153--201.




\bibitem[MP49]{MinakshisundaramPleijel}
       \textsc{S. Minakshisundaram, A. Pleijel}:
       \textit{Some properties of the eigenfunctions of the Laplace-operator on Riemannian manifolds},
        Canadian Journ. of Math. \textbf{1} (1949), 242--256. 

\bibitem[RS70]{ReedSimon}
        \textsc{M. Reed, B. Simon:}
        \textit{Methods of modern mathematical physics: Vol.~1, Functional Analysis},
         Academic Press, 1970.

\bibitem[Sa92]{Saloff-Coste}
        \textsc{L. Saloff-Coste}:
        \textit{A note on Poincaré, Sobolev, and Harnack inequalities},
        Intern. Math. Res. Not. \textbf{1992}(2) (1992), 27--38.
        
        
\bibitem[St94]{Sturm1}
	\textsc{K.-T. Sturm}:
	\textit{Analysis on local Dirichlet spaces. I. Recurrence, conservativeness and $L^p$-Liouville properties},
	 J. Reine Angew. Math. \textbf{456} (1994), 173-196.
	 
\bibitem[St95]{Sturm2}
	\textsc{K.-T. Sturm}:
	\textit{Analysis on local Dirichlet spaces. II. Upper Gaussian estimates for the fundamental
solutions of parabolic equations},
	 Osaka J. Math. \textbf{32} (2) (1995) 275-312.
        
\bibitem[St96]{Sturm3}
	\textsc{K.-T. Sturm}:
	\textit{Analysis on local Dirichlet spaces. III. The parabolic Harnack inequality},
	J. Math. Pures Appl. \textbf{75}(3) (1996), 273-297.
	

\end{thebibliography}
\end{document}